\documentclass[11pt]{amsart}

\usepackage{graphicx}        
\usepackage{multicol}        
\usepackage[bottom]{footmisc}
\usepackage{amscd,amsmath,amssymb,amsfonts}
\usepackage[cmtip, all]{xy}

\newtheorem{theorem}{Theorem}[section]
\newtheorem{proposition}[theorem]{Proposition}
\newtheorem{lemma}[theorem]{Lemma}
\newtheorem{corollary}[theorem]{Corollary}

\newtheorem{IntroThm}{Theorem}

\theoremstyle{definition}

\theoremstyle{remark}
\newtheorem{remark}[theorem]{Remark}
\newtheorem{remarks}[theorem]{Remarks}
\newtheorem{ex}[theorem]{Example}
\newtheorem{exs}[theorem]{Examples}

\numberwithin{equation}{section}

\newcommand{\mm}{\mathfrak{1}}

\newcommand{\an}{{\rm an}}

\newcommand{\CH}{{\rm CH}}

\newcommand{\codim}{{\rm codim}}
\newcommand{\rank}{{\rm rank}}
\newcommand{\Pic}{{\rm Pic}}
\newcommand{\End}{{\rm End}}
\newcommand{\Hom}{{\rm Hom}}

\newcommand{\Spec}{{\rm Spec\,}}

\newcommand{\Char}{{\rm char}}
\newcommand{\Tr}{{\rm Tr}}

\newcommand{\0}{\emptyset}
\newcommand{\sA}{{\mathcal A}}

\newcommand{\sC}{{\mathcal C}}

\newcommand{\sF}{{\mathcal F}}

\newcommand{\sH}{{\mathcal H}}
\newcommand{\sI}{{\mathcal I}}

\newcommand{\sK}{{\mathcal K}}

\newcommand{\sM}{{\mathcal M}}

\newcommand{\sO}{{\mathcal O}}

\newcommand{\sU}{{\mathcal U}}
\newcommand{\sV}{{\mathcal V}}
\newcommand{\sW}{{\mathcal W}}
\newcommand{\sX}{{\mathcal X}}
\newcommand{\sY}{{\mathcal Y}}

\newcommand{\A}{{\mathbb A}}

\newcommand{\C}{{\mathbb C}}

\newcommand{\G}{{\mathbb G}}

\renewcommand{\P}{{\mathbb P}}
\newcommand{\Q}{{\mathbb Q}}
\newcommand{\R}{{\mathbb R}}

\newcommand{\Z}{{\mathbb Z}}

\renewcommand{\det}{\operatorname{det}}

\newcommand{\id}{{\operatorname{\rm Id}}}
\newcommand{\Zar}{{\text{\rm Zar}}} 
\newcommand{\Nis}{{\text{\rm Nis}}}

\newcommand{\colim}{{\mathop{\rm colim}}}

\newcommand{\<}{\langle}
\renewcommand{\>}{\rangle}

\renewcommand{\dim}{{\operatorname{\rm dim}}}

\newcommand{\del}{\partial}

\newcommand{\Spc}{{\mathbf{Spc}}}

\newcommand{\Sm}{{\mathbf{Sm}}}

\newcommand{\Proj}{{\operatorname{Proj}}}
\newcommand{\Ab}{{\mathbf{Ab}}}

\newcommand{\Sym}{{\operatorname{Sym}}}

\newcommand{\Gr}{{\operatorname{\rm Gr}}} 
\newcommand{\rnk}{{\operatorname{\text{rank}}}} 

\newcommand{\crit}{\mathfrak{c}}

\newcommand{\GW}{{\operatorname{GW}}} 
\newcommand{\sGW}{{\mathcal{GW}}} 
\newcommand{\SH}{{\operatorname{SH}}} 
\newcommand{\Th}{{\operatorname{Th}}} 
\renewcommand{\th}{{\operatorname{th}}} 
\newcommand{\sHom}{\mathcal{H}om}
\newcommand{\Kos}{{\operatorname{Kos}}}

\newcommand{\Frob}{{\operatorname{Frob}}}

\newcommand{\BGL}{\operatorname{BGL}}
\newcommand{\GL}{\operatorname{GL}}
\newcommand{\SL}{\operatorname{SL}}
\newcommand{\BSL}{\operatorname{BSL}}
\newcommand{\EM}{{\operatorname{EM}}}

\newcommand{\can}{\text{can}}

\newcommand{\ev}{\text{\it ev}}

\newcommand{\KO}{{\operatorname{KO}}}
\newcommand{\CW}{{\operatorname{CW}}}
\newcommand{\KGL}{{\operatorname{KGL}}}

\newcommand{\ind}[1]{}

\newcommand{\inp}[1]{}

\newcommand{\+}{\widehat{+}}

\newcommand{\mS}{\mathbb{S}}
\newcommand{\HM}{{\sH\sM}}
\newcommand{\gm}{{{\operatorname{gm}}} }

\begin{document}
\setcounter{tocdepth}{1}

\title[Enumerative geometry with quadratic forms]{Aspects of enumerative geometry with quadratic forms}

\author{Marc Levine}

\begin{abstract} Using the motivic stable homotopy category over a field $k$, a smooth   variety $X$ over $k$ has an Euler characteristic $\chi(X/k)$ in the Grothendieck-Witt ring $\GW(k)$. The rank of $\chi(X/k)$ is the classical $\Z$-valued Euler characteristic, defined using singular cohomology or \'etale cohomology, and the signature of $\chi(X/k)$ under a real embedding $\sigma:k\to \R$ gives the topological Euler characteristic of the real points $X^\sigma(\R)$.

We develop tools to compute $\chi(X/k)$, assuming  $k$ has characteristic $\neq2$  and apply these to refine some classical formulas in enumerative geometry, such as formulas for the top Chern class of the dual, symmetric powers and tensor products of bundles, to identities for the Euler classes in Chow-Witt groups. We also refine  the classical Riemann-Hurwitz formula to an identity in $\GW(k)$ and compute  $\chi(X/k)$ for all  hypersurfaces in $\P^{n+1}_k$ defined by a polynomial of the form $\sum_{i=0}^{n+1}a_iX_i^m$; this latter includes the case of an arbitrary quadric hypersurface. 

This paper is a revision of ``Toward an enumerative geometry with quadratic forms'' \cite{LevQuad}.
\end{abstract}
\maketitle

\tableofcontents

\section*{Introduction}\label{sec:Intro}
We work throughout in the category of smooth quasi-projective schemes over a field $k$, 
$\Sm_k$, with $\Char(k)\neq 2$. The main goal of this paper is to take steps toward constructing a good theory of enumerative geometry with values in quadratic forms, refining the classical $\Z$-valued enumerative geometry. The foundations of this theory have been laid by work of Barge-Morel \cite{BargeMorel}, Fasel \cite{FaselCW}, Fasel-Srinivas \cite{FaselSri} and  Morel \cite{MorelICTP, MorelA1} (and many others), and first steps in this direction have been taken by Hoyois \cite{HoyoisGL}, Kass-Wickelgren \cite{KassWickelgren, KassWickelgrenLines} and Pauli \cite{Pauli}. 

The main tool is the replacement of the Chow groups $\CH^n(X)$ of a smooth variety $X$, viewed via Bloch's formula as the cohomology of the Milnor $K$-sheaves
\[
\CH^n(X)\cong H^n(X, \sK^M_n),
\]
with the {\em  Chow-Witt groups} of Barge-Morel \cite{BargeMorel, FaselCW}
\[
\widetilde{\CH}^n(X;L):=H^n(X, \sK^{MW}_n(L)).
\]
Here $\sK^{MW}_n(L)$ is the $n$th Milnor-Witt sheaf, as defined by Hopkins-Morel \cite{MorelICTP, MorelA1}, twisted by a line bundle $L$ on $X$. This theory has many of the formal properties of the Chow ring, with the subtlety that one needs suitable twists to allow for the pushforward maps: 
 for a proper morphism $f:Y\to X$  of relative dimension $d$, one has 
\[
f_*: H^a(Y, \sK^{MW}_b(f^*L\otimes\omega_{Y/k}))\to H^{a-d}(X, \sK^{MW}_{b-d}(L\otimes\omega_{X/k})),
\]
where $\omega_{X/k}:=\det\Omega_{X/k}$, $\omega_{Y/k}:=\det\Omega_{Y/k}$ are the respective dualizing sheaves.

The second important difference is that, although a rank $r$ vector bundle $V\to X$ has an Euler class \cite[\S 2.1]{BargeMorel}
\[
e^\CW(V)\in H^r(X, \sK^{MW}_r(\det^{-1}V)),
\]
the group this class lives in depends on $V$ (or at least $\det V$). Under the map
\[
\sK^{MW}_*(\det^{-1}V)\to \sK_*^M
\]
$e^\CW(V)$ maps to the top Chern class $c_r(V)$ and  $e^\CW(\det V)$ maps to $c_1(V)$, but   there is no projective bundle formula for the oriented Chow groups, and thus no obvious ``intermediate''   classes lifting the other Chern classes of $V$ to the oriented setting. There are versions of the classical Pontryagin classes,  but we will not study these in detail here. 

There is still enough here to define an Euler characteristic of a smooth  projective $k$-scheme $p:X\to\Spec k$ as
\[
\chi^{\CW}(X/k):=p_*(e^\CW(T_X))\in K^{MW}_0(k),
\]
where, if $X$ has dimension $d$ over $k$,  $e^\CW(T_X)\in H^d(X, \sK^{MW}_d(\omega_{X/k}))$ is the Euler class of the tangent bundle $T_X$. Morel  \cite[Lemma 6.3.8]{MorelICTP} identifies $K^{MW}_0(k)$ with the Grothendieck-Witt group of non-degenerate quadratic forms over $k$, $\GW(k)$, so we have the Euler characteristic $\chi^\CW(X/k)\in \GW(k)$. The fact that the Euler class $e^\CW(T_X)$  maps to  $c_d(T_X)\in H^d(X, \sK^M_d)=\CH_0(X)$ under the map of sheaves $\sK^{MW}_*(\omega_{X/k})\to \sK^M_d$ shows that the image $\chi^{\CW}(X/k)$ under the rank homomorphism $\GW(k)\to \Z$ is the classical Euler characteristic of $X$,  which agrees with the topological Euler characteristic  of $X(\C)$ defined using singular cohomology if $k\subset \C$, or the $\ell$-adic Euler characteristic of $X_{\bar{k}}$, defined using \'etale cohomology. 

One can also define a categorical Euler characteristic $\chi(X/k)$, by using the infinite suspension spectrum $\Sigma_T^\infty X_+\in \SH(k)$, where $\SH(k)$ is the motivic stable homotopy category over $k$. Hoyois \cite[Theorem 5.22]{Hoyois6}, Hu \cite[Appendix A]{Hu}, Riou \cite{Riou}  and Voevodsky  \cite[\S2]{Voev} have shown that this suspension spectrum is always a strongly dualizable object in $\SH(k)$, so it gives rise in a standard way  to an endomorphism of the unit object $\mS_k$:
\[
\chi(X/k)\in \End_{\SH(k)}(\mS_k).
\]
By Morel's theorem  \cite[Theorem 6.4.1, Remark 6.4.2]{MorelICTP} there is a canonical isomorphism $\End_{\SH(k)}(\mS_k)\cong \GW(k)$, so we have a second Euler characteristic in $\GW(k)$\footnote{Morel's theorem is for $k$ perfect. In positive characteristic, one needs to invert the characteristic if $k$ is not perfect, but this is mostly harmless, see Remark~\ref{rems:NonperfectFields}}. 

We should mention that for $k\subset \R\subset \C$, the image of the categorical Euler characteristic $\chi(X/\R)$ in $\GW(\R)$ has the property that its signature gives the Euler characteristic of $X(\R)$, while  the rank gives the Euler characteristic of $X(\C)$.  

In our paper  with A. Raksit \cite{LREulerChar} we showed that these two Euler characteristics agree.

\begin{IntroThm}[\hbox{\cite[Theorem 8.4]{LREulerChar}}]\label{thm:comparison} Let $X$ be a smooth projective variety of pure dimension over $k$. Then 
\[
\chi^{\CW}(X/k)=\chi(X/k)
\]
in $\GW(k)$.
\end{IntroThm}
One  consequence of this comparison result is the fact that the Euler characteristic of an odd dimensional smooth projective variety is always hyperbolic  (Corollary~\ref{cor:OddEulerChar}); one can view this a a generalization of the fact that the topological Euler characteristic of a real oriented manifold of dimension $4m+2$ is always even. This has already been proven in our paper \cite{LREulerChar} with Raksit using hermitian $K$-theory, but we include this somewhat different proof relying on the Chow-Witt Euler class here. 

We then turn to developing some computational techniques. Here the main goal is to compare $e^\CW(V)$ and $e^\CW(V\otimes L)$ for a line bundle $L$, without having the ``lower Chern classes'' of $V$ on hand. We also prove a useful formula relating Euler class of a vector bundle $V$ with that of its dual (Theorem~\ref{thm:Dual}): 
\[
e^\CW(V)=(-1)^{\text{rank}V}e^\CW(V^\vee).
\]
and compute the Euler class of symmetric powers and tensor products of rank two bundles. 

Kass-Wickelgren \cite{KassWickelgren} have  constructed an ``Euler number'' in $\GW(k)$ for a relatively oriented algebraic vector bundle with enough sections on a smooth projective $k$-scheme; as one application,  they use this in \cite{KassWickelgrenLines}  to lift the count of lines on a smooth cubic surface  over $k$ to an equality in $\GW(k)$. For a pencil $f: X \to \P^1$ of curves on a smooth projective surface $X$ over $k$, they lift the classical computation of the Euler number of $T^*_X \otimes f^*T_{\P^1}$ in terms of the singularities of the fibers of $f$ to an equality in $\GW(k)$. This approach to Euler numbers has been studied further by Bachmann-Wickelgren \cite{BachmannWickelgren}.

We approach the question of lifting such classical degeneration formulas to $\GW(k)$ from a somewhat different point of view. We apply Theorem~\ref{thm:comparison} and the results obtained in \S\ref{sec:Local}-\ref{sec:Twisting} to give a generalization of the classical degeneration formulas for counting singularities in a morphism $f:X\to C$, with $X$ a smooth projective variety and $C$ a smooth projective curve (admitting for technical reasons a half-canonical line bundle in case $X$ has odd dimension); for $X$ a curve, this a refinement of the classical Riemann-Hurwitz formula. Our generalization gives an identity in $\GW(k)$; applying the rank homomorphism recovers the classical numerical formulas. In the case of even dimensional varieties, we apply the degeneration formula to compute the Euler characteristic of generalized Fermat hypersurfaces, that is, a hypersurface  $X\subset \P^{n+1}_k$ defined by a polynomial of the form $\sum_{i=0}^{n+1}a_iX_i^m$, see Theorem~\ref{thm:EulerCharDiagonalHyp}. As a special case, we find an explicit formula for the Euler characteristic of a quadric hypersurface, Corollary~\ref{cor:EvenQuadric EulerChar}. The question of computing the Euler characteristic of a quadric hypersurface was raised by Kass and Wickelgren\footnote{private communication}.

This current version is a substantial revision of the original \cite{LevQuad}, helped along by many developments in this area since then. Much of the first version was concerned with showing that the pushforward maps for the Chow-Witt groups, as defined by Fasel \cite{FaselCW}, agree with those using the structure of $H^*(-, \sK^{MW}_*)$ as an $\SL$-oriented theory, and using this to prove Theorem~\ref{thm:comparison}. Both of these results have been subsumed in our paper with Raksit \cite{LREulerChar}. The original proof of Theorem~\ref{thm:Dual} followed Asok-Fasel \cite{AsokFaselEuler} in viewing $e^\CW(V)$ as an obstruction class; recent work of Wendt  \cite{Wendt}, building on the paper of Hornbostel-Wendt \cite{HornbostelWendt}, allows a quicker path to this result and also gives a nice extension to the Chow-Witt groups of products of classifying spaces $\BGL_n$ and $\BSL_n$ of the fact that the map $(\rnk, \pi):\GW(k)\to \Z\times W(k)$ is injective, where $W(k)$ is the Witt ring and $\pi:\GW(k)\to W(k)$ is the canonical surjection. This shows that one can detect universal identities among  Chow-Witt-valued Euler classes by passing to the corresponding top Chern classes and the Euler classes with values in the the cohomology of the Witt sheaves (see Proposition~\ref{prop:Reduction} and Theorem~\ref{thm:GL2SplittingPrinc}). This useful fact also allows us to give a much simpler proof of our result comparing the Euler characteristics of a vector bundle $V$ and the $L$-twisted bundle $V\otimes L$, for $L$ a line bundle.  Using the Witt sheaves also enables us to improve our main result, the quadratic Riemann-Hurwitz formula (Corollary~\ref{cor:Fibering1}), removing the hypothesis of the existence of a theta-characteristic on the target curve in case the source variety has even dimension. We have also added a section discussing the work of Kass-Wickelgren and Bachmann-Wickelgren, which gives a description of the local indices for a section of a vector bundle with isolated zeros in terms of an associated
 Scheja-Storch quadratic form; in our previous version, we had restricted this explicit representation to the case of ``diagonalizable'' sections. 

I am grateful to Aravind Asok for a number of very helpful suggestions, as well as corrections to an earlier version of this manuscript. I also wish to thank Kirsten Wickelgren for raising a number of questions on the results in the earlier version, for example, asking if the Riemann-Hurwitz formula would hold in the even-dimensional case without assuming the existence of a theta-divisor on the target curve. Finally, thanks are due to Matthias Wendt  for very helpful discussions about his paper \cite{Wendt} and to the referee, whose comments  greatly improved the organization and presentation of this paper. \footnote{This paper is part of a project that  has received funding from the European Research Council (ERC) under the European Union's Horizon 2020 research and innovation programme (grant agreement No. 832833)} 

\section{The categorical Euler characteristic}\label{sec:CatEulerChar}

In this section we review and collect a number of facts and constructions concerning Euler characteristics in a symmetric monoidal category, specializing quickly to the motivic stable homotopy category $\SH(k)$ over a  field $k$. Most of the results here are not new, but we include this introductory section to give an overview of some of the basic properties of the Grothendieck-Witt-valued Euler characteristic. Beside the motivic stable homotopy category $\SH(k)$, we will be using the unstable motivic homotopy category $\sH(k)$, the category of spaces over $k$, $\Spc(k)$, and the pointed versions $\sH_\bullet(k)$ and $\Spc_\bullet(k)$, as well as the classical stable homotopy category $\SH$. We use \cite{Hoyois6, JardineMotSym, MorelVoevodsky} as references for these constructions and their basic properties. 

\subsection{Properties of the categorical Euler characteristic}\label{subsec:CatEuler} Let   $(\sC,\otimes, \mm,\tau)$ be a symmetric monoidal category.  Following \cite{DoldPuppe}, we have the notion of the {\em  dual} $(X^\vee, \delta_X:\mm\to X\otimes X^\vee, \ev_X:X^\vee\otimes X\to \mm)$ of an object $X$ of $\sC$, where one requires that the maps $ \delta_X$ and $\ev_X$ satisfy the following conditions: the compositions
\begin{equation}\label{eqn:Duality1}
X\cong \mm\otimes X \xrightarrow{\delta_X\otimes\id}  X\otimes X^\vee\otimes X
\xrightarrow{\id\otimes\ev_X} X\otimes\mm\cong X
\end{equation}
and
\begin{equation}\label{eqn:Duality2}
X^\vee\cong X^\vee\otimes\mm\xrightarrow{\id\otimes\delta_X}X^\vee\otimes X\otimes X^\vee\xrightarrow{\ev_X\otimes\id} \mm\otimes X^\vee\cong X^\vee
\end{equation}
are the respective identity maps. It follows from an easy computation that if $(X^\vee, \delta_X, \ev_X)$ is a dual of $X$, then
$(X, \tau_{X,X^\vee}\circ\delta_X, \ev_X\circ\tau_{X^\vee, X})$ is a dual of $X^\vee$. 

It follows from \cite[Theorem 1.3]{DoldPuppe} that notion of dual described above is equivalent to that of ``strong dual'' given in  \cite[\S 1]{DoldPuppe}. It also  follows from \cite[Theorem 1.3]{DoldPuppe} that  the  dual $(X^\vee, \delta_X, \ev_X)$ of an object $X$, if it exists, is unique up to unique isomorphism.  An object admitting a dual is called {\em strongly dualizable}. 

Recall that an object $X\in \sC$ is {\em invertible} if there exists an object $Y$ and an isomorphism $\alpha:\mm\to Y\otimes X$; clearly this determines the isomorphism class of $Y$. We call $Y$ the {\em inverse} to $X$ and write  $Y=X^{-1}$. By \cite[Proposition 4.11]{Dugger} an invertible $X$ is strongly dualizable with dual $X^{-1}$ and with $\delta_{X^{-1}}=\alpha$ (but note, $\ev_{X^{-1}}$ is not necessarily $\alpha^{-1}\circ\tau_{X^{-1},X}$).
 
 For $X$ strongly dualizable, we have the {\em categorical Euler characteristic} $\chi_\sC(X)\in \End_\sC(\mm)$ defined as the composition
\[
\mm\xrightarrow{\delta_X}X\otimes X^\vee\xrightarrow{\tau_{X,X^\vee}}X^\vee\otimes X\xrightarrow {\ev_X}\mm.
\]
Clearly the collection of strongly dualizable objects in $\sC$ is closed under $\otimes$
and 
\[ \chi_\sC(X\otimes Y)=\chi_\sC(X)\cdot \chi_\sC(Y),\
\chi_\sC(X^\vee)=\chi_\sC(X),
\]
for strongly dualizable objects $X$ and $Y$.

Let  $\sC$ be a triangulated tensor category. May \cite[Theorem 0.1]{May} has given conditions under which   the collection of strongly dualizable objects in $\sC$ forms the objects in a thick subcategory of $\sC$, and the Euler characteristic is additive in distinguished triangles: if $A\to B\to C\to A[1]$ is a distinguished triangle of strongly dualizable objects, then $\chi_\sC(B)=\chi_\sC(A)+\chi_\sC(C)$, and $\chi_\sC(A[1])=-\chi_\sC(A)$. In particular, for $X\in \sC$ a strongly  dualizable object, each translation $X[p]$ of $X$ is a strongly dualizable object. The May axioms are not satisfied for an arbitrary  triangulated tensor category, but as noted in {\it loc. cit.}, they are satisfied for the classical stable homotopy category $\SH$ and for $\SH(k)$. In addition to assuming that $\sC$ is a {\em closed} symmetric monoidal category (i.e., there are internal Homs), May requires various compatibilities of the monoidal product with the triangulated structure. See \cite[\S 4]{May} for details.

The respective sphere spectra $\mS\in \SH$, $\mS_k\in \SH(k)$ are the units in the symmetric monoidal categories $\SH$, $\SH(k)$.

\begin{remarks}\label{rems:NonperfectFields} 1. As we have mentioned in the introduction, for $k$ a perfect field, Morel's theorem \cite[Theorem 6.4.1, Remark 6.4.2]{MorelICTP}  gives a natural isomorphism $\GW(k)\to \End_{\SH(k)}(\mS_k)$; we will usually view the categorical Euler characteristic $\chi_{\SH(k)}(-)$ as being valued in $\GW(k)$. For $F$ a field and $u\in F^\times$ a unit, we let $\<u\>\in \GW(F)$ denote the rank one quadratic form $x\mapsto ux^2$; for  a positive integer $n$, we set $n_\epsilon:=\sum_{i=0}^{n-1}\<-1\>^i$.
\\[2pt]
2. We have also mentioned in the introduction that for $X$ smooth and projective over $k$, the suspension spectrum $\Sigma_T^\infty X_+$ is a strongly dualizable object in $\SH(k)$ (see 
\cite[Theorem 5.22]{Hoyois6},\cite[Appendix A]{Hu},  \cite{Riou}  and   \cite[\S2]{Voev}). If $k$ admits resolution of singularities and $U$ is in $\Sm_k$, then taking a smooth projective completion with complement a normal crossing divisor, and using suitable localization distinguished triangles,  May's results mentioned above imply that $\Sigma^\infty_TU_+$ is strongly dualizable in $\SH(k)$.

More generally,  Riou \cite[Theorem B.2]{EllipticFlop} has shown that for $k$ a perfect field of characteristic $p>0$,  and $U\in \Sm_k$, $\Sigma^\infty_TU_+$ is strongly dualizable in $\SH(k)[1/p]$, so has a well-defined Euler characteristic in  $\GW(k)[1/p]$. Assuming as we do $p$ to be odd, since each element of the kernel $I(k)$ of the rank homomorphism $\GW(k)\to \Z$ has finite order a power of 2, the map $\GW(k)\to \GW(k)[1/p]$ is injective, and defines an isomorphism from $I(k)$ to the kernel of $ \GW(k)[1/p]\to \Z[1/p]$. Moreover, since the \'etale 
Euler characteristic of $U$ is $\Z$-valued, the categorical Euler characteristic $\chi_{\SH(k)[1/p]}(\Sigma^\infty_TU_+)$ lands in $\GW(k)\subset \GW(k)[1/p]$. Thus, even in positive characteristic, we may treat each $U\in \Sm_k$ as dualizable by passing to $\SH(k)[1/p]$, and we still get a a categorical Euler characteristic valued in $\GW(k)$.  

Similarly, for $k$ an arbitrary field of positive characteristic $p\neq2$, we may pass to the perfect closure $k^{perf}$. For each $n$, the base-extension map $\GW(k)\to \GW(k^{1/p^n})$ is an isomorphism, with inverse induced by the Frobenius map $\Frob^n:k^{1/p^n}\to k$, so $\GW(k)\cong \GW(k^{perf})$.  We may therefore  work in $\SH(k^{perf})[1/p]$ and still have a $\GW(k)$-valued Euler characteristic. We will silently pass to $\SH(k)[1/p]$ or $\SH(k^{perf})[1/p]$ as needed in the remainder of the paper, and refer to a space $\sX\in \Spc(k)$ as dualizable if $\Sigma^\infty_T\sX_+\in \SH(k^{perf})[1/p]$ is strongly dualizable. 

For  a dualizable space $\sX\in \Spc(k)$, we write $\chi(\sX/k)$ for $\chi_{\SH(k)}(\Sigma_T^\infty\sX_+)$ or for $\chi_{\SH(k^{perf})[1/p]}(\Sigma_T^\infty\sX_+)$ if we need to invert the characteristic and pass to $k^{perf}$.  For a dualizable space
$\sX\in \Spc_\bullet(k)$ we similarly write $\chi(\sX/k)$ for $\chi_{\SH(k)}(\Sigma_T^\infty\sX)$ or $\chi_{\SH(k^{perf})[1/p]}(\Sigma_T^\infty\sX)$.
\end{remarks}

In $\SH(k)$ we have for $a, b\in \Z$ the suspension operators $\Sigma^{a,b}:\SH(k)\to \SH(k)$, $\Sigma^{a,b} :=\Sigma^{a-b}_{S^1}\circ \Sigma^b_{\G_m}$,  which are commuting autoequivalences with  $\Sigma^{a,b}\circ \Sigma^{a',b'}= \Sigma^{a+a', b+b'}$ and $\Sigma^{0,0}=\id$. Moreover, we have the canonical isomorphisms $\Sigma^{a,b}\sX\wedge \Sigma^{a',b'}\sY\cong\Sigma^{a+a', b+b'}(\sX\wedge \sY)$. This gives us the invertible objects $\mS_k^{a,b}:=\Sigma^{a,b}\mS_k$ of $\SH(k)$ with inverse $\mS_k^{-a,-b}$, which are thus strongly dualizable. 
 
For all $p\ge q\ge0$, we have  the sphere $S_k^{p,q}:=S^{p-q}\wedge \G_m^{\wedge q}\in \Spc_\bullet(k)$, and a canonical isomorphism $\mS_k^{p,q}\cong \Sigma^\infty_T S^{p,q}_k$ in $\SH(k)$. Thus $S_k^{p,q}$  is strongly dualizable for  $p\ge q\ge0$. For example $(\A^n_k\setminus\{0^n\}, \{1^n\})\cong S^{2n-1,n}_k$ is strongly dualizable.

\begin{lemma}\label{lem:EulerSphere} $\chi_{\SH(k)}(\mS_k^{p,q})=(-1)^p\cdot\<-1\>^q$. 
\end{lemma}

\begin{proof}  We have $(\mS_k^{p,q})^\vee=\mS_k^{-p,-q}$, so $\chi_{\SH(k)}(\mS_k^{p,q})=
\chi_{\SH(k)}(\mS_k^{-p,-q})$, reducing us to the case $q\ge0$. Since 
$\Sigma^{p,q}\mS_k\cong (\Sigma^{2q,q}\mS_k)[p-2q]$, we have
\[
\chi_{\SH(k)}(\mS_k^{p,q})=(-1)^p\chi_{\SH(k)}(\mS_k^{2q,q})=
(-1)^p(\chi_{\SH(k)}(\mS_k^{2,1}))^q.
\]
This reduces us to showing that $\chi_{\SH(k)}(\mS_k^{2,1})=\<-1\>$. Since $\mS_k^{2,1}\cong \Sigma^\infty_T(\P^1,\infty)$ we have $\Sigma^\infty_T\P^1_+\cong \mS_k^{2,1}\oplus \mS_k$. Since  $\mS_k$ is the unit in $\SH(k)$, we have $\chi_{\SH(k)}(\mS_k)=\<1\>$, so we need to show that $\chi(\P^1/k)=\<1\>+\<-1\>$.

This is proven by Hoyois \cite[Example 1.7]{HoyoisGL} and also follows from our result with Raksit \cite[Corollary 8.7]{LREulerChar}.
\end{proof}

\begin{remarks}\label{rem:CatTopEuler} 1. The categorical Euler characteristic in an arbitrary symmetric monoidal category is clearly natural with respect to symmetric mo\-no\-idal functors. In particular, if $k=\C$, the image of $\chi(\sX/k)$ for a dualizable space $\sX\in \Spc(\C))$ under the Betti realization functor $Re_B:\SH(\C)\to \SH$ is the Euler characteristic of $Re_B(\sX)$ computed in $\SH$. As the map $\pi_0(\mS)=\End_\SH(\mS)\to \End_{D(\Ab)}(\Z)=\Z$ under the $\Z$-linearization map is an isomorphism, the Euler characteristic in $\SH$ of $\Sigma^\infty T_+$, for a finite CW complex $T$, is just the topological Euler characteristic of $T$. Since $\GW(\C)=\Z$ by rank, we see that, for $k\subset \C$, and for $X\in \Sm_k$, $\rnk\chi(X/k)$ is the topological Euler characteristic of the complex manifold $X(\C)^{an}$ of $\C$-points of $X$.

We have as well the $\R$-Betti realization functor $Re_{B\R}:\SH(\R)\to \SH$, which for $X\in \Sm_\R$ sends the suspension spectrum $\Sigma_T^\infty X_+$ to the suspension spectrum of the real points of $X$, $\Sigma^\infty X(\R)^\an$. We note that the induced map $\GW(\R)\to \End_\SH(\mS)=\Z$ is the signature homomorphism. Indeed, we need only check that $\<-1\>$ goes to $-1$. To see this, the map $\GW(k)\to \End_{\SH(k)}(\mS_k)$ is constructed by sending the one-dimensional form $\<u\>$ to the automorphism  $m_u$ of $\P^1$ given by  $[x_0:x_1]\mapsto [x_0:ux_1]$. On   $\P^1(\R)^\an=S^1$, $m_{-1}$ is the map $\theta\mapsto -\theta$ and hence has degree $-1$.\footnote{I am grateful to Fabien Morel for this argument.} Concretely, for $X\in \Sm_\R$, the rank of $\chi(X/\R)$ is the Euler characteristic of $X(\C)^\an$ and the signature of $\chi(X/\R)$ is the Euler characteristic of $X(\R)^\an$.\\\\
2. For $q\in \GW(\R)$ with signature  $\text{sgn}(q)$, one has $\text{rank}(q)\equiv \text{sgn}(q)\mod 2$. This implies that for $X\in \Sm_\R$, the Euler characteristic of $X(\C)^\an$ and $X(\R)^\an$ are congruent modulo 2. At least for proper $\R$-schemes, this is an easy consequence of the fact (see for example \cite[pg. 76]{Milnor}) that for a compact Riemannian manifold $M$ with an isometry  $f:M\to M$, the fixed point locus $M^f$ has Euler characteristic given by the Lefschetz number 
\[
\chi^{top}(M^f)=\sum_i(-1)^i\Tr(f^*_{|H^i(M,\Q)}).
\]
One applies this to complex conjugation $\mathfrak{c}:X(\C)^\an\to X(\C)^\an$, after decomposing $H^i(X(\C)^\an,\Q)$ into plus and minus eigenspaces for the action of $\mathfrak{c}$, to give the congruence. Probably this argument can be extended without much trouble to the case of open smooth varieties. 

There is also an upper bound for $\chi^{top}(X(\R))$ in terms of the Hodge theory of $X$, due to Abelson \cite{Abelson}, namely, if $X/\R$ is smooth and projective and has even dimension $2n$ over $\R$, then
\[
|\chi^{top}(X(\R)^\an)|\le \dim_\C H^{n,n}(X_\C).
\]
The proof uses the Hodge decomposition, the hard Lefschetz theorem and the Lefschetz fixed point theorem as above.

On the other hand, as mentioned in \cite{LREulerChar}, this last inequality also follows from our theorem with Raksit \cite[Theorem 1.3]{LREulerChar}. In fact, for $X$ smooth and projective of even dimension $2n$ over $k$,  this result shows that $\chi(X/k)$ is of the form $Q+m\cdot H$, where $H$ is the hyperbolic form $H(x,y)=x^2-y^2$, $m$ is an integer and $Q$ is the quadratic form associated to the symmetric bilinear form
\[
H^n(X, \Omega^n_{X/k})\times  H^n(X, \Omega^n_{X/k})\xrightarrow{\cup}
H^{2n}(X, \Omega^{2n}_{X/k})\xrightarrow{\Tr}k,
\]
where $\cup$ is cup product and $\Tr$ is the canonical trace map corresponding to $1\in H^0(X,\sO_X)$ by Serre duality. 
 This shows that for $k\subset \R$, 
\begin{align*}
|\chi^{top}(X(\R)^\an)|&=|\text{sig}(\chi(X/k))|\\
&=|\text{sig}(Q)|\\
&\le \dim_kH^n(X, \Omega^n_{X/k}),
\end{align*}
which recovers Abelson's inequality.
 \end{remarks}

Here are some additional elementary but useful properties of the Euler characteristic $\chi(-/k)$.

\begin{proposition}\label{prop:EulerCharProperties} 1. Let $F$, $X$ and $Y$ be in $\Sm_k$ and let $p:Y\to X$ be a Zariski locally trivial fiber bundle with fiber $F$.  Then
\[
\chi(Y/k)=\chi(X/k)\cdot\chi(F/k).
\]
2. Let $X$ be in $\Sm_k$ and let $p:V\to X$ be a rank $r$ vector bundle. Then the Thom space $\Th(V)$ is dualizable and
\[
\chi(\Th(V)/k)=\<-1\>^r\chi(X/k).
\]
3. Let $X$ be in $\Sm_k$, let $j:U\to X$ be an open subscheme with closed complement $i:Z:=X\setminus U\to X$. Suppose that $Z$ is smooth over $k$ and of pure codimension $c$ in $X$. Then
\[
\chi(X/k)=\chi(U/k)+\<-1\>^c\chi(Z/k).
\]
4. Let $X$ be in $\Sm_k$ and let $p:V\to X$ be a rank $r$ vector bundle. Let $q:\P(V)\to X$ be the associated projective space bundle $\Proj_X(\Sym^*V^\vee)$. Then  
\[
\chi(\P(V)/k)=r_\epsilon\cdot\chi(X/k).
\]
5. Let $i:Z\to X$ be a codimension $c$ closed immersion in $\Sm_k$.   Let $\tilde X$ be the blow up of $X$ along $Z$. Then 
\[
\chi(\tilde X/k)=\chi(X/k)+\<-1\>\cdot (c-1)_\epsilon\cdot \chi(Z/k).
\]
6. Let $\sigma:k\to F$ be an extension of fields, inducing the homomorphism $\sigma_*:\GW(k)\to \GW(F)$. Then for   $X\in \Sm_k$,   
\[
\chi(X_F/F)=\sigma_*(\chi(X/k)).
\]
\end{proposition}

\begin{proof}  (1) Take a finite Zariski open cover $\sU=\{U_i\}$ of $X$ that trivializes the bundle $Y\to X$. Since
\[
\chi((U_{i_0}\cap\ldots\cap U_{i_n})\times F/k)=
\chi(U_{i_0}\cap\ldots\cap U_{i_n}/k)\cdot \chi(F/k)
\]
the additivity of $\chi$ in distinguished triangles together with the Mayer-Vietoris triangles for $\sU$ and for $\sV:=\{U_i\times F\}$ shows that 
\[
\chi(Y/k)=\chi(X/k)\cdot\chi(F/k).
\]

For (2), the distinguished triangle
\[
\Sigma^\infty_T (V\setminus 0_X)_+\to \Sigma^\infty_T V_+\to
\Sigma^\infty_T \Th(V)\to
\]
shows that $\Th(V)$ is dualizable and gives
\[
\chi(\Th(V)/k)=\chi(X/k)-\chi(V\setminus 0_X/k)
\]
Since $p:V\to X$ is Zariski locally trivial, so is $V\setminus 0_X\to X$, so 
\[
\chi(V\setminus 0_X/k)=
\chi(\A^r\times X\setminus 0_X/k)=\chi(\A^r\setminus 0/k)\cdot \chi(X/k).
\]
Since $\Sigma^\infty_T\A^r\setminus 0\cong \mS^{2r-1,r}_k\oplus \mS_k$, we have
\[
\chi(\A^r\setminus 0/k)=\<1\>-\<-1\>^r
\]
by Lemma~\ref{lem:EulerSphere}.  Thus 
\[
\chi(\Th(V)/k)=\chi(X/k)-(\<1\>-\<-1\>^r)\cdot \chi(X/k)=\<-1\>^r\cdot \chi(X/k)
\]

For (3),  we have the cofiber sequence $U_+\to X_+\to X/U$. The Morel-Voevodsky homotopy purity theorem \cite[Theorem 3.2.23]{MorelVoevodsky} gives the isomorphism $X/U\cong \Th(N_{Z/X})$ in the unstable pointed motivic homotopy category $\sH_\bullet(k)$, where $N_{Z/X}$ is the normal bundle of $Z$ in $X$. This gives us the distinguished triangle in $\SH(k)$
\[
\Sigma^\infty_TU_+\to \Sigma^\infty_TX_+\to \Sigma^\infty_T\Th(N_{Z/ X})\to 
\]
hence by (2), $\chi(X/k)=\chi(U/k)+\chi(\Th(N_{Z\subset X})/k)=\chi(U/k)+\<-1\>^c\chi(Z/k)$.

For (4),  (1) reduces us to the computation of $\chi(\P^{r-1}/k)$. Letting $U=\A^{r-1}_k\subset \P^{r-1}_k$ with complement $Z\cong \P^{r-2}_k$, (3) gives the identity $\chi(\P^{r-1}/k)=\<1\>+\<-1\>\cdot \chi(\P^{r-2}_k)$, so (4) follows by induction  in $r$.

For (5), let $p:N_Z\to Z$ be the normal bundle of $i$ and let $E\subset \tilde{X}$ be the exceptional divisor, so $E=\P(N_Z)$. Let   $U=X\setminus Z=\tilde{X}\setminus E$. By (2), (3) and (4), we have
\begin{multline*}
\chi(X/k)-\<-1\>^c\chi(Z/k)=\chi(\tilde{X}/k)-\<-1\>\cdot\chi(\P(N_Z)/k)\\=
\chi(\tilde{X}/k)-\<-1\>\cdot c_\epsilon\chi(Z/k)
\end{multline*}
which proves (5). 

For (6), let $\pi:\Spec F\to \Spec k$ be the morphism induced by $\sigma$. Then we have the exact symmetric monoidal  functor $\pi^*:\SH(k)\to \SH(F)$, with $\pi^*\Sigma^\infty_TX_+=\Sigma^\infty_TX_{F+}$. Moreover the map  $\pi^*:\End_{\SH(k)}(\mS_k)\to \End_{\SH(F)}(\mS_F)$ is equal to the map $\sigma_*:\GW(k)\to \GW(F)$ via Morel's identification $\End_{\SH(k)}(\mS_k)\cong \GW(k)$, $\End_{\SH(F)}(\mS_F)\cong \GW(F)$; this is clear from the definition of Morel's map on the generators of $\GW(k)$, sending the rank one form $\<u\>$, $u\in k^\times$, to the endomorphism of $\mS_k$ induced by the endomorphism of $\P^1_k$ sending $(x:y)$ to $(x:uy)$. Since $\pi^*$ is compatible with duality, these facts prove (6).
\end{proof}

One has a simple expression for the Euler characteristic of a smooth cellular scheme. Recall that a reduced finite type $k$-scheme $X$ is {\em cellular} if $X$ admits a filtration
\[
\0=X_{-1}\subset X_0\subset\ldots\subset X_n=X
\]
with $X_i\setminus X_{i-1}$ a disjoint union of affine spaces $\A^i_k$. $X_i$ is called the $i$-skeleton of the filtration.

We recall the following result of Hoyois' (private communication). 

\begin{proposition}\label{prop:Cellular} Let $X$ be a smooth cellular $k$-scheme of dimension $n$ with $i$ skeleton $X_i$. Suppose that $X_i\setminus X_{i-1}$ is the disjoint union of $m_i$ copies of $\A^i$. Then $X$ is dualizable and 
\[
\chi(X/k)=\sum_{i=0}^n m_i\<-1\>^{n-i}.
\]
\end{proposition}

\begin{proof} Let $d$ be the minimum $i$ such that $X_i\neq\0$; the proof is by downward induction on $d$. If $d=n$, then $X=\amalg^{m_n}\A^n$, which is isomorphic in $\sH(k)$ to $\amalg^{m_n}\Spec k$, so $\chi(X/k)=m_n\cdot \<1\>$, proving the assertion in this case. If $d<n$, apply the induction hypothesis to $U:=X\setminus X_d$, which gives
\[
\chi(U/k)=\sum_{i=d+1}^n m_i\<-1\>^{n-i}.
\]
By Proposition~\ref{prop:EulerCharProperties}(3), we have
\[
\chi(X/k)=\chi(U/k)+\<-1\>^{n-d}\chi(\amalg^{m_d}\A^d)=\sum_{i=d}^n m_i\<-1\>^{n-i}.
\] 
\end{proof}

\begin{exs} 1. As a simple example, Proposition~\ref{prop:Cellular} gives another proof that
\[
\chi(\P^n_k/k)= (n+1)_\epsilon .
\]
2. Let $X$ be a Severi-Brauer variety over $k$ of  dimension $n$. The Euler characteristic of Severi-Brauer varieties have been computed by Hoyois (private communication). Using his quadratic refinement of the  Lefschetz trace formula \cite[Theorem 1.3]{HoyoisGL} and the fact that for a central simple algebra $\sA$ over $k$,  $\SL_1(\sA)$ is $\A^1$-connected, he shows
\[
\chi(X/k)=\chi(\P^n_k/k).
\]
In fact, the case of even $n$ follows from the fact that $X$ is split by a separable field extension $k\subset F$ of odd degree and  $\GW(k)\to \GW(F)$ is injective if $[F:k]$ is odd and $F/k$ is separable.

If $X$ has odd dimension $n$, then by Corollary~\ref{cor:OddEulerChar} below and the fact that 
$\chi(X/k)$ and $\chi(\P^n_k/k)$ have the same rank, we have
\[
\chi(X/k)=\frac{n+1}{2}\cdot H=\chi(\P^n_k/k).
\]
\end{exs}

\section{$\SL$--oriented and $\GL$--oriented theories}\label{sec:SLOrient}
We recall some basic facts about $\SL$-oriented and $\GL$-oriented ring spectra, Thom isomorphisms, and other related notions; for details, we refer the reader to  \cite{Anan, Anan19}. We also introduce the theories we will be using here: $\EM(\sK^{MW}_*)$,  $\EM(\sW_*)$,
 $\EM(\sK^{M}_*)$, representing the cohomology of the sheaves of Milnor-Witt $K$-theory, the Witt sheaf and the sheaves of Milnor $K$-theory, respectively. We will also discuss hermitian $K$-theory, represented by $\KO\in \SH(k)$, and Quillen $K$-theory, represented by $\KGL\in \SH(k)$.

For a commutative ring spectrum $E\in \SH(k)$ an {\em $\SL$-orientation}  is the assignment of a Thom class $\th(V,\rho)\in E^{2r, r}(\Th(V))$ for each pair $(V,\rho)$ consisting of a rank $r$ vector bundle $V\to X$, $X\in \Sm_k$, and an isomorphism $\rho:\det V\xrightarrow{\sim} \sO_X$, such that this assignment satisfies the axioms of \cite[Definition 3.3]{Anan19}. A commutative ring spectrum $E$ together with an $\SL$-orientation is called an $\SL$-oriented ring spectrum. 

Similarly, a choice of Thom classes  $\th(V)\in E^{2r, r}(\Th(V))$ for each rank $r$ vector bundle $V\to X$, satisfying the axioms of \cite[Definition 1.9]{Panin} for the associated Thom isomorphisms, is a  {\em $\GL$-orientation},  or simply, an orientation, for $E$. An oriented theory is automatically $\SL$-oriented; in the case of a $\GL$-orientation,  the Thom class is independent of the choice of isomorphism $\det V\cong \sO_X$. 

For $W\subset Y$ a closed subset of some $Y\in \Sm_k$, and $E\in \SH(k)$, one defines $E^{a,b}_W(Y):=E^{a.b}(Y/(Y\setminus W))$. For $V\to X$ a  rank $r$ vector bundle, let $\det V$ denote the line bundle $\Lambda^rV$ and write $\det^{-1}V$ for the dual of  $\det V$. The rank $r+1$ vector bundle $V\oplus \det^{-1}V$ has a canonical isomorphism $\can_V:\det(V\oplus \det^{-1}V)\xrightarrow{\sim} \sO_X$. For $L\to X$ a line bundle with zero-section $s_0:X\to L$, we define 
\[
E^{a,b}(X, L):=E^{a+2, b+1}(\Th(L))= E^{a+2, b+1}_{s_0(X)}(L)
\]
and for $Z\subset X$ a closed subset
\[
E^{a,b}_Z(X, L):= E^{a+2, b+1}_{s_0(Z)}(L).
\]

For an $\SL$-oriented ring spectrum $E$, an $X\in \Sm_k$ with closed subset $Z$ and rank $r$ vector bundle $p:V\to X$ with zero-section $s_0:X\to V$,   the Thom class $\th(V\oplus \det^{-1}V,\can_V)$ induces the {\em Thom isomorphism}
\begin{equation}\label{eqn:ThomIso}
\vartheta_V:E^{a,b}_Z(X)\xrightarrow{\sim}E^{a+2r, b+r}_{s_0(Z)}(V, p^*\det^{-1}(V))
\end{equation}
The {\em canonical Thom class} $\th(V)\in E^{2r, r}_{s_0(X)}(V, p^*\det^{-1}(V))$ is defined as $\vartheta_V(1_X)$, where $1_X\in E^{0,0}(X)$ is the unit. The Thom isomorphism satisfies
\[
\vartheta_V(x)=p^*(x)\cup \th(V),
\]
where $\cup$ is the cup product 
\[
E^{a,b}_{p^{-1}(Z)}(V)\times E^{2r, r}_{s_0(X)}(V, p^*\det^{-1}(V))\to
E^{a+2r, b+r}_{s_0(Z)}(V, p^*\det^{-1}(V)).
\]

Using the Thom isomorphisms and the six-functor formalism, one has functorial pushforward maps 
\[
f_*:E^{a,b}(X, \omega_{X/k}\otimes f^*(L))\to E^{a-2d,b-d}(Y, \omega_{Y/k}\otimes L) 
\]
for each proper map $f:X\to Y$ in $\Sm_k$ of relative dimension $d$. For $f$ the zero-section $s_0:X\to V$, as above, $s_{0*}(1_X)\in E^{2r,r}(V, p^*\det^{-1}V)$ is the image of $\th(V)$ under the ``forget the supports" map $E^{2r, r}_{s(X)}(V, p^*\det^{-1}(V))\to E^{2r, r}(V, p^*\det^{-1}(V))$

Applying this to the zero-section $s_0:X\to V$ for a rank $r$ vector bundle, one arrives at the Euler class
\[
e(V):=s_0^*s_{0*}(1_X)=s_0^*(\th(V))\in E^{2r, r}(X,\det^{-1}V)
\]
For $s:X\to V$ an arbitrary section, if $Z\subset X$ is a closed subset containing the closed subset $s^{-1}(s_0(X))$, we have the Euler class with supports
\[
e_Z(V, s):=s^*(\th(V))\in E^{2r, r}_Z(X,\det^{-1}V)
\]
mapping to $e(V)$ under $E^{2r, r}_Z(X,\det^{-1}V)\to E^{2r, r}(X,\det^{-1}V)$.

Before discussing the particular theories we will need, we recall some basic notions concerning {\em homotopy modules}. This setting will enable us to unify a number of arguments across different  cohomology theories. We refer the reader to \cite[\S 5]{MorelICTP},  \cite[Chapter 5]{MorelA1}, \cite{Feld} for details.

For $\sF$ a strictly $\A^1$-invariant Nisnevich sheaf on $\Sm_k$, we have the strictly $\A^1$-invariant sheaf $\sF_{-1}:=\sHom(\G_m, \sF)$. Recall \cite[Defintion 5.2.4]{MorelICTP} that a homotopy module is a sequence $(M_n)_{n\ge0}$ of strictly $\A^1$-invariant Nisnevich sheaves on $\Sm_k$ together with isomorphisms $\delta_n:M_n\to (M_{n+1})_{-1}$. For $n<0$, define $M_n$ inductively as $M_n:=(M_{n+1})_{-1}$. We let $\HM(k)$ denote the category of homotopy modules.

With the evident notion of morphism, $\HM(k)$ forms an abelian category. Via \cite[Theorem 5.2.6]{MorelICTP}, there is an equivalence from the category of homotopy modules on $\Sm_k$ to the heart of the homotopy $t$-structure on $\SH(k)$, which we denote by $M_*\mapsto \EM(M_*)$. For a homotopy module $M_*$, the corresponding cohomology theory $\EM(M_*)^{*,*}$ satisfies $\EM(M_*)^{a,b}(X)=H^{a-b}_\Nis(X, M_b)=H^{a-b}_\Zar(X, M_b)$.  Conversely, for $E\in \SH(k)$, the corresponding homotopy module is $\tau^t_0(E):=(\pi_{-n,-n}(E))_n$. 

The primary example of a homotopy module is given by the Milnor-Witt sheaves $\sK^{MW}_*$, about which we recall a few facts. For a field $F$, the {\em Milnor-Witt $K$-theory of $F$}, $K^{MW}_*(F)$, is the $\Z$-graded  $\Z$-algebra with generators $[u]\in K^{MW}_1(F)$, for each unit $u\in F^\times$ and an additional generator $\eta\in K^{MW}_{-1}(F)$, with relations given in \cite[Definition 6.3.1]{MorelICTP}. As explained in \cite[\S3.2]{MorelA1}, this construction extends to a Nisnevich sheaf $\sK^{MW}_*$ on $\Sm_k$, with stalk $K^{MW}_*(k(X))$ at the generic point $\eta_X\in X\in \Sm_k$. For a field $F$, sending the rank one quadratic form $\<u\>$, $u\in F^\times$, to the element $\<u\>:=1+[u]\eta\in K^{MW}_0(F)$ extends uniquely to an isomorphism of rings $\GW(F)\to K^{MW}_0(F)$ (see \cite[Lemma 6.3.8]{MorelICTP}). The Hopkins-Morel presentation of $K^{MW}_*(F)$ mentioned above extends to an analogous presentation of the sheaf $\sK^{MW}_*$ \cite[Definition 5.1]{GSZ} and Morel's isomorphism $\GW(F)\xrightarrow{\sim} K^{MW}_0(F)$ extends to an isomorphism of Nisnevich sheaves $\sGW\xrightarrow{\sim} \sK^{MW}_0$ \cite[Theorem 6.3]{GSZ} (assuming $k$ is infinite).

As explained in \cite[\S6]{MorelICTP}, the sheaf $\sK^{MW}_*$ defines a homotopy module on $\Sm_k$; in particular (we omit the ${}_\Nis$ and ${}_\Zar$ from the notation)
\[
\EM(\sK^{MW}_*)^{a,b}(X)=H^{a-b}(X, \sK^{MW}_b).
\]
Specifically, Morel's theorem identifying $\sK^{MW}_n$ with $\pi_{-n,-n}(\mS_k)$ shows that $(\sK^{MW}_n)_n=\tau_0(\mS_k)$. As $\mS_k$ is the unit in $\SH(k)$, this shows that for an arbitrary homotopy module $(M_n)_n$, $M_n$ is canonically a sheaf of $\sK^{MW}_0$-modules.

One also has a sheaf-theoretic description of the $L$-twisted theory for $L$ a line bundle on $X\in \Sm_k$. The  multiplication action of $\sK^{MW}_0$ on $\sK^{MW}_n$ induces an action of the sheaf of units $\sK^{MW\times}_0$ on $\sK^{MW}_n$. Sending a unit $u\in \sO_{X,x}^\times$ to  $\<u\>\in \sK^{MW}_0(\sO_{X,x})$ defines a homomorphism of sheaves of abelian groups $\G_m\to \sK^{MW\times}_0$. For $L\to X$ a line bundle, the action of $\sO_X^\times$ on $L$ makes $L$ into a $\Z[\sO_X^\times]$-module, and similarly the sheaf $\sK^{MW}_n$ on $X$ is a $\Z[\sO_X^\times]$-module. The twisted sheaf $\sK^{MW}_n(L)$ on $X$ is defined as
\[
\sK^{MW}_n(L):=\sK^{MW}_n\otimes_{\Z[\sO_X^\times]}L=\sK^{MW}_n\otimes_{\sK^{MW}_0}\sK^{MW}_0(L).
\]
See \cite[Section 1.2]{CalmesFasel} for details. For an arbitrary homotopy module $M_*$, we thus have the $L$-twisted version $M_n(L):=M_n\otimes_{\sK^{MW}_0}\sK^{MW}_0(L)$, defining the homotopy module $M_*(L)$.

For $X\in \Sm_k$ and $M_*$ a homotopy module, we have the {\em Rost-Schmid complex}
 (see \cite[\S5]{MorelA1} for details)
\begin{multline*}
C^*(X, M_n, L):=\oplus_{x\in X^{(0)}}i_{x*}M_n(k(x), i_x^*L\otimes\det^{-1}\mathfrak{m}_x/\mathfrak{m}_x^2)\\\to \oplus_{x\in X^{(1)}}i_{x*}M_{n-1}(k(x),i_x^*L\otimes\det^{-1}\mathfrak{m}_x/\mathfrak{m}_x^2)\to\\\ldots\to \oplus_{x\in X^{(p )}}i_{x*}M_{n-p}(k(x), , i_x^*L\otimes\det^{-1}\mathfrak{m}_x/\mathfrak{m}_x^2)\to\ldots
\end{multline*}
and a canonical isomorphism $H^p(X, M_n(L))=H^p(\Gamma(X, C^*(X, M_n, L))$. More generally, for $Z\subset X$ a closed subset, the part of $C^*(X, M_n, L)$ supported in $Z$,
\begin{multline*}
C^*_Z(X, M_n, L):=\oplus_{x\in X^{(0)}\cap Z}i_{x*}M_n(k(x), i_x^*L\otimes\det^{-1}\mathfrak{m}_x/\mathfrak{m}_x^2)\\\to \oplus_{x\in X^{(1)}\cap Z}i_{x*}M_{n-1}(k(x), i_x^*L\otimes\det^{-1}\mathfrak{m}_x/\mathfrak{m}_x^2)\to\\\ldots\to \oplus_{x\in X^{(p)}\cap Z}i_{x*}M_{n-p}(k(x),  i_x^*L\otimes\det^{-1}\mathfrak{m}_x/\mathfrak{m}_x^2)\to\ldots,
\end{multline*}
computes $H^p_Z(X, M_n(L))$ as $H^p(\Gamma(X, C^*_Z(X, M_n, L))$. 

Feld \cite{Feld} defines a category of {\em Milnor-Witt cycle modules} and shows in \cite[Theorem 4.2]{Feld2} that this category is equivalent to the category of homotopy modules. Via this equivalence the {\em Milnor-Witt complex} defined in \cite[\S 3.1]{Feld2} goes over the the Rost-Schmid complex; we will state and various results proven about the Milnor-Witt complexes for the corresponding Rost-Schmid complex without mentioning this correspondence explicitly. For example, the isomorphism $H^p_Z(X, M_n(L))\cong H^p(\Gamma(X, C^*_Z(X, M_n, L))$ stated above is a consequence of the acyclicity theorem \cite[Theorem 8.1]{Feld} for Milnor-Witt cycle modules.

For $Z\subset X$ a smooth closed subscheme of codimension $c$ with normal bundle $N_{Z/X}$,  the evident isomorphism 
\[
C^*_Z(X, M_n, L)\cong C^*(Z, M_{n-c}, i_Z^*L\otimes \det N_{Z/X})[-c]
\]
gives rise to the purity isomorphism \cite[Remarque 9.3.5]{FaselCW} 
\[
H^p_Z(X, M_n(L))\cong H^{p-c}(Z, M_{n-c}(i_Z^*L\otimes \det N_{Z/X})).
\]
More generally, for $W\subset Z\subset X$ closed, we have 
\begin{equation}\label{eqn:PurityIso}
H^p_W(X, M_n(L))\cong H^{p-c}_W(Z, M_{n-c}(i_Z^*L\otimes \det N_{Z/X})).
\end{equation}
It follows directly from the construction that these purity isomorphisms are functorial with respect to compositions of closed immersions. 

For $Z$ a closed subset of codimension $\ge c$, the complex $C^*_Z(X, M_n, L)$ is 0 in degrees $<c$, hence
\begin{equation}\label{eqn:CohVanishing}
H^p_Z(X, M_n(L))=0\text{ for } p< \codim_X Z.
\end{equation}

The  natural isomorphisms $\EM(M_*)^{a+b,b}_Z(X)\cong H^a_Z(X, M_b)$ extend to the twists by a line bundle $L\to X$. To see this, we have $\EM(M_*)^{a+b,b}_Z(X, L):=\EM(M_*)^{a+b+2,b+1}_{s_0(Z)}(L)$, with $s_0:X\to L$ the 0-section. The purity isomorphism
\[
H^{a+1}_{s_0(Z)}(L, M_{b+1})\cong H^a_Z(X, M_b(L))
\]
thus gives the isomorphisms
\begin{multline}\label{eqn:EMTwistedIso}
\EM(M_*)^{a+b,b}_Z(X, L):=\EM(M_*)^{a+b+2,b+1}_{s_0(Z)}(L)\\\cong H^{a+1}_{s_0(Z)}(L, M_{b+1})\cong H^a_Z(X, M_b(L))
\end{multline}
as claimed.

 For $L\to X$ a line bundle, the isomorphism $\sGW\to \sK^{MW}_0$ of sheaves on $X$ extends to an isomorphism $\sGW(L)\to \sK^{MW}_0(L)$, where $\sGW(L)$ is the sheaf of (virtual) $L$-valued non-degenerate quadratic forms. For $L'\to X$ a second line bundle, we have the isomorphism 
$\psi_{L'}:\sGW(L)\to \sGW(L\otimes L^{\prime\otimes 2})$
defined as follows: if $q:V\to L$ is an $L$-valued non-degenerate quadratic form, then $\psi_{L'}(q)$ is the induced form $V\otimes L\to L\otimes L^{\prime\otimes 2}$, which in local coordinates is given by $\psi_{L'}(q)(v\otimes\lambda)=q(v)\otimes\lambda^2$. Via the description of 
$M_n(L)$ as $M_n\otimes_{\sK^{MW}_0}\sGW(L)$, the isomorphism $\psi_{L'}:\sGW(L)\to \sGW(L\otimes L^{\prime\otimes 2})$ defines an isomorphism 
 \begin{equation}\label{eqn:SquareTwistIso}
 \psi_{L'}:M_*(L)\to M_*(L\otimes L^{\prime\otimes 2})
 \end{equation}
 of homotopy modules.

 Similarly, an isomorphism of line bundles $\rho:L\to L'$ induces an isomorphism of sheaves 
 \begin{equation}\label{eqn:TwistIso}
 \rho_*:
M_n(L)\to M_n(L').
\end{equation}
and a corresponding isomorphism on cohomology with supports
\[
\rho_*:H^p_Z(X, M_n(L))\to H^p_Z(X, M_n(L'))
\]

Let $X$ be in $\Sm_k$ with closed subset $Z$. Via the canonical isomorphism  $\EM(\sK^{MW}_*)^{a+b,b}(-)\cong H^{a-b}(-, \sK^{MW}_b)$, the suspension isomorphism 
\begin{multline*}
\EM(\sK^{MW}_*)^{a+b,b}(X/(X\setminus Z))\\
\xrightarrow{\alpha_n} \EM(\sK^{MW}_*)^{a+b+2n,b+n}((\A^n/(\A^n\setminus\{0\})\wedge X/(X\setminus Z))
\end{multline*}
transforms to the isomorphism
\begin{equation}\label{eqn:Suspension}
\alpha_n:H^a_Z(X, \sK^{MW}_b)\xrightarrow{\sim} H^{a+n}_{0\times Z}(\A^n\times X, \sK^{MW}_{b+n})
\end{equation}

\begin{lemma}\label{lem:PuritySuspension} The suspension isomorphism \eqref{eqn:Suspension} is equal to the inverse of the purity isomorphism $\beta_n:H^{a+n}_{0\times Z}(\A^n\times X, \sK^{MW}_{b+n})\xrightarrow{\sim} H^a_Z(X, \sK^{MW}_b)$.
\end{lemma}

\begin{proof} Since the suspension isomorphism  for some $n\ge1$ is the composition of suitable suspension isomorphisms for $n=1$, and the same holds for the purity isomorphism, we reduce to the case $n=1$.

We first handle the case $X=\Spec F$, $F$ a finitely generated field extension of $k$.

The suspension isomorphism $\alpha_1$ relies on the bonding isomorphism $\epsilon_b:\sK^{M}_b\to (\sK^{MW}_{b+1})_{-1}$ as follows: Letting $\G_m$ be the pointed scheme $(\A^1\setminus\{0\}, \{1\})$, we have 
\begin{multline*}
(\sK^{MW}_{b+1})_{-1}(-)=\sK^{MW}_{b+1}(\G_m\wedge (-)_+)\\=
\ker(\sK^{MW}_{b+1}(\A^1\setminus\{0\}\times (-))\xrightarrow{i_1^*} 
\sK^{MW}_{b+1}( 1\times (-))
\end{multline*}
so $\epsilon_b$ induces the map $\tilde{\epsilon}_{b,F}:K^{MW}_b)(F)\to  \sK^{MW}_{b+1}(\A^1_F\setminus\{0\})$ by  $\tilde{\epsilon}_b(x):=[t]\cdot x$. We can pass from the $\G_m$-loops to the $\P^1$-loops via the standard affine cover of $\P^1$, $\sU:=\{U_0, U_1\}$, $U_i=\P^1\setminus\{X_i=0\}$, giving the pushout diagram
\[
\xymatrix{
U_0\cap U_1\ar[r]\ar[d]&U_1\ar[d]\\
U_0\ar[r]&\P^1
}
\]
We view $U_0\cap U_1$ as the open subset $U_0\setminus \{(1:0)\}\cong \A^1\setminus\{0\}$. 
The map $\tilde{\epsilon}_b$ induces the map
\[
\xi_{b,F}:K^{MW}_b(F)\to H^1(\P^1_F, \sK^{MW}_{b+1})
\]
by sending $x\in K^{MW}_b(F)$ to the image of the \v{C}ech 1-cocycle $(\tilde{\epsilon}_b, U_0\cap U_1)$.  We have the isomorphisms
\begin{equation}\label{eqn:ExcisionIsos}
 H^1_{0}(\A^1_F, \sK^{MW}_{b+1})\cong H^1_0(\P^1_F, \sK^{MW}_{b+1})\cong H^1(\P^1_F, \sK^{MW}_{b+1})
 \end{equation}
the first being excision, and the second following from the strict $\A^1$-homotopy invariance of 
$\sK^{MW}_{b+1}$, so we may consider $\xi_{b, F}$ as a map
\[
\xi_{b,F}:K^{MW}_b(F)\to H^1_0(\A^1_F, \sK^{MW}_{b+1}),
\]
and this is the suspension isomorphism for $X=\Spec F$.

Via the Rost-Schmid complex, we have the isomorphism
\[
H^1_0(\A^1_F, \sK^{MW}_{b+1})\cong K^{MW}_b(F, F\cdot\del/\del t)
\]
and via this isomorphism, $\xi_b$ sends $x\in K^{MW}_b(F)$ to $x\otimes \del/\del t\in K^{MW}_b(F, F\cdot\del/\del t)$.  The purity isomorphism is this map, composed with the map
$K^{MW}_b(F, F\cdot\del/\del t)\to K^{MW}_b(F)$ sending $y\otimes  \del/\del t$ to $y$.

In general, the  map $\tilde{\epsilon}_b$  is represented by the map of Rost-Schmid complexes
\[
C^*(\tilde{\epsilon}_b):C^*_{Z}(X, \sK^{MW}_b)\to
C^*_{\A^1\setminus\{0\}\times Z}(\A^1\setminus\{0\}\times X, \sK^{MW}_{b+1})
\]
which on the summand $K^{MW}_{b-a}(k(y))$, $y\in y\in Z\cap X^{(a)}$, is the map 
\[
\tilde{\epsilon}_{b, k(y)}:K^{MW}_{b-a}(k(y))\to K^{MW}_{b-a+1}(k(t, y))
\]
with $K^{MW}_{b-a+1}(k(t, y))$ in the summand indexed by $(\eta,y)\in (\A^1\setminus\{0\}\times X)^{(a)}\cap \A^1\times Z$, where $\eta$ is the generic point of $\A^1$. 

The suspension map  is natural with respect morphisms of spaces, in particular, with respect to maps of schemes and with respect to  maps of the form $U\to U/(U\setminus p)$ for $p\in U$. This implies that our description of the suspension map for   $X=\Spec F$ extends termwise on the Rost-Schmid complex, to give the suspension map
\[
C^*_Z(\alpha_1):C^*_Z(X, \sK^{MW}_b)\to C^*_{0\times Z}(\A^1\times X, \sK^{MW}_{b+1}), 
\]
where $C^*_Z(\alpha_1)$ sends an element $x_y\otimes v\in K^{MW}_{b-a}(k(y))\otimes \det^{-1}\mathfrak{m}_y/\mathfrak{m}_y^2$ in the summand for $y\in Z\cap X^{(a)}$ to the element
$x_y\otimes \del/\del t\wedge v$ of $K^{MW}_{b-a}(k(y))\otimes \det^{-1}\mathfrak{m}_{0,y}/\mathfrak{m}_{0,y}^2$ in the summand indexed by $(0,y)$. As the purity isomorphism sends this latter element to $x_y\otimes v\in K^{MW}_{b-a}(k(y))\otimes \det^{-1}\mathfrak{m}_y/\mathfrak{m}_y^2$ in the summand for $y\in Z\cap X^{(a)}$, this completes the proof.
  \end{proof}
  
\begin{remark} It was not completely clear to us whether the map $\tilde{\epsilon}_{b,F}$ should send $x$ to $[t]\cdot x$ or $x\cdot [t]$, in other words, if the $\G_m$-suspension used to define the homotopy module $(\sK^{MW}_n)_n$ is the left or right smash product; we used the left smash product. However, in the  case of the right smash product, one would also replace $\A^1\times X$ with $X\times \A^1$ throughout, the map $C^*_Z(\alpha_1)$ would send $x_y\otimes v$ to $x_y\cdot [t]\otimes v$, then to $x_y\otimes v\wedge \del/\del t$, and the purity isomorphism would send this element back to $x_y\otimes v$, so the result would still hold.
\end{remark}

We now explain how one uses  the  purity isomorphism to define the $\SL$-orientation on $\EM(\sK^{MW}_*)$. 

Let $p:V\to X$ be  a rank $r$ bundle  with trivialized determinant $\rho:\det V\to \sO_X$. Via the purity isomorphism, the isomorphism $s_0^*N_{s_0(X)/V}\cong V$ gives the isomorphism
\[
H^r_{s_0(V)}(V, \sK^{MW}_r(p^*\det^{-1}V))\cong H^0(X, \sGW)
\]
via which we have the element $\th(V)\in H^r_{s_0(V)}(V, \sK^{MW}_r(p^*\det^{-1}V))$ corresponds to $\<1\>\in H^0(X, \sGW)$. If we have an isomorphism $\rho:\det V\xrightarrow{\sim} \sO_X$, applying the induced map $H^r_{s_0(V)}(V, \sK^{MW}_r(p^*\det^{-1}V))\xrightarrow{\sim}
H^r_{s_0(V)}(V, \sK^{MW}_r)$ to $\th(V)$ gives us the class $\th(V,\rho)\in H^r_{s_0(V)}(V, \sK^{MW}_r)$

\begin{proposition} 1. The assignment $(V\to X,\rho:\det V\xrightarrow{\sim} \sO_X)\mapsto 
\th(V,\rho)\in H^r_{s_0(V)}(V, \sK^{MW}_r)$, for $V$ a rank $r$ vector bundle on $X\in \Sm_k$ with trivialization $\rho$ of $\det V$ defines an $\SL$-orientation on $\EM(\sK^{MW}_*)$.\\[2pt]
2. For $p:V\to X$ a rank $r$ vector bundle on $X\in \Sm_k$, the element $\th(V)\in H^r_{s_0(V)}(V, \sK^{MW}_r(p^*\det^{-1}V))$ is the canonical Thom class associated to the $\SL$-orientation on 
$\EM(\sK^{MW}_*)$ given by (1).
\end{proposition}

\begin{proof} We note that the presheaf $X\mapsto \EM^{0,0}(\sK^{MW}_*)(X)=\sK^{MW}_0(X)$ is a Zariski sheaf on $\Sm_k$.  By \cite[Theorem 1.2]{Anan}, $\EM(\sK^{MW}_*)$ admits a unique ``normalized'' $\SL$-orientation, $(V,\rho)\mapsto \tilde{\th}(V,\rho)$; the proof of {\em loc. cit.} shows that the classes  $\tilde{\th}(V,\rho)$ are characterized by three properties:\\[5pt]
i.  For the trivial bundle $V=\sO_X^r$ with $\rho$ the canonical isomorphism $\det\sO_X^r=\sO_X$, 
$ \tilde{\th}(V,\rho)$ is the image of $1\in \EM^{0,0}(\sK^{MW}_*)(X)$ under the suspension isomorphism $\EM^{0,0}(\sK^{MW}_*)(X)\cong \EM^{2r,r}_{s_0(X)}(\sK^{MW}_*)(V)$.\\[2pt]
ii. The classes $\tilde{\th}(V,\rho)$ are natural with respect to vector bundle isomorphisms: if $f:V\to V'$ is an isomorphism of vector bundles on $X$ and we have trivializations $\rho:\det V\to \sO_X$, $\rho':\det V'\to \sO_X$ such that $\rho'\circ \det f=\rho$, then $f^*\tilde{\th}(V',\rho')=\tilde{\th}(V,\rho)$.\\[2pt]
iii. The classes $\tilde{\th}(V,\rho)$ are natural with respect to restriction by open immersions.\\[5pt]
The property (i) for the classes  $\th(V,\rho)$ follows from Lemma~\ref{lem:PuritySuspension} and the properties (ii)  and (iii) follows from the fact that the purity isomorphism is natural with respect to smooth morphisms. This proves (1).

For (2), the canonical Thom class $\tilde{\th}(V)\in \EM(\sK^{MW}_*)^{2r, r}_{s_0(X)}(V, p^*\det^{-1}V)$ is by definition the Thom class $\tilde{\th}(V\oplus \det^{-1}V,can_V)\in  \EM(\sK^{MW}_*)^{2r+2, r+1}_{s'_0(X)}(V\oplus\det^{-1}V)$, where 
$can:\det(V\oplus\det^{-1}V)\to \sO_X$ is the canonical isomorphism and  $s_0':X\to V\oplus\det^{-1}V$ is the 0-section. We have
\[
 \EM(\sK^{MW}_*)^{2r+2, r+1}_{s'_0(X)}(V\oplus\det^{-1}V)=H^{r+1}_{s'_0(X)}(V\oplus\det^{-1}V, \sK^{MW}_{r+1}). 
 \]
 Letting $s_V:V\to V\oplus\det^{-1}V$ be the 0-section over $V$, the normal bundle of $s_V$ is $p^*\det^{-1}V$, giving  the purity isomorphism
\[
 \beta: H^{r+1}_{s'_0(X)}(V\oplus\det^{-1}V, \sK^{MW}_{r+1})\xrightarrow{\sim} H^r_{s_0(X)}(V, \sK^{MW}_{r}(p^*\det^{-1}V))
 \]
Since the diagram of purity isomorphisms
\[
\xymatrix{
H^{r+1}_{s'_0(X)}(V\oplus\det^{-1}V, \sK^{MW}_{r+1})\ar[d]_\beta^\wr\ar[dr]^\sim\\
H^r_{s_0(X)}(V, \sK^{MW}_{r}(p^*\det^{-1}V))\ar[r]^-\sim&H^0(X, \sK^{MW}_0)\\
&
}
\]
commutes, we have $\th(V)=\beta(\tilde{\th}(V))$, which proves (2).
\end{proof}

The next theory we consider, $\EM(\sW_*)$, arises from the sheaf $\sW$ of Witt groups. For $F$ a field, the Witt group $W(F)$ is the quotient $\GW(F)/(H)$, where $(H)$ is the two-sided ideal generated by the hyperbolic form $H(x,y)=x^2-y^2$. Since $q\cdot H=\rnk(q)\cdot H$ for $q\in \GW(F)$,   $(H)$ is also the additive subgroup $\Z\cdot H$  of $\GW(F)$ generated by $H$. Via the isomorphism $\GW(F)\xrightarrow{\sim}K^{MW}_0(F)$, $H$ maps to the element $h:=2+\eta[-1]$.  The surjective map $\times\eta:K^{MW}_0(F)\to K^{MW}_{-1}(F)$ has kernel exactly $(h)$, identifying $W(F)$ with $K^{MW}_{-1}(F)$. For $n<0$, $\times \eta:K^{MW}_n(F)\to K^{MW}_{n-1}(F)$ is an isomorphism, so we have $W(F)\cong K^{MW}_n(F)$ for all $n<0$. All these assertions extend to the sheaf level, giving in particular an isomorphism
\[
\sW\cong \colim_{n\to-\infty}\sK^{MW}_n
\]
with the colimit taken with respect to the maps $\times\eta:\sK^{MW}_n\to \sK^{MW}_{n-1}$. For $L\to X$ a line bundle, this extends to an isomorphism $\sW(L)\cong \colim_n\sK^{MW}_n(L)$, where $\sW(L):=\sGW(L)/H\cdot \sGW(L)$. Defining $\sW_n:=\sW$, we have the homotopy module 
\[
\sW_*:=\colim_{\times\eta}\sK^{MW}_*, 
\]
the associated $T$-spectrum $\EM(\sW_*)$ and cohomology theory $\EM(\sW_*)^{a,b}(X)=H^{a-b}(X, \sW)$. The $\SL$-orientation for $\EM(\sK^{MW}_*)$ induces an $\SL$-orientation for $\EM(\sW_*)$.

We will also  use the $\SL$-oriented theory of hermitian $K$-theory $\KO$, see \cite{Schlichting, SchlichtingTripathi} for the basic construction and first properties. The canonical Thom class $\th(V)\in \KO^{2r,r}_{s_0(X)}(V;p^*\det^{-1}V)$ for a rank $r$ vector bundle $p:V\to X$ is the Koszul complex
\[
\Kos(V,\can):=\Lambda^rp^*V^\vee\to\ldots\to p^*V^\vee\to \sO_V
\]
endowed with the $\det^{-1}p^*V[r]$-valued quadratic form $p^*q_{V}:\Kos(V,\can)\otimes\Kos(V,\can)\to \det^{-1}p^*V[r]$ given by the exterior product maps
\[
-\wedge-:\Lambda^ip^*V^\vee[i]\otimes \Lambda^{r-i}p^*V^\vee[r-i]\to \Lambda^rp^*V^\vee[r].
\]
See \cite[Theorems 1.4, 5.1]{PaninWalter} for details. The Euler class is thus $(\oplus_i\Lambda^iV^\vee[i], q_V)\in \KO^{2r,r}(X; \det^{-1}V)$.

We will also use the $\GL$-oriented theories $\EM(K^M_*)$ associated to the homotopy module $\sK^M_*$, 
\[
\EM(K^M_*)^{a,b}(X)=H^{a-b}(X, \sK^M_b),
\]
and Quillen algebraic $K$-theory $\KGL$, $\KGL^{a,b}(X)=K_{2b-a}(X)$. Via Bloch's formula and purity, $H^n_Z(X, \sK^M_n)=\CH^{n-c}(Z)$ for $Z\subset X$ a smooth codimension $c$ closed subscheme of $X\in \Sm_k$,   the Thom class for $V$ is represented by the 0-section in $V$ and the Euler class is the top Chern class $c_r(V)$, $r=\rnk(V)$.  The Thom class in $K$-theory is represented by 
$\Kos(V,\can)$ and the Euler class is $\sum_{i=0}^r(-1)^i[\Lambda^iV^\vee]$. This all follows by a similar argument to what we used to construct $\SL$-orientations above,  from the fact that  $\EM(K^M_*)$ and $\KGL$ admit purity isomorphisms $\EM(K^M_*)^{a+2c, b+c}_Z(X)\cong \EM(K^M_*)^{a,b}(Z)$ and $\KGL^{a+2c, b+c}_Z(X)\cong \KGL^{a,b}(X)$ for $Z\subset X$ a codimenison $c$ closed immersion in $\Sm_k$. The purity isomorphism for $\EM(K^M_*)$ is a direct consequence of the Gersten resolution for $K^M_*$ \cite{Kerz}. For $\KGL$ the purity isomorphism is a consequence of Quillen's localization theorem for the $K$-theory of coherent sheaves \cite{Quillen} and the fact that $\KGL$ represents Quillen $K$-theory \cite{PPR}.

We have the surjection of homotopy modules $\pi:\sK^{MW}_*\to \sK^M_*=\sK^{MW}_*/(\eta)$ and the induced map $\EM(\pi):\EM(\sK^{MW}_*)\to \EM(\sK^M_*)$ is a map of $\SL$-oriented theories. Similarly, we have the morphism of ring spectra $\KO\to \KGL$, which is also a map of $\SL$-oriented theories. Finally, we have the homotopy module $\sI^*$, where $\sI\subset \sK^{MW}_0$ is the augmentation ideal for the rank homomorphism $\sK^{MW}_0\to \sK^M_0=\Z$ and $\sI^n$ is the $n$ power of this sheaf of ideals. In fact, $\sI^{n+1}$ is the kernel of the surjection $\pi:\sK^{MW}_n\to \sK^M_n$ \cite[Corollaire 5.4]{MorelPI}, which shows that $\sI^*$ is indeed a homotopy module\footnote{The result of Morel cited here is for fields, but this extends to sheaves using the fact that $\sK^{MW}_n$ and $\sK^M_n$ are unramified sheaves.}.

If the context does not make clear the choice of cohomology theory, we write $e^{\CW}$, $e^{\KO}$, $e^{\CH}$, $e^\KGL$ for the Euler classes for $\EM(K^{MW}_*)$, $\KO$, $\EM(\sK^M_*)$ and $\KGL$, respectively, and similarly for the Thom classes, pushforward maps, etc. We reserve the standard notation for Chern classes, $c_n$, for the Chern classes with values in $\CH^n(-)=H^n(-, \sK^M_n)$.

\section{Euler class and Euler characteristic}\label{sec:Euler}

We recall two special cases of the general motivic Gau{\ss}-Bonnet theorem of Deglis\'e-Jin-Khan \cite[Theorem 4.6.1]{DJK}.

\begin{theorem}\label{thm:CWEulerClass} Let $\pi_X:X\to \Spec k$ be a smooth projective dimension $d$ $k$-scheme, with tangent bundle $T_X\to X$.  Then   we have
\[
\chi(X/k)=\pi^\CW_{X*}(e^{\CW}(T_X))=\pi^\KO_{X*}(e^{\KO}(T_X))\in \GW(k)
\]
\end{theorem}
As consequence (see Remark~\ref{rem:CatTopEuler}), we have classical versions of Gau{\ss}-Bonnet:
\[
\chi^{top}(X)=\rnk\chi(X/k)=\pi^\CH_{X*}(e^{\CH}(T_X))=\pi^\KGL_{X*}(e^\KGL(T_X))\in \Z.
\]
With Raksit,  have also given a proof of a motivic Gau{\ss}-Bonnet formula \cite[Theorem 1.5]{LREulerChar} in the setting of $\SL$-oriented theories.

To give Theorem~\ref{thm:CWEulerClass} a concrete expression, we have shown in the proof of  \cite[Theorem 8.4, Theorem 8.6]{LREulerChar} that the pushforward maps for $\EM(\sK^{MW}_*)$ defined using the $\SL$-orientation and the six-functor formalism agree with those defined by Fasel and Fasel-Srinivas \cite{FaselCW, FaselSri} and those for $\KO$ agree with the ones defined by Grothendieck-Serre duality, and used by Calm\'es-Hornbostel \cite{CalmesHornbostel} (for the Witt groups). Similarly, the six-functor pushforward maps for 
$\EM(\sK^{M}_*)$ and $\KGL$ are the ``standard'' ones: on $\EM(\sK^{M}_*)^{2*,*}=\CH^*$, the standard ones are the classical pushforward maps on $\CH^*$ and on $\KGL^{2*,*}=K_0$, these are the usual $\sF\mapsto \sum_i(-1)^iR^if_*\sF$, using Quillen's resolution theorem to identify $K_0(X)$ with the Grothendieck group of coherent sheaves on $X$, for $X\in \Sm_k$.

We give a first consequence of the Gau{\ss}-Bonnet theorem.

\begin{corollary}\label{cor:OddEulerChar} Let $Y$ be an integral smooth projective $k$-scheme of odd dimension over $k$. Then the Euler characteristic $\chi(Y/k)$  is hyperbolic: $\chi(Y/k)=m\cdot H$ for some $m\in \Z$,  hence $\text{rank}(\chi(Y/k))=2m$ is even.
\end{corollary}

The proof is based on the following lemma, which is also of independent interest.

\begin{lemma}\label{lemma:OddEulerClass} Let $\pi:V\to Y$ be a vector bundle of odd rank $r$ over some $Y\in \Sm_k$. Then for all $u\in k^\times$, 
\[
e^\CW(V)=\<u\>\cdot e^\CW(V)
\]
in $H^r(Y,\sK^{MW}_r(\det^{-1}V))$. Moreover, 
\[
\eta\cdot e^\CW(V)=0
\]
in $H^r(Y,\sK^{MW}_{r-1}(\det^{-1}V))$
\end{lemma}

\begin{proof}   Let $\phi_u:V\to V$ be the map multiplication by $u$. The naturality of the Thom class says that 
\[
(\phi_u,\det^{-1}\phi_u)^*(\th^\CW(V))=\th^\CW(V)
\]
Since $\det^{-1}\phi_u:\det^{-1}V\to \det^{-1}V$ is  multiplication by $u^{-r}$, $\det^{-1}\phi_u^*$ acts by  multiplication by $\<u^{-r}\>=\<u\>$ on the sheaf $\sK^{MW}_r(\pi^*\det^{-1}V)$. Thus
$(\id, \det^{-1}\phi_u)^*$ acts by $\times\<u\>$ on $H^r(V,\sK^{MW}_r(\pi^*\det^{-1}V))$. 

Letting $\phi_u^*:=(\phi_u, \id)^*$, we have $(\phi_u,\det^{-1}\phi_u)^*=\phi_u^*\circ (\id, \det^{-1}\phi_u)^*$, and thus 
\[
\phi_u^*(\th^\CW(V))=\<u\>\cdot \th^\CW(V).
\]
Since $\phi_u\circ s_0=s_0$, we have
\begin{align*}
e^\CW(V)&=s_0^*(\th^\CW(V))\\
&=s_0^*(\phi_u^*(\th^\CW(V)))\\
&=s_0^*(\<u\>\cdot\th^\CW(V))\\
&=\<u\>\cdot e^\CW(V).
\end{align*}

We now show that $\eta\cdot e^\CW(V)=0$. Let $\sO$ be the local ring $k[t]_{(t)}$, with quotient field $K:=k(t)$ and residue field $k$. Let $\sY=Y\times_k\sO$, with open subscheme $j:Y_K\to \sY$, closed complement $i:Y\to \sY$ and projections $p:\sY\to Y$, $p_K:Y_K\to Y$ We have the exact localization sequence
\begin{multline*}
\ldots\to H^r(\sY, \sK^{MW}_m(p^*\det^{-1}V)) \xrightarrow{j^*}
H^r(Y_K, \sK^{MW}_m(p_K^*\det^{-1}V))\\\xrightarrow{\delta} H^r(Y, \sK^{MW}_{m-1}(\det^{-1}V))\to\ldots
\end{multline*}
and a similar sequence for $\sO$,
\[
\ldots\to K^{MW}_n(\sO)\xrightarrow{j_0^*}K^{MW}_n(K)\xrightarrow{\delta_0} K^{MW}_{n-1}(k)\to\ldots\ .
\]

We claim that
\[
\delta(\alpha\cdot  p_K^*x)=\delta_0(\alpha)\cdot x
\]
for $x\in H^r(Y, \sK^{MW}_m(\det^{-1}V))$, $\alpha\in K^{MW}_n(K)$. To see this, first represent $x$ as an $r$-cocycle in the Rost-Schmid complex $C^*(Y, \sK^{MW}_m(\det^{-1}V))$,
\[
x=[\sum_{y\in Y^{(r )}}x_y\otimes v_y;\quad x_y\in K^{MW}_{m-r}(k(y)), v_y\in \det^{-1}\mathfrak{m}_y/\mathfrak{m}_y^2\otimes \det^{-1}V].
\]
Here the $[-]$ means: take the associated cohomology class.
This represents $\alpha\cdot  p_K^*x$ as
\[
\alpha\cdot  p_K^*x=[\sum_{y\in Y^{(r )}}\alpha\cdot p_K^*x_y\otimes v_y],
\]
where $\alpha\cdot p_K^*x_y\otimes v_y$ is in the summand corresponding to $y_K:=\Spec K\otimes_kk(y)$. Thus $\delta(\alpha\cdot  p_K^*x)$ is represented by 
\[
\delta(\alpha\cdot  p_K^*x)=[\sum_{y\in Y^{(r )}}\delta_y(\alpha\cdot p_K^*x_y\otimes  v_y)]
\]
where $\delta_y: K^{MW}_{m-r}(K\otimes_kk(y))\to  K^{MW}_{m-r-1}(k(y))$ is the boundary map associated to the DVR $\sO\otimes_kk(y)=k(y)[t]_{(t)}$ with parameter $t$. But since 
$p_K^*x_y$ extends to an element of $K^{MW}_{m-r}(\sO\otimes_kk(y))$ that maps to $x_y$ under the quotient map $sO\otimes_kk(y)\to k(y)$,  we have
$\delta_y(p_K^*x_y)=0$ and the explicit formula for $\delta_y$ and $\delta_0$ show that
\[
\delta_y(\alpha\cdot p_K^*x_y)=\delta_0(\alpha)\cdot x_y.
\]

 Taking $x=e^\CW(V)$, $\alpha=\<t\>$, and noting that $\delta_0(\<t\>)=\delta_0(1+\eta\cdot [t])=\eta$, this gives
\[
\delta(\<t\>\cdot p_K^*e^\CW(V))=\eta\cdot e^\CW(V).
\]
Similarly,
\[
\delta(p_K^*e^\CW(V))=0
\]
and since 
\[
\<t\>\cdot p_K^*e^\CW(V)=\<t\>\cdot e^\CW(p_K^*V)=e^\CW(p_K^*V)=p_K^*e^\CW(V)
\]
we have $\eta\cdot e^\CW(V)=0$.
\end{proof}

\begin{proof}[Proof of Corollary~\ref{cor:OddEulerChar}] Suppose $Y$ is integral of odd dimension $d$ over $k$. Applying Lemma~\ref{lemma:OddEulerClass}, we have
\[
\eta\cdot e^\CW(T_Y)=0\in H^d(Y, \sK^{MW}_{d-1}(\omega_{Y/k}));
\]
pushing forward to $\Spec k$ and using Theorem~\ref{thm:CWEulerClass} gives
\[
\eta\cdot \chi(Y/k)=0.
\]
Via the isomorphisms $\GW(k)\cong K^{MW}_0(k)$, $W(k)\cong K^{MW}_{-1}$, the surjection $\times \eta:K^{MW}_0(k)\to K^{MW}_{-1}(k)$ transforms to the canonical surjection $\GW(k)\to W(k)$ with kernel the ideal generated by the hyperbolic form $H$. As $q\cdot H=\rnk(q)\cdot H$ in $\GW(k)$ for each $q\in \GW(k)$, the identity $\eta\cdot \chi(Y/k)=0$ says $\chi(Y/k)=m\cdot H$ for some $m\in \Z$. Since $\rnk(m\cdot H)=2m$, this finishes the proof.
\end{proof}

\begin{proposition}\label{prop:EulerVanishing} Let $\pi:V\to Y$ be a vector bundle of odd rank $r$ over some $Y\in \Sm_k$. Then the Euler class $e^\sW(V)\in H^r(Y, \sW(\det^{-1}V))$ is zero.
\end{proposition}

\begin{proof} Since $\sW\cong\sK^{MW}_*[\eta^{-1}]$ and $e^\sW(V)$ is the image of $e^\CW(V)$ under the canonical map $H^r(Y, \sK^{MW}_r(\det^{-1}V))\to H^r(Y, \sW(\det^{-1}V))$, this follows from Lemma~\ref{lemma:OddEulerClass}.
\end{proof}

An analog of part of Lemma~\ref{lemma:OddEulerClass} holds for $e^\KO$; the proof is even easier. Recall the {\em hyperbolic map} $h_L:\KGL^{a,b}(Y)\to \KO^{a,b}(Y; L)$ (see, e.g. \cite[\S 4.7]{Schlichting2}). For a vector bundle $V$ and $(a,b)=(2r,r)$,  $h_L(V)=(V\oplus V^\vee\otimes L[r], h(\can))$, where $\can:V\times V^\vee\otimes L[r]\to L[r]$ is the canonical pairing, and
\[
h(\can)=\begin{pmatrix} 0&\can\\\can&0\end{pmatrix}:(V\oplus V^\vee\otimes L)\times(V\oplus V^\vee\otimes L[r])\to L[r]
\]

\begin{lemma}\label{lemma:OddKOEulerClass}  Let $\pi:V\to Y$ be a vector bundle of odd rank $r$ over some $Y\in \Sm_k$. Then $e^\KO(V)=h_{\det^{-1}V}(\oplus_{i=0}^{[r/2]}(-1)^i[\Lambda^iV^\vee])$ in $\KO^{2r,r}(Y, \det^{-1}V)$. As consequence $\<u\>\cdot e^\KO(V)=e^\KO(V)$ for all $u\in \Gamma(Y, \sO_Y^\times)$ and $\eta\cdot e^\KO(Y)=0$.
\end{lemma}

\begin{proof} This follows easily from the explicit form of $e^\KO(V)$ as
\[
e^\KO(V)=(\oplus_{i=0}^r\Lambda^iV^\vee[i], q_V)
\]
where $q_V$ is the sum of the exterior product maps $-\wedge -:\Lambda^iV^\vee[i]\otimes \Lambda^{r-i}V^\vee[r-i]\to \Lambda^rV^\vee[r]$. The induced isomorphism $\Lambda^{r-i}V^\vee[r-i]\otimes \Lambda^rV[-r]\cong (\Lambda^iV^\vee[i])^\vee$ gives an isomorphism of the   restriction $q_{V,i}$ of $q_V$,
\[
q_{V,i}:(\Lambda^iV^\vee[i]\oplus  \Lambda^{r-i}V^\vee[r-i])\otimes(\Lambda^iV^\vee[i]\oplus \Lambda^{r-i}V^\vee[r-i]\to \Lambda^rV^\vee[r],
\]
with the hyperbolic form on $\Lambda^iV^\vee[i]$, which gives the identity
\[
h_{\det^{-1}V}(\sum_{i=0}^{[r/2]}(-1)^i[\Lambda^iV^\vee])=(\oplus_{i=0}^r\Lambda^iV^\vee[i], q_V)
\]
in $\KO^{2r,r}(Y, \det^{-1}V)$.

The two further assertions follow from $\<u\>\cdot h_L(x)=h_L(x)$ and $\eta\cdot h_L(x)=0$ for all $x\in \KGL^{2r,r}(Y)$, $u\in \Gamma(Y, \sO_Y^\times)$, $L\to X$ a line bundle.
\end{proof}

 \begin{remark} One can also prove Corollary~\ref{cor:OddEulerChar} using  $\chi(X/k)=\pi_{X*}(e^\KO(T_X))$ and the explicit form this latter pushforward takes; this is the proof given in \cite[Corollary 8.7]{LREulerChar}.
 \end{remark}

\section{Local indices}\label{sec:Local} We consider the problem of computing the Euler class with support associated to a section $s$ of  a vector bundle $\pi:V\to X$ on a smooth $k$-scheme $X$. Kass and Wickelgren \cite{KassWickelgren} have  defined a ``degree of the Euler class" for a so-called relatively oriented vector bundle $V$ on a smooth and proper $k$-scheme $X$, assuming that $V$ has rank equal to the dimension of $X$ and comes with a section having isolated zeros (plus some additional technical assumptions). Their definition relies on the construction of an explicit symmetric bilinear form associated to the given section $\sigma$ and a zero of $\sigma$, going back to work of  Scheja and Storch \cite{SchejaStorch}. Bachmann and Wickelgren \cite{BachmannWickelgren} have refined this method and their results show  that the Scheja-Storch form computes the local Euler class as defined above, for a section with isolated zeros, without the introduction of a relative orientation. We recall the definition of the Scheja-Storch form here and explain how the results of Bachmann-Wickelgren give this computation.

Let $\sO$ be a regular local ring with residue field $k$ and maximal ideal $\mathfrak{m}$. We assume that the quotient map $\pi:\sO\to k$ splits, that is, $\sO$ is a $k$-algebra. Let $t_*:=t_1,\ldots, t_n$ be a system of parameters for $\sO$ and let $s_*:=s_1,\ldots, s_n$ be elements of $\mathfrak{m}$ such that the ideal $(s_1,\ldots, s_n)$ is $\mathfrak{m}$-primary. Let $J(s_*)=\sO/(s_1,\ldots, s_n)$. Then $J(s_*)$ is a finite dimensional $k$-algebra with quotient map $p:\sO\to J(s_*)$.   

For an element $f\in \sO$, let $f^\delta=f\otimes 1-1\otimes f\in  \sO\otimes_k\sO$, and let $I_\delta\subset \sO\otimes_k\sO$ be the ideal $(t_1^\delta,\ldots, t_n^\delta)$. One sees easily that $I_\delta$ is the kernel of the multiplication map $\mu:\sO\otimes_k\sO\to \sO$ and that $f^\delta$ is in $I_\delta$ for all $f\in \mathfrak{m}$. In particular, there are elements $a_{ij}\in  \sO\otimes_k\sO$ with
\begin{equation}\label{eqn:SSElt}
s_i^\delta=\sum_{j=1}^n a_{ij}\cdot t_j^\delta;\quad i=1,\ldots, n.
\end{equation}
The {\em Scheja-Storch element} $e_{t_*,s_*}\in J(s_*)$ is defined as 
\[
e_{t_*, s_*}:=(p\otimes \pi)(\det(a_{ij}))\in J(s_*)\otimes_kk=J(s_*).
\]

Let $\Delta_{t_*, s_*}\in J(s_*)\otimes_k J(s_*)$ be the element $(p\otimes p)(\det(a_{ij}))$. By \cite[Satz 3.3]{SchejaStorch}, the map
\[
\Theta_{t_*, s_*}: \Hom_k(J(s_*), k)\to J(s_*)
\]
defined by $\Theta_{t_*, s_*}(\phi):=(\phi\otimes\id)(\Delta_{t_*, s_*})$ is an isomorphism (of $J(s_*)$-modules).
Let  $\eta_{t_*, s_*}:=\Theta_{t_*, s_*}^{-1}(1)$; $\eta_{t_*, s_*}$ is called the {\em generalized trace} in \cite{SchejaStorch}.

We summarize the main facts about $J(s_*)$ and $e_{t_*, s_*}$.  

\begin{theorem}\label{thm:SS} 1. $e_{t_*, s_*}\in J(s_*)$ is independent of the choice of the $a_{ij}$.\\[2pt]
2. The socle of $J(s_*)$, that is, $\{x\in J(s_*)\mid \mathfrak{m}\cdot x=0\}$, is a one-dimensional $k$-vector space, with generator $e_{t_*, s_*}$.\\[2pt]
3. Let $\Tr:J(s_*)\to k$ be a $k$-linear map such that $\Tr(e_{t_*, s_*})=1$. Then the bilinear form on   $J(s_*)$ 
\[
B_{s_*, t_*}(x,y):=\Tr(xy)
\]
is non-degenerate, and $[B_{s_*, t_*}]\in \GW(k)$ is independent of the choice of $\Tr$ (satisfying $\Tr(e_{t_*, s_*})=1$).\\[2pt]
4. Suppose we have a new system of parameters $(t_1',\ldots, t_n')$ for $\mathfrak{m}$ and a second set of generators $(s_1',\ldots, s_n')$ for the ideal $(s_1,\ldots, s_n)$. Write
\[
t_i'=\sum_j\alpha_{ij}t_j,\ s_i'=\sum_j\beta_{ij}s_j
\]
for $\alpha_{ij},\beta_{ij}\in \sO$. Let $\alpha,\beta\in k$ be the respective images of $\det(\alpha_{ij})$, $\det(\beta_{ij})$ in $k$. Then
\[
[B_{s'_*, t'_*}]=\<\alpha\cdot\beta\>\cdot [B_{s_*, t_*}]\in \GW(k)
\]
5. If we write $s_i=\sum_j\bar{a}_{ij}t_j$ in $\sO$, then $e_{t_*, s_*}=\det(\bar{a}_{ij})$.\\
6. The map $\eta_{t_*, s_*}:J(s_*)\to k$ satisfies $\eta_{t_*, s_*}(e_{t_*, s_*})=1$, hence
\[
[B_{s_*, t_*}]=[(x,y)\mapsto \eta_{t_*, s_*}(xy)]\in \GW(k).
\]
\end{theorem}

\begin{proof} By \cite[Lemma 1.2($\alpha$)]{SchejaStorch}, $(p\otimes p)(\det(a_{ij}))\in J(s_*)\otimes_k J(s_*)$ is independent of the choice of the $a_{ij}$, which proves (1). (2) is proven in \cite[Lemma 3]{KassWickelgren}. (3) follows from \cite[Lemma 5]{KassWickelgren} and (4) follows from the transformation law \cite[Satz 1.1]{SchejaStorchRes}. For (5),  we apply $(p\otimes \pi)$ to the equation \eqref{eqn:SSElt}, giving $s_i=\sum_j (p\otimes \pi)(a_{ij})\cdot t_j$, so $e_{t_*, s_*}=\det((p\otimes \pi)(a_{ij}))$. On the other hand, if $s_i=\sum_j\bar{a}_{ij}t_j$, then  again by 
\cite[Lemma 1.2($\alpha$)]{SchejaStorch}, we have $\det(\bar{a}_{ij})=\det((p\otimes \pi)(a_{ij}))=e_{t_*, s_*}$. 

For (6), choose a $k$-basis $b_1,\ldots, b_n$ for $J(s_*)$.  We may assume that $b_1=1$ and $b_n=e_{t_*, s_*}$ (unless $n=1$, in which case we take $b_1=1$) and that $b_j$ is in $\mathfrak{m}/(s_1,\ldots, s_n)$ for $1<j<n$. Write
\[
\Delta_{t_*, s_*}=\sum_{ij}b_{ij}\cdot b_i\otimes b_j;\quad b_{ij}\in k.
\]
By \cite[Satz 3.1]{SchejaStorch} $b_{ij}=b_{ji}$ for all $i,j$. Then $e_{t_*, s_*}=(p\otimes \pi)(\Delta_{t_*, s_*})=(\pi\otimes p)(\Delta_{t_*, s_*})$, so $b_{i1}=b_{1i}=0$ for $i<n$ and $b_{1n}=b_{n1}=e_{t_*, s_*}$. By construction, 
\[
1= (\eta_{t_*, s_*}\otimes\id)(\Delta_{t_*, s_*})=\eta_{t_*, s_*}(e_{t_*,s_*})\cdot 1+\sum_{i,j=2}^n\eta_{s_*}(b_{ij})\cdot b_j,
\]
 so $\eta_{t_*,s_*}(e_{t_*,s_*})=1$. 
\end{proof}

\begin{remark} With $\sO$, $t_*$ and $s_*$ as above, the $k$-algebra $J(s_*)$ has dualizing module
\[
\omega_{J(s_*)/k}\cong \omega_{\sO/k}\otimes_\sO \det^{-1} (s_1,\ldots, s_n)/(s_1,\ldots, s_n)^2
\]
so the basis elements  
$\bar{s}_1\wedge\ldots\wedge\bar{s}_n$ for $ \det (s_1,\ldots, s_n)/(s_1,\ldots, s_n)^2$ and $dt_1\wedge\ldots\wedge dt_n$ for $\omega_{\sO/k}$ give an isomorphism $\omega_{J(s_*)/k}\cong J(s_*)$. Via this isomorphism, the   isomorphism given by Grothendieck duality theory
\[
\Hom_{k}(J(s_*), k)\cong \Hom_{J(s_*)}(J(s_*), \omega_{J(s_*)/k})\cong J(s_*)
\]
is  given by $\phi\mapsto (\phi\otimes\id)(\Delta_{t_*, s_*})$; this is proven in \cite[Theorem 2.18]{BachmannWickelgren}

\end{remark}

Let $X$ be a smooth $k$-scheme of dimension $d$ over $k$ and let $x\in X$ be a closed point. We have the purity isomorphism \eqref{eqn:PurityIso}
\[
H^d_x(X,\sK^{MW}_d(L))\cong \GW(k(x), \det^{-1}\mathfrak{m}/\mathfrak{m}^2\otimes L).
\]

\begin{corollary}\label{cor:SSLocalEulerClass} Let $k$ be a field. Let $p:V\to X$ be a rank $d$ vector bundle, with  $X\in \Sm_k$ of dimension $d$ over $k$ and let $s:X\to V$ be a section. Suppose a closed point $x\in X$ is an isolated zero of $s$; suppose in addition that $k(x)$ is a separable extension of $k$. Choose a framing $e_1,\ldots, e_d$ for $V$ in a neighborhood of $x$ and let $t_*:=t_1,\ldots, t_d$ be a system of parameters for the maximal ideal $\mathfrak{m}_x\subset \sO_{X,x}$.  Write $s=\sum_{i=1}^ds_i e_i$ near $x$ and let $s_*=s_1,\ldots, s_d$. Then
\[
e_x(V, s)\in H^d_x(X, \sK^{MW}_d(\det^{-1}V))= \GW(k(x), \det^{-1}\mathfrak{m}/\mathfrak{m}^2\otimes \det^{-1}V)
\]
is given by 
\[
e_x(V, s)=[B_{s_*, t_*}]\otimes \frac{\del}{\del t_1}\wedge\ldots\wedge  \frac{\del}{\del t_d}\otimes
(e_1\wedge\ldots\wedge e_d)^{-1}
\]
\end{corollary}

\begin{proof} Since the finite extension $k\subset k(x)$ is separable, we have a canonical isomorphism
\[
\omega_{X/k}\otimes k(x)\cong \det\mathfrak{m}/\mathfrak{m}^2
\]
The choice of framing $e_1,\ldots, e_d$ and the choice of parameters $t_1,\ldots, t_d$ uniquely define an isomorphism in a neighborhood  $U$ of $x$
\[
\phi: \det^{-1} V\xrightarrow{\sim} \omega_{X/k}
\]
by $\phi((e_1\wedge\ldots\wedge e_d)^{-1})= dt_1\wedge\ldots\wedge dt_n$. Via this isomorphism, we have the corresponding  isomorphism
\[
 \det^{-1}\mathfrak{m}/\mathfrak{m}^2\otimes \det^{-1}V\xrightarrow{\sim} k(x)
 \]
giving the isomorphism
\[
\GW(\phi): \GW(k(x), \det^{-1}\mathfrak{m}/\mathfrak{m}^2\otimes \det^{-1}V)\to\GW(k(x))
\]
sending $[B_{s_*, t_*}]\cdot \frac{\del}{\del t_1}\wedge\ldots\wedge  \frac{\del}{\del t_d}\otimes
(e_1\wedge\ldots\wedge e_d)^{-1}$ to $[B_{s_*, t_*}]$.

Note that both $e_x(V, s)$ and $[B_{s_*, t_*}]\cdot \frac{\del}{\del t_1}\wedge\ldots\wedge  \frac{\del}{\del t_d}\otimes
(e_1\wedge\ldots\wedge e_d)^{-1}$ are unchanged if we replace $X$ with a Nisnevich neighborhood $(X',x)\to (X,x)$, so we may assume that $k=k(x)$. Moreover, the isomorphism $\phi$ defines a relative orientation for $V$ over $U$. In this case, \cite[Proposition 2.32 and Theorem 7.6]{BachmannWickelgren} says that 
$\GW(\phi)(e_x(V, s))=[B_{s_*, t_*}]$, which proves the result.
\end{proof}

\begin{remark} If $k$ is perfect, the separability assumption in the statement of Corollary~\ref{cor:SSLocalEulerClass} is automatically satisfied; we can always reduce to this case by the base-change $k\to k^{perf}$ following Remark~\ref{rems:NonperfectFields}.
\end{remark}

\begin{ex}\label{ex:Diagonal} As an example, suppose we have a local framing $e_1,\ldots, e_d$ for $V$ near $x$, local parameters $t_1,\ldots, t_d$, units $u_i\in \sO_{X,x}^\times$ and positive integers $n_i$    such that $s=\sum_iu_it_i^{n_i}e_i$; we call such a section ``diagonalizable''. The Scheja-Storch element is  $e_{t_*,s_*}=\prod_i u_it_i^{n_i-1}\in J(s_*)=\sO_{X,x}/(t_1^{n_1},\ldots, t_d^{n_d})$. If $d=1$, $n=n_1$, $u=u_1$,  the Scheja-Storch form  has class
\[
[B(u, n)]=\begin{cases} (n/2)\cdot H&\text{ for $n$ even}\\ \<u\>+(1/2)(n-1)\cdot H&\text{ for $n$ odd}
\end{cases}
\]
and in general $[B_{s_*, t_*}]=\prod_{i=1}^d [B(u_i, n_i)]$.  Since $x\cdot H=\rnk(x)\cdot H$  for $x\in \GW$, we see that $[B_{s_*, t_*}]=(1/2)(\prod_in_i)\cdot H$ if at least one $n_i$ is even and is $\<\prod_iu_i\>+(1/2)(\prod_in_i-1)\cdot H$ if all $n_i$ are odd. We can also express this as
\[
[B_{s_*, t_*}]=\<u\> n_\epsilon
\]
where $u=\prod_iu_i$ and $n=\prod_in_i$. 
\end{ex}

\begin{ex}\label{ex:NonDegn} We consider the simplest case  of a vector bundle $\pi:V\to X$ of rank $d=\dim_kX$ with a section $s:X\to V$ which is transverse to the zero-section. If $x\in X$ is a zero of $s$, choose a basis of sections $\lambda_1,\ldots, \lambda_d$ of $V$ in a neighborhood of $x$, and a system of parameters $t_1,\ldots, t_d\in \mathfrak{m}_x$. As in the proof  of Corollary~\ref{cor:SSLocalEulerClass}, we may assume that $k(x)=k$. If we write $s$ as $s=\sum_{i=1}^ds_i\lambda_i$, the condition that $s$ is transverse to the zero-section at $x$ translates into the the fact that the matrix
\[
\del s/\del t_{|x}:=\begin{pmatrix}\del s_i/\del t_j\end{pmatrix}\in M_{d\times d}(k(x))
\]
is invertible.  We then have $J(s_*)=k$ and the Scheja-Storch form is $\del s/\del t_{|x}$. By 
Corollary~\ref{cor:SSLocalEulerClass}, we thus have 
\begin{equation}\label{eqn:localNDFormula}
e_x(V;s)=\<\det(\del s/\del t_{|x})\>\otimes \wedge\lambda_*^{-1}\otimes\wedge\del/\del t_*.
\end{equation}
\end{ex}

\begin{remark} Although we are restricting to the case of a bundle of rank equal to the dimension of the base-scheme, this is not an essential restriction. The local Euler class $e_{s=0}(V,s)$ for a vector bundle $V\to Y$ is determined by the restriction to $\Spec \sO_{Y,y}$ for all generic points $y$ of $(s=0)$. Working over a perfect field $k$, we can always find a subfield $K$ of $\sO_{Y,y}$ such that $k(y)$ is a separable extension of $K$. Thus, we can reduce to the case $\text{rank}(V)=\dim Y$ if the section $s$ has zero-locus of codimension equal to the rank of $V$\end{remark}

\section{Cohomology of classifying spaces}

Let $G$ be an algebraic group over $k$. We use the so-called geometric classifying space to define $BG$ as an object of $\sH(k)$. This follows the method introduced by Totaro \cite{Totaro}, with details to be found in \cite[\S 4.2]{MorelVoevodsky}. One takes a faithful representation $\rho:G\to \GL_N$ for some $N$ and considers an increasing sequence of compatible $\GL_N$-representations $V_0\subset V_1\subset\ldots\subset V_m\subset\ldots$ and $\GL_N$-stable open subsets $U_m\subset V_m$ such that $\GL_N$ acts freely on $U_m$ and such that the codimension $c_m$ of $W_m:=V_m\setminus U_m$ in $V_m$ goes to infinity with $m$. In addition,   one assumes that the inclusions $i_m:V_m\to V_{m+1}$ satisfy $i_m^{-1}(U_{m+1})=U_m$ and that the $i_m$ are split by   $\GL_N$-equivariant linear projections $\pi_m:V_{m+1}\to V_m$ with $W_{m+1}\subset p_m^{-1}(W_m)$. Setting $X_m:=G\backslash U_m$, one then defines $B_\gm G$ as the colimit (in the category of Nisnevich sheaves on $\Sm_k$)
\[
B_\gm G:=\colim_m X_m.
\]
This is the same as the Nisnevich sheaf represented by the ind-scheme $(X_m)_m$. 
We will write $BG$ for $B_\gm G$; we will only use this construction for $G=\GL_n, \SL_n$ and products of these groups. 

By  \cite[Lemma 2.5]{MorelVoevodsky}, the corresponding object  $BG\in\sH(k)$ is independent of the various choices.  One such choice (given $N$) is to take $V_m\cong \A^{N\cdot (N+m)}$ to be the space of $N\times N+m$ matrices  with $\GL_N$ acting by left multiplication, and $U_m\subset V_m$ to be the matrices of rank $N$. The inclusion $V_m\to V_{m+1}$ is given by adding a 0 column at the right, and the projection is the projection to the first $m+N$ columns.  In this case $W_m:=V_m\setminus U_m$ has codimension $c_m=m+1$. 

We first concentrate on the two cases $G=\GL_n, \SL_n$.  For $G=\GL_n$ with the identity embedding, this choice yields $X_n=\Gr(n, m+n)$ and for $G=\SL_n$ with the standard inclusion $\SL_n\subset \GL_n$, we have  $X_n=\tilde{\Gr}(n,m+n)$, the $\G_m$-bundle   $\det E_{n,m+n}\setminus 0_{\Gr(n,m+n)}$  over $\Gr(n, m+n)$, where $E_{n,m+n}\to \Gr(n, m+n)$ is the tautological rank $n$ vector bundle. For $G$ a product, $G=\prod_{i=1}^r\GL_{n_i}\times \prod_{j=1}^s\SL_{m_j}$, we will use the product of these choices, so $X_m=\prod_{i=1}^r\Gr(n_i, n_i+m)\times \prod_{j=1}^s\tilde{\Gr}(m_j, m_j+m)$.

By a vector bundle on $BG$, we mean a choice of vector bundles $E_m\to X_m$ for each $m$ together with isomorphisms $\psi_m:E_{m+1|X_m}\xrightarrow{\sim} E_m$; we similarly define $\G_m$-bundles, etc. Since $\Pic(U_m)$ is trivial for all $m$, one has a canonical isomorphism of $\Pic(X_m)$ with the group of characters of $G$, and this isomorphism is compatible with  the closed immersions $X_m\hookrightarrow X_{m+1}$, so the group $\Pic(BG)$ of isomorphism classes of line bundles on $BG$ is isomorphic to the group of characters of $G$.   For example the system of vector bundles $(E_{n,m+n}\to \Gr(n, m+n))_m$ defines the {\em tautological} vector bundle $E_n\to \BGL_n$  and the pullbacks $(\tilde{E}_{n,m+n}\to \tilde{\Gr}(n, m+n))_m$ defines the tautological vector bundle $\tilde{E}_n\to \BSL_n$. Similarly, the line bundles $(\det E_{n,m+n}\to \Gr(n, m+n))_m$ define the line bundle $\det E_n\to \BGL_n$, which is a generator of $\Pic(\BGL_n)\cong \Z$. For $G=\prod_{i=1}^r \GL_{n_i}\times\prod_{j=1}^s\SL_{m_j}$, $\Pic(BG)\cong \Z^r$, with basis the pullback of the line bundles $\det E_{n_i}\to \BGL_{n_i}$, $i=1,\ldots, r$.

\begin{lemma}\label{lem:HMConst} Let $\rho:G\to \GL_N$ be a faithful representation. Let $M_*$ be a homotopy module, let $V_m, U_m, W_m$ be as chosen above and let $c_m$ denote the codimension of $W_m$ in $V_m$. Let $\chi:G\to \G_m$ be a character and let $L\to  BG$ be the corresponding line bundle Let $X_m=G\backslash U_m$ with inclusion $i_m:X_m\to X_{m+1}$ and let $L_m\to X_m$ be the line bundle corresponding to $\chi$. Then for $c_m>p+1$, and $q\in \Z$ arbitrary, the restriction map $i_m^*:H^p(X_{m+1}, M_q(L_{m+1}))\to H^p(X_m, M_q(L_m))$ is an isomorphism.
\end{lemma}

\begin{proof} Let $U_{m+1}'=V_{m+1}\setminus \pi_m^{-1}(W_m)\cong U_m\times \ker \pi_m$. The projection $\pi_m:U_{m+1}'\to U_m$ realizes $U_{m+1}'$ as a $G$-vector bundle over $U_m$ with 0-section $i'_m$ induced by $i_m$. Thus, letting $X'_{m+1}:=G\backslash U_{m+1}'$,  $\pi_m$ induces a projection  $p_m:X'_{m+1}\to X_m$ making $X'_{m+1}$ a vector bundle over $X_m$. Since $M_q(L_m)$ is strictly $\A^1$-invariant, we have the isomorphism 
\[
p_m^*:H^p(X_m, M_q(L_m))\to H^p(X_{m+1}', M_q(p_m^*L_m))
\]
inverse to $i_m^*:H^p(X_{m+1}', M_q(p_m^*L_m))\to H^p(X_m, M_q(L_m))$.

We have the open immersion $j:X_{m+1}'\to X_{m+1}$ with complement $Y_{m+1}:=G\backslash(\pi_m^{-1}(W_m)\setminus W_{m+1}$ and a canonical isomorphism $j^*L_{m+1}\cong p_m^*L_m$. Since $W_m$ has codimension $c_m$ in $V_m$, $Y_{m+1}$ has codimension $c_m$ in $X_{m+1}$. By \eqref{eqn:CohVanishing}, $H^j_{Y_{m+1}}(X_{m+1}, M_q(L_{m+1}))=0$ for $c_m>j$, and thus 
\[
j^*:H^p(X_{m+1}, M_q(L_{m+1}))\to H^p(X'_{m+1}, M_q(p_m^*L_m))
\]
is an isomorphism for $c_m>p+1$. Since $i_m=j\circ i_m'$, we see that $i_m^*:H^p(X_{m+1}, M_q(L_{m+1}))\to H^p(X_m, M_q(L_m))$ is an isomorphism for $c_m>p+1$.
\end{proof}

\begin{proposition}\label{lem:HtpyModCoh} Let $M_*$ be a homotopy module and let $L\to BG$ be a line bundle. The the pro-system $(H^{a-b}(X_m, M_q(L)))_m$ is eventually constant. Moreover, the canonical map
\[
\EM(M_*(L))^{a,b}(BG)\to \lim_\leftarrow H^{a-b}(X_m, M_b(L))
\]
is an isomorphism and the restriction map
\[
\EM(M_*(L))^{a,b}(BG)\to   H^{a-b}(X_m, M_b(L)) 
\]
is an isomorphism for all $m$ sufficiently large.
\end{proposition}

\begin{proof} The first assertion is just a rephrasing of Lemma~\ref{lem:HMConst}. For the second, we have the Milnor sequence
\begin{multline*}
0\to R^1\lim H^{a-b-1}(X_m, M_b(L))\to \EM(M_q(L))^{a,b}(BG)\\\to  \lim_\leftarrow H^{a-b}(X_m, M_q(L))\to0.
\end{multline*}
Since the system $(H^{a-b-1}(X_m, M_q(L)))_m$ is eventually constant, the $R^1\lim $ vanishes.
\end{proof}
We will write $ H^{a}(BG, M_b(L))$ for $\EM(M_*(L))^{a+b,b}(BG)$.

\begin{lemma}\label{lem:WittKuenneth} Let $Y$ be in $\Sm_k$, let $X$ be a  smooth $k$-scheme, $M\in \Pic(X)$, $L\in \Pic(Y)$ and fix an integer $N$. Suppose that for each finitely generated field extension $F$ of $k$, and for $n\le N$,   
\[
H^n(X_F, \sW(M))\cong W(F)\otimes_{W(k)}H^n(X, \sW(M)), 
\]
with map induced by the pullback map $H^n(X, \sW(M))\to H^n(X_F, \sW(M))$. Suppose in addition that  $H^n(X, \sW(M))$ is a projective $W(k)$-module.   Then the external product map
\[
\oplus_{a+b=n}H^a(Y, \sW(L))\otimes_{W(k)}H^b(X, \sW(M))\to
H^n(Y\times_kX, \sW(p_1^*L\otimes p_2^*M))
\]
is an isomorphism for all $n< N$. 
\end{lemma}

\begin{proof}  Let $Y^{(p )}$ be the set of codimension $p$  points of $Y$. We take the ind-stratification of $Y\times_kX$ with codimension $p$ stratum $Y^{(p )}\times_k X$.  Gluing together the corresponding sequences of cohomology with support gives the spectral sequence
\begin{multline*}
E_1^{p,q}:=H^{p+q}_{Y^{(p )}\times_kX}(Y\setminus Y^{(p+1)}\times_k X, \sW(p_1^*L\otimes p_2^*M))\\
\Rightarrow H^{p+q}(Y\times_k X, \sW(p_1^*L\otimes p_2^*M)).
\end{multline*}
By purity   we have the isomorphisms
\[
E_1^{p,q}\cong \oplus_{y\in Y^{(p )}}H^{q}( k(y) \times_kX, \sW(p_1^*(L\otimes \det^{-1}\mathfrak{m}_y/\mathfrak{m}_y^2)\otimes p_2^*M)),
\]
where $\mathfrak{m}_y\subset \sO_{Y,y}$ is the maximal ideal.  Since $L\otimes \det^{-1}\mathfrak{m}_y/\mathfrak{m}_y^2$ is (non-canonically) isomorphic to $k(y)$, our assumption on $X$ implies that 
 the external product gives a canonical isomorphism for $q\le N$
\begin{multline*}
E_1^{p,q}\cong \sW(k(y), L\otimes \det^{-1}\mathfrak{m}_y/\mathfrak{m}_y^2\otimes)\otimes_{W(k)} 
H^{q}(X, \sW(M))\\\xrightarrow{\sim}
H^{q}(k(y)\times_k X, \sW(p_1^*(L\otimes \det^{-1}\mathfrak{m}_y/\mathfrak{m}_y^2)\otimes p_2^*M)).
\end{multline*}

We have the analogous spectral sequence
\[
E_1^{p,q}(Y):= H^{p+q}_{Y^{(p)}}((Y\setminus Y^{(p+1)}, \sW(L))
\Rightarrow H^{p+q}(Y, \sW(L))
\]
Since $H^n(X, \sW(L))$ is a projective $W(k)$-module for $n\le N$, this spectral sequence gives rise to a truncated spectral sequence
\begin{multline*}
\tilde{E}_1^{p,q}:=[H^*_{Y^{(p)}}((Y\setminus Y^{(p+1)}), \sW(L))\otimes_{W(k)}H^*(X, \sW(M))]^{\prime p+q}\\
\Rightarrow [H^*(Y, \sW(L))\otimes_{W(k)}H^*(X, \sW(M))]^{\prime\prime p+q},\quad p+q\le N
\end{multline*}
where the ${}'$ means that the term is zero if $p+q>N$ and is the same as without the ${}'$ if $p+q\le N$. The ${}''$ means  the same as the ${}'$ if $p+q\neq N$ and we ignore what the spectral sequence converges to for $p+q= N$.  

Pull-back by the  projection $p_1:Y\times_kX\to Y$ together with the cup product action via $p_2^*$ of  $H^*(X, \sW(M))$ on $H^*_{F\times X}(Y\setminus G\times X, \sW(p_1^*L\otimes p_2^*M))$ for $G\subset Y$ and $F\subset Y\setminus G$ both closed gives a map of spectral sequences 
\[
p_1^*\cup p_2^*:\tilde{E}_*^{*,*}\to E^{\prime*,*}_*
\]
which is an isomorphism on the $E_1$-terms, where we truncate the terms in  spectral sequence $ E^{*,*}_*$ to be zero if $p+q>N$ to form the spectral sequence $E^{\prime*,*}_*$. Since both spectral sequences are strongly convergent in total degree $<N$, we see that the external product 
\[
p_1^*\cup p_2^*:H^*(Y, \sW(L))\otimes_{W(k)}H^*(X, \sW(M))\to
H^*(Y\times_kX, \sW(p_1^*L\otimes p_2^*M)) 
\]
is an isomorphism in total degree $<N$.
\end{proof}

\begin{proposition}\label{prop:GrKuenneth} Let $X=\prod_{i=1}^r\BGL_{n_i}\times \prod_{i=r+1}^{r+s}\BSL_{n_i}$. \\[5pt]
1. For each $p$, $\CH^p(X):=H^p(X, \sK^M_p)$ is a free, finitely generated abelian group. 
\\[2pt]
2. Let $L\to X$ be a line bundle. Then for each $p$ $H^p(X, \sW(L))$ is a finitely generated free  $W(k)$-module.\\[5pt]
Moreover, we have
\[
\CH^*(X)=\otimes_{\Z, i=1}^r\CH^*(\BGL_{n_i})\otimes_\Z\otimes_{\Z, i=r+1}^{r+s}\CH^*(\BSL_{n_i})
\]
and for $L=L_1\boxtimes\ldots\boxtimes L_{r+s}$, we have 
\begin{multline*}
H^*(X, \sW(L))\\=\otimes_{W(k), i=1}^rH^*(\BGL_{n_i}, \sW(L_i))\otimes_{W(k)}\otimes_{W(k), i=r+1}^{r+s}H^*(\BSL_{n_i}, \sW(L_i)).
\end{multline*}

\end{proposition}

\begin{proof} It is well-known that $\Gr(n,N)$ is a cellular variety and that
\[
\CH^*(\Gr(n, N))=\Z[c_1,\ldots, c_n]/(I_{n,N})
\]
where $c_i=c_i(E_{n, N})$, $E_{n, N}\to \Gr(n, N)$ the tautological rank $n$ vector bundle and
the ideal $I_{n,N}$ is homogeneous with generators in degree $>N-n$ (where we give $c_i$ degree $i$). As $c_1(\det E_{n,N})=c_1(E_{n,N})$, the localization sequence for the open immersion $\tilde{\Gr}(n, N)\to \det E_{n,N}$ gives
\[
\CH^*(\Gr(n, N))=\Z[c_1,\ldots, c_n]/(I_{n,N}, c_1)\cong \Z[c_2,\ldots, c_n]/(J_{n,N})
\]
with $J_{n,N}$ homogeneous with generators again in degree $>N-n$. By Lemma~\ref{lem:HtpyModCoh}, this shows that
\[
\CH^*(\BGL_n)=\Z[c_1,\ldots, c_n],\ \CH^*(\BSL_n)=\Z[c_2,\ldots, c_n].
\]

Since cellular varieties satisfy the K\"unneth formula for the Chow groups, a similar argument shows that
\begin{multline*}
\CH^*(\prod_{i=1}^r \Gr(n_i, N)\times\prod_{i=r+1}^{r+s}\tilde{\Gr}(n_i, N))\\=\otimes_{i=1}^r\CH^*(\Gr(n_i, N))\otimes_\Z \otimes_{i=r+1}^{r+s}\CH^*(\tilde{\Gr}(n_i, N))
\end{multline*}
with the tensor products over $\Z$. Thus by  Lemma~\ref{lem:HtpyModCoh}
\[
\CH^*(X)=\otimes_{i=1}^r\CH^*(\BGL_{n_i})\otimes_\Z\otimes_{i=r+1}^{r+s}\CH^*(\BSL_{n_i}),
\]
which proves (1).
 
 For (2),  Ananyevskiy \cite[Introduction]{Anan} computes 
\[
H^*(\BSL_n, \sW)
=\begin{cases} 
W(k)[p_1,\ldots, p_{n/2},e]/(p_{n/2}-e^2)& \text{for $n$ even}\\
W(k)[p_1,\ldots, p_{n-1/2}]&\text{for  $n$ odd}
\end{cases}
\]
In \cite[Theorem 4.1]{LevBeckerGottlieb}, we have shown that the pullback by the projection $\BSL_n\to\BGL_n$ induces an inclusion $H^*(\BGL_n, \sW(L))\subset H^*(\BSL_n, \sW)$ with image given by
\[
H^*(\BGL_n, \sW(L))
=\begin{cases} W(k)[p_1,\ldots, p_{[n/2]}]&\text{for }L=\sO_{\BGL_n}\\
e\cdot W(k)[p_1,\ldots, p_{n/2}]&\text{for $n$ even and }L=\det E_n\\
0&\text{for  $n$ odd and }L=\det E_n.
\end{cases}
\]
Applying Lemma~\ref{lem:HMConst}, and then using Lemma~\ref{lem:WittKuenneth} for the finite dimensional approximations to $\prod_{i=1}^r\BGL_{n_i}\times \prod_{i=r+1}^{r+s}\BSL_{n_i}$, we see that for $L=\boxtimes_{i=1}^r L_i$ we have 
\[
H^*(X, \sW(L))\cong \otimes_{i=1}^r H^*(\BGL_{n_i}, \sW(L_i))\otimes_{W(k)}\otimes_{i=r+1}^{r+s} H^*(\BSL_{n_i}, \sW)
\]
where all tensor products are over $W(k)$. This together with our description of $H^*(\BGL_n, \sW(L))$ and $H^*(\BSL_n, \sW)$ proves (2). 
\end{proof}

\begin{remark} The last two results for the product scheme $\prod_{i=1}^r\Gr(n_i, N_i)\times \prod_{j=1}^s \tilde{\Gr}(m_j, M_j)$ and the product ind-scheme $\prod_{i=1}^r\BGL_{n_i}\times \prod_{j=1}^s \BSL_{m_j}$ certainly hold more generally, but as we only need them in these cases, we have refrained from formulating our results in greater generality.
\end{remark}

\section{Decomposing the Chow-Witt Euler class}\label{sec:Decomp}

In this section we show that in universal cases,   the Milnor-Witt Euler class is determined by the associated top Chern class together with the Euler class in $\sW$-cohomology. Wendt \cite{Wendt} and Hornbostel-Wendt \cite{HornbostelWendt}) give a detailed description of the Milnor-Witt Chow groups of Grassmannians, which forms an essential part of the argument; we recall some aspects of this treatment here.

As mentioned at the end of \S~\ref{sec:SLOrient}, we have the sheaf $\sI\subset \sGW$ of augmentation ideals, that is, the kernel of the rank homomorphism $\rnk:\sGW\to \Z$.  We have the powers $\sI^m$ and the twisted version $\sI^m(L)\subset \sGW(L)$ fitting into an exact sequence
\[
0\to \sI^{m+1}(L)\to \sK^{MW}_m(L)\xrightarrow{\pi_m}\sK^M_m\to 0,
\]
where $\sI^m(L)=\sW(L)$ and $\sK^M_m=0$ for $m<0$. In addition, the graded subsheaf $\sI^{*+1}\subset \sK^{MW}_*$ forms a sub-homotopy module of $ \sK^{MW}_*$. Thus, we may apply Lemma~\ref{lem:HMConst} and Proposition~\ref{lem:HtpyModCoh} to  define $H^*(BG, \sI^m(L))$.

 Multiplication by $\eta\in K^{MW}_{-1}(k)$ induces the map $\times\eta:\sK^{MW}_m(L)\to\sK^{MW}_{m-1}(L)$  and the corresponding colimit is isomorphic to $\sW(L)$. This gives us the map $\phi_m:\sK^{MW}_m(L)\to \sW(L)$. 

\begin{proposition}\label{prop:Reduction}  Let $m$ be a non-negative integer. Let 
\[
X=\prod_{i=1}^r\BGL_{n_i}\times\prod_{j=1}^s\BSL_{m_i}, 
\]
and take $L\in \Pic(X)$. The map
\[
H^m(X, \sK^{MW}_m(L))\xrightarrow{(\phi_m, \pi_m)} H^m(X, \sW(L))\times \CH^m(X)
\]
is injective.
\end{proposition}

\begin{proof}
 Morel \cite[Th\'eor\`eme 5.3]{MorelPI}  (see also \cite[proof of Proposition 2.3.1]{AsokFaselEuler}) shows that $\sK^{MW}_m(L)$ fits into  a fiber diagram
\[
\xymatrix{
 \sK^{MW}_m(L)\ar[d]_{p_m}\ar[r]^{\pi_m}&\sK^M_m\ar[d]^{red}\\
 \sI^m(L)\ar[r]_{\rho_m}&\sK^M_m/2
 }
 \]
and multiplication by $\eta$ induces a commutative diagram
\[
\xymatrix{
 \sK^{MW}_m(L)\ar[d]_{p_m}\ar[r]^{\times\eta}&\sK^{MW}_{m-1}\ar[d]^{p_{m-1}}\\
 \sI^m(L)\ar[r] &\sI^{m-1}(L),
 }
 \]
 where $ \sI^m(L)\to \sI^{m-1}(L)$ is the inclusion for $m\ge1$, the canonical surjection $\sGW(L)=\sK^{MW}_0(L)\to \sW(L)$ for $m=0$ and the identity map $\sW(L)\to \sW(L)$ for $m<0$. This identifies $\sW(L)$ with $\colim_{m\to -\infty} \sI^m(L)$, giving the map $\tilde\phi_m: \sI^m(L)\to \sW(L)$, and factors $\phi_m:H^m(X, \sK_m^{MW}(L))\to H^m(X, \sW(L))$ as
 \[
H^m(X, \sK_m^{MW}(L))\xrightarrow{p_m} H^m(X, \sI^m(L))\xrightarrow{\tilde{\phi}_m}
H^m(X, \sW(L)).
\]

By Proposition~\ref{prop:GrKuenneth}, $H^m(X, \sW(L))$ is  a free $\sW(k)$-module.  We apply \cite[Lemma 2.3]{Wendt}, which shows that the map $\rho_m:H^m(X, \sI^m(L))\to H^m(X,\sK^M_m/2)$ is injective on the kernel of $\tilde{\phi}_m:H^m(X, \sI^m(L))\to H^m(X, \sW(L))$\footnote{Although the results of Wendt and Wendt-Hornbostel used here and throughout the proof are for smooth $k$-schemes rather than the ind-smooth scheme $X$, we may apply these results to $X$ by using the approximation result Proposition~\ref{lem:HtpyModCoh}}. Following the remark of Wendt \cite[\S 2.3]{Wendt}, a twisted version of \cite[Proposition 2.11]{HornbostelWendt} and the fact that  $\CH^m(X)$ has no non-trivial 2-torsion implies that the  map
\[
(p_m, \pi_m):H^m(X, \sK^{MW}_m(L))\to H^m(X, \sI^m(L))\times H^m(X, \sK^M_m)
\]
is injective. Since ${red}\circ \pi_m=\rho_m\circ p_m$, it follows that 
\[
(\phi_m, \pi_m):H^m(X, \sK^{MW}_m(L))\to H^m(X, \sW(L))\times H^m(X, \sK^M_m)
\]
is injective as well. Noting that $H^m(X, \sK^M_m)=\CH^m(X)$ finishes the proof.
 \end{proof}
 
We recall Ananyevskiy's $\SL_2$ splitting principle.

\begin{theorem}[Ananyevskiy \hbox{\cite{Anan}}]\label{thm:ASplit} Let $\sA\in \SH(k)$ be an $\SL$-oriented ring spectrum such that $\times\eta$ acts invertibly on $\sA^{*,*}(k)$. Let $\iota_n:\SL_2\times\ldots\times \SL_2\to \SL_{2n}$ be the block-diagonal embedding (with $n$ copies of $\SL_2$). Then the induced map
\[
\iota_n^*:\sA^{**}(\BSL_{2n})\to  \sA^{**}(\BSL_2\times\ldots\times \BSL_2)
\]
is injective. Moreover
\[
\sA^{**}(\BSL_2\times\ldots\times \BSL_2)\cong \sA^{**}(\BSL_2)^{\otimes n}
\]
where the tensor product is over $\sA^{**}(k)$.
\end{theorem}
This is not stated as such in \cite{Anan}, but follows directly from \cite[Theorem 6, Theorem 10]{Anan}. Using Proposition~\ref{prop:Reduction}, we can refine this to a $\GL_2$-splitting principle for Milnor-Witt cohomology

\begin{theorem}\label{thm:GL2SplittingPrinc}  Let $n$ be a positive integer. For $n=2m$ even, consider the block-diagonal embedding
\[
\iota_n:G_n:=(\GL_2)^m\to \GL_n
\]
and for $n=2m+1$ odd, consider the 
the block-diagonal embedding
\[
\iota_n:G_n:=(\GL_2)^m\times\GL_1\to \GL_n
\]
Then for $L\in \Pic \BGL_n$, $B\iota_n$ induces an injection
\[
B\iota_n^*:H^p(\BGL_n, \sK^{MW}_p(L))\to 
H^p(BG_n, \sK^{MW}_p(\iota_n^*L)).
\]
Moreover the map  $\BSL_2^{[n/2]}\to \BGL_n$ induced by the block-diagonal map 
$(\SL_2)^m\to \GL_n$ for $n=2m$ and $(\SL_2)^m\times\id\to \GL_n$ for $n=2m+1$ induces an 
 injection 
\[
H^p(\BGL_n, \sW(L))\to H^p((\SL_2)^{[n/2]}, \sW)
\]
\end{theorem}

\begin{proof} By Proposition~\ref{prop:Reduction}, we have the injection
\[
H^p(\BGL_n, \sK^{MW}_*(L))\to H^p(\BGL_n, \sW(L))\times \CH^p(\BGL_n).
\]
 \cite[Theorem 4.1]{LevBeckerGottlieb} and the fact that $\BSL_1=\{*\}$ implies that 
 $H^p(\BGL_1, \sW(L))=0$ for $p>0$ and all $L$, $H^0(\BGL_1, \sW(\det E_1))=0$ and 
 $H^0(\BGL_1, \sW)=W(k)$. Similarly,  \cite[Theorem 4.1]{LevBeckerGottlieb}, together with Ananyeskiy's $\SL_2$ splitting principle gives the injection  
\[
H^p(\BGL_n, \sW(L))\to H^p((\BSL_2)^{[n/2]}, \sW)
\]
factoring as
\[
H^p(\BGL_n, \sW(L))\xrightarrow{B\iota_n^*}H^p(BG_n, \sW(\iota_n^*L))\to   H^p((\BSL_2)^{[n/2]}, \sW),
\]
with the second map induced by the inclusion $\SL_2\to \GL_2$ (and $\{1\}\to \GL_1$ if $n$ is odd). 
The classical splitting principle shows that $B\iota_n^*:\CH^p(\BGL_n)\to \CH^p(BG_n)$ is injective, which completes the proof.
\end{proof} 

\begin{corollary}\label{cor:GLSplitting} Let $n$ be a positive integer, $L\in \Pic\BGL_n$. Then the map
\[
H^p(\BGL_n, \sK^{MW}_p(L))\to [H^*(\BSL_2, \sW)^{\otimes [n/2]}]^p\times [\CH^*(\P^\infty)^{\otimes n}]^p
\]
induced by the diagonal embeddings $(\SL_2)^{[n/2]}\to \GL_n$, $\GL_1^n\to \GL_n$  and the maps $\phi_2, \pi_1$ is injective. Here the first tensor product is over $W(k)$ and the second is over $\Z$.
\end{corollary}

\begin{proof} This follows from Proposition~\ref{prop:Reduction} and Theorem~\ref{thm:GL2SplittingPrinc}, together with  Theorem~\ref{thm:ASplit}  and the classical splitting principle for the Chow groups.
\end{proof}

\begin{remark} We recall that $H^*(\BSL_2, \sW)=W(k)[e^\sW(E_2)]$ and $\CH^*(\P^\infty)=\Z[c_1(E_1)]$. 
\end{remark}

\section{Dual bundles}\label{sec:DualBundle}

The main result of this section is the comparison of the Chow-Witt Euler classes for $E$ and the dual $E^\vee$.

\begin{theorem}\label{thm:Dual} Let $X$ be a smooth quasi-projective scheme over $k$, $E$ a rank $n$ vector bundle on $X$ and $E^\vee$ the dual bundle. Let 
\[
\psi:H^n(X, \sK^{MW}_n(\det^{-1}E^\vee))\to H^n(X, \sK^{MW}_n(\det^{-1}E))
\]
be the isomorphism induced by  $\psi_{\det^{-1}E}$ \eqref{eqn:SquareTwistIso}.  Then 
\[
\psi(e^\CW(E^\vee))=(-1)^ne^\CW(E)
\]
in $H^n(X, \sK^{MW}_n(\det^{-1}E))$.
\end{theorem}

In the forerunner \cite{LevQuad} of this article, we proved this result by identifying the Euler class as an obstruction class. 

\begin{proof} To simplify the notation, we drop the mentions of the isomorphism $\psi$. For $X$ a smooth quasi-projective $k$-scheme and $E\to X$ a rank $n$ bundle, we have a Jouanolou cover $p:\tilde{X}\to X$ of $X$,  an affine space bundle over $X$ with $\tilde{X}$ affine.  Since an affine space bundle is locally trivial in the Zariski topology (Hilbert theorem 90)  and  Milnor-Witt cohomology is $\A^1$-homotopy invariant \cite[Theorem 3.37]{MorelA1}, the Mayer-Vietoris property for Zariski cohomology implies that $p$ induces an isomorphism $H^n(X, \sK^{MW}_n(\det^{-1}E))\to H^n(\tilde{X}, \sK^{MW}_n(\det^{-1}p^*E))$. Thus it suffices to prove the result for $X$ affine, in which case   $E^\vee$ is globally generated and is thus $E$ is the pull-back of the tautological sub-bundle $E_n\to \BGL_n$ for some morphism $f:X\to \Gr(n, n+m)\subset \BGL_n$. Thus we need only show that $e^\CW(E_n^\vee)=(-1)^ne^\CW(E_n)$.

 The map of sheaves $\phi_n:\sK^{MW}_n\to\sW$ arises from a map of $\SL$-oriented theories, and in particular is compatible with the respective Euler classes of vector bundles; the same holds for $\pi_n:\sK^{MW}_n\to \sK^M_n$. Passing to $H^n(-, \sK^M_n)=\CH^n(-)$, the Euler class of $E_n$ is the top Chern class $c_n(E_n)\in \CH^n(\BGL_n)$, and it follows easily from the splitting principle that $c_n(E_n^\vee)=(-1)^nc_n(E_n)$, that is, $\pi_n(e^\CW(E_n^\vee))=\pi_n((-1)^ne^\CW(E_n))\in \CH^n(\BGL_n)$.

We have the Euler class $e^\sW(E_n)\in H^n(\BGL_n\sW(\det^{-1}(E_n)))$. In case $n$ is odd, it follows from Lemma~\ref{lemma:OddEulerClass} and the identity $\sW=\sK^{MW}_*[\eta^{-1}]$ that $e^\sW(E_n)=0=e^\sW(E_n^\vee)$ in $H^n(\BGL_n,\sW(\det^{-1}(E_n)))$.

 If $n=2m$ is even, we use Corollary~\ref{cor:GLSplitting}. By the naturality of the Euler classes, the image of $e^\CW(E_n)$ under the map
 \[
H^n(\BGL_n, \sK^{MW}_n(\det^{-1}(E_n)))\to [H^*(\BSL_2, \sW)^{\otimes_{W(k)} m}]^n\times [\CH^*(\BGL_2)^{\otimes m}]^n
\]
is $(e^\sW(\tilde{E}_2)^{\otimes m}, c_2(E_2)^{\otimes m})$, and similarly for the image of $e^\CW(E_n^\vee)$. This reduces us to the case  of a rank 2 bundle with trivialized determinant.

For $V\to X$ of rank 2, we have the canonical isomorphism $V\cong V^\vee\otimes \det V$ induced by the perfect pairing $V\times V\to \det V$, inducing the identity
\[
e^\sW(V)=e^\sW(V^\vee\otimes \det V)\in H^2(X, \sW(\det^{-1}V))
\]
An isomorphism $\rho:\det V\to \sO_X$ induces the isomorphism $\id_{V^\vee}\otimes\rho:
V^\vee\otimes \det V\to V^\vee$, giving the identity
\[
\rho_*(e^\sW(V^\vee\otimes \det V))=e^\sW(V^\vee)\in H^2(X, \sW(\det^{-1}V^\vee))
\]
The map $\rho_*$ involves the isomorphism $\det\rho^{-1}\circ \det\rho^\vee: \sW(\det^{-1}V)\cong \sW(\det^{-1}V^\vee))$. A direct computation shows that  $\det\rho^{-1}\circ \det\rho^\vee=\psi_{\det V}$, so
\[
\psi_{\det^{-1} V}(e^\sW(V^\vee))=e^\sW(V).
\]

Thus, we have the identities $\pi_n(e^\CW(E_n^\vee))=(-1)^n\pi_n(e^\CW(E_n))$ in $\CH^n(\BGL_n)$ and
$\phi_n(e^\CW(E_n^\vee))=(-1)^n \phi_n(e^\CW(E_n))$ in $H^n(\BGL_n, \sW(\det^{-1}(E_n)))$; Proposition~\ref{prop:Reduction} completes the proof.
\end{proof}

The corresponding identity for hermitian $K$-theory is  also valid and the proof is considerably easier. For $L, M\to X$ line bundles, the exact functor $-\otimes_{\sO_X} M$ on  complexes of $\sO_X$-modules induces the map
\[
\psi_M:\KO^{a,b}(X,L)\to  \KO^{a,b}(X,L\otimes M^{\otimes 2})
\]

\begin{theorem}\label{thm:KODual} Let $X$ be a smooth quasi-projective scheme over $k$, $E$ a rank $n$ vector bundle on $X$ and $E^\vee$ the dual bundle.   Then 
\[
\psi_{\det^{-1}E}(e^\KO(E^\vee))=(-1)^ne^\KO(E) 
\]
in $\KO^{2n,n}(X, \det^{-1}E)$.
\end{theorem} 

\begin{proof} We have the hyperbolic maps
\[
h_{\det^{-1}E, n}:K_0(X)=\KGL^{2n,n}(X)\to \KO^{2n,n}(X, \det^{-1}E)
\]
and
\[
h_{\det E, n}:K_0(X)=\KGL^{2n,n}(X)\to \KO^{2n,n}(X, \det E)
\]
with $h_{\det^{-1}E, n}(V)=(V\oplus (V^\vee\otimes\det^{-1}E[n])), h(\can_{\det^{-1}E[n]}))$, 
$h_{\det E, n}(V)=(V\oplus  (V^\vee\otimes\det E[n])),h(\can_{\det E[n]}))$, where $can_{L[n]}$ is the canonical pairing $V\times (V^\vee\otimes L[n])\to L[n]$ and 
\[
h(\can_{L[n]})=\begin{pmatrix}0&\can_{L[n]}\\\can_{L[n]}&0\end{pmatrix}. 
\]
Explicitly
\begin{multline*}
 \psi_{\det^{-1}E}((-1)^i(h_{\det E, n}(\Lambda^iE))=\psi_{\det^{-1}E}(h_{\det E, n}(\Lambda^iE[i]))\\= h_{\det^{-1} E,n}(\Lambda^iE\otimes \det^{-1}E[i])\\\cong
h_{\det^{-1} E,n}(\Lambda^iE^\vee[n-i])=(-1)^{n+i}h_{\det^{-1} E,n}(\Lambda^iE^\vee).
\end{multline*}
the isometry arising from exchanging the order of the summands $\Lambda^iE\otimes \det^{-1}E[i]=(\Lambda^iE^\vee[n-i])^\vee\otimes  \det^{-1}E[n]$ and $\Lambda^iE^\vee[n-i]=(\Lambda^iE\otimes \det^{-1}E[i])^\vee\otimes \det^{-1}E[n]$. 

For $n$ odd, we have
\[
e^\KO(E)=\sum_{i=0}^{n-1/2}(-1)^i h_{\det^{-1} E,n}(\Lambda^iE^\vee),\ 
e^\KO(E^\vee )=\sum_{i=0}^{n-1/2}(-1)^i h_{\det E,n}(\Lambda^iE)
\]
so $\psi_{\det^{-1}E}(e^\KO(E^\vee ))=(-1)^n e^\KO(E)$. 

For $n$ even, 
\[
e^\KO(E)=h_{\det^{-1}E,n}(\oplus_{i=0}^{n/2-1}(-1)^i\Lambda^iE^\vee) + (\Lambda^{n/2}E^\vee[n/2], q_{E,n/2})
\]
where $q_{E, n/2}$ is the restriction of $q_E$. We have a similar formula for $e^\KO(E^\vee)$. Arguing as above, we have 
\[
\psi_{\det^{-1}E}(h_{\det E,n}(\oplus_{i=0}^{n/2-1]}(-1)^i\Lambda^iE))=
h_{\det^{-1}E,n}(\oplus_{i=0}^{n/2-1}(-1)^i\Lambda^iE^\vee), 
\]
so  we need to show that
\[
\psi_{\det^{-1}E}((\Lambda^{n/2}E[n/2], q_{E^\vee, n/2}))=(\Lambda^{n/2}E^\vee[n/2], q_{E,n/2})
\]
in $\KO^{2n,n}(X)$.    This follows from the isomorphism $\Lambda^{n/2}E^\vee\cong \Lambda^{n/2}E\otimes \det^{-1} E$ defined  by the perfect pairing $\Lambda^{n/2}E\otimes \Lambda^{n/2}E\to \det E$. 
\end{proof} 

\section{Symmetric powers and tensor products} We give a few additional applications of the  results of \S\ref{sec:Decomp}.

Recall the hyperbolic element $h\in K^{MW}_0(k)$, $h=2+[-1]\eta=\<1\>+\<-1\>$, corresponding to the hyperbolic class $H\in \GW(k)$. Let $L\to X$ be a line bundle on $X\in \Sm_k$. The relations $\eta\cdot h=0$ and $\sK^M_n=\sK^{MW}_n(L)/\eta\cdot \sK^{MW}_{n+1}(L)$ show that the map $\times h:\sK^{MW}_n(L)\to \sK^{MW}_n(L)$ descends to the {\em hyperbolic map} of sheaves on $X$
\[
\bar{h}_L:\sK^M_n\to  \sK^{MW}_n(L)
\]
and $\pi_n\circ \bar{h}_L:\sK^M_n\to  \sK^{M}_n$ is multiplication by 2.

\begin{theorem} \label{thm:EulerClassSym} Let $V\to X$ be a rank two bundle on $X\in \Sm_k$. Suppose $\Char k=0$  or  $2n$ is prime to $\Char k$. Let $L=\det^{-1}\Sym^nV$.  Then there are universal integers $B_{i,n}$, $i=0, \ldots, [\frac{n-1}{2}]$, such that 
\[
e^\CW(\Sym^n V)= \begin{cases} \bar{h}_L(c_1(V)\cdot \sum_{i=0}^n \sum_{i=0}^{{(n-2)/2}}B_{i,n}c_1(V)^{2i}c_2(V)^{(n-2i)/2})\\&\hskip-70pt\text{ for $n$ even}\\
n!!e^\CW(V)^{(n+1)/2}\\\hskip 20pt+\bar{h}_L( \sum_{i=0}^{{(n-1)/2}}B_{i,n}c_1(V)^{2i}c_2(V)^{(n-2i+1)/2})\\
&\hskip-70pt\text{ for $n$ odd.}
\end{cases}
\]
Here $n!!=n(n-2)(n-4)\cdots  3\cdot 1$ for $n$ odd.  
\end{theorem}

\begin{proof} $\Sym^n V$ has rank $n+1$; we first compute $c_{n+1}(\Sym^n V)$. Suppose $V$ has Chern roots $\xi_1, \xi_2$. Then $\Sym^n V$ has Chern roots $\{(n-i)\xi_1+i\xi_2\}_{0\le i\le n}$ and thus
\[
c_{n+1}(\Sym^n V)=\prod_{i=0}^n((n-i)\xi_1+i\xi_2)
\]
Note that 
\begin{align*}
((n-i)\xi_1+i\xi_2)(((n-i)\xi_2+i\xi_1&)=i(n-i)(\xi_1^2+\xi_2^2)+(i^2+(n-i)^2)\xi_1\xi_2\\&=
i(n-i)c_1(V)^2+(2i-n)^2c_2(V)
\end{align*}
so
\[
c_{n+1}(\Sym^n V)=\begin{cases}(n/2)c_1(V)\cdot\prod_{i=0}^{(n-2)/2}i(n-i)c_1(V)^2+(n-2i)^2c_2(V)
\\&\hskip-50pt\text{ for $n$ even}\\
\prod_{i=0}^{(n-1)/2}i(n-i)c_1(V)^2+(n-2i)^2c_2(V)\\
&\hskip-50pt\text{ for $n$ odd}
\end{cases}
\]
This gives us the universal expression 
\[
c_{n+1}(\Sym^n V))=\begin{cases} c_1(V)\cdot \sum_{i=0}^{{(n-2)/2}}A_{i,n}c_1(V)^{2i}c_2(V)^{(n-2i)/2}
&\text{ for $n$ even}\\
\sum_{i=0}^{{(n-1)/2}}A_{i,n}c_1(V)^{2i}c_2(V)^{(n+1-2i)/2}
&\text{ for $n$ odd}
\end{cases}
\]
with the $A_{i,n}\in \Z$. In case $n>0$ is even, all the $A_{i,n}$ are even, and in case $n$ is odd, all the $A_{i,n}$ except for $A_{0,n}$ are even. Let
\[
B_{i,n}=\begin{cases} (1/2)A_{i,n}&\text{ for $n$ even, or for $n$ odd and $i>0$}\\
(1/2)(A_{0,n}-n!!)&\text{ for $n$ odd and $i=0$.}
\end{cases}
\]

By \cite[Theorem 8.1]{LevBeckerGottlieb}, 
\[
e^\sW(\Sym^n V)=\begin{cases} 0
&\text{ for $n$ even}\\
n!!\cdot e^\sW(V)^{(n+1)/2}&\text{ for $n$ odd.}
\end{cases}
\]
By Proposition~\ref{prop:Reduction} and the identities 
\begin{gather*}
\pi_{n+1}\circ \bar{h}_L= 2\cdot \id,\ \pi_{n+1}(e^\CW(-))=c_{n+1}(-),\\ \phi_{n+1}\circ \bar{h}_L=0,\ \phi_{n+1}(e^\CW(-))=e^\sW(-)
\end{gather*}
the result follows in case $V$ is the universal rank 2 bundle on $\BGL_2$; the result in general follows by using a Jouanolou cover for $X$, reducing to the  case of $V$ globally generated, and then pulling back from the universal case.
\end{proof}

Our next formula is for the Euler class of $V\otimes V'$, for $V, V'$ rank two bundles. The expression for the  Euler class in $\sW$-cohomology was worked out in \cite[Proposition 9.1]{LevBeckerGottlieb}; there is a perhaps surprising asymmetry in the formula, which we should explain.

Let $V, V'$ be rank two bundles on some $X\in \Sm_k$. We have the universal isomorphism
\[
\rho_{V,V'}: \det^{-1}(V\otimes V')\xrightarrow{\sim}(\det V)^{\otimes -2}\otimes (\det V')^{\otimes -2}
\]
which in terms of local framings $e_1, e_2$ for $V$ and $f_1, f_2$ for $V'$ sends
\[
[(e_1\otimes f_1)\wedge(e_2\otimes f_1)\wedge(e_1\otimes f_2)\wedge(e_2\otimes f_2)]^{-1}
\]
to 
\[
(e_1\wedge e_2)^{\otimes -2}\otimes (f_1\wedge f_2)^{\otimes -2}.
\]
Note that the diagram 
\[
\xymatrix{
 \det^{-1}(V\otimes V')\ar[d]_{\det^{-1}(\tau_{V,V'})}\ar[r]^-{\rho_{V,V'}}&(\det V)^{\otimes -2}\otimes (\det V')^{\otimes -2}\ar[d]^{\tau_{(\det V)^{\otimes -2}, (\det V')^{\otimes -2}}}\\
 \det^{-1}(V'\otimes V)\ar[r]^-{\rho_{V',V}}&(\det V)^{\otimes -2}\otimes (\det V')^{\otimes- 2}
}
\]
anti-commutes.  With this in mind, we recall the formula for $e^\sW(V\otimes V')$,
\begin{equation}\label{eqn:EulerClassTensorProd}
\rho_{V,V'*}(e^\sW(V\otimes V'))=e^\sW(V)^2 - e^\sW(V')^2\in H^4(X, \sW)
\end{equation}
For simplicity, we have omitted the canonical isomorphisms
\[
H^4(X, \sW((\det V)^{\otimes -2}\otimes (\det V')^{\otimes -2}))\cong H^4(X, \sW)
\]
and
\[
H^2(X, \sW(\det^{-2}(V))\cong H^2(X,\sW)\cong H^2(X, \sW(\det^{-2}(V'))
\]

\begin{theorem}\label{thm:TensorProd} Let $V, V'$ be rank two bundles on $X\in \Sm_k$. Then
\begin{multline}\label{eqn:TensorProduct}
\rho_{V,V'*}(e^\CW(V\otimes V'))=
e^\CW(V)^2+\<-1\>\cdot e^\CW(V')^2\\+
e^\CW(V)\cdot e^\CW(\det V')\cdot e^\CW(\det V\otimes \det V')\\+
e^\CW(V')\cdot e^\CW(\det V)\cdot e^\CW(\det V\otimes \det V')\\
-\bar{h}(c_2(V)c_2(V'))\in H^4(X, \sK^{MW}_4)
\end{multline}
\end{theorem}

\begin{proof} Let $L=(\det V)^{\otimes -2}\otimes (\det V')^{\otimes -2}$. Note that the terms
\begin{gather*}
e^\CW(V)\cdot e^\CW(\det V')\cdot e^\CW(\det V\otimes \det V'),\\ e^\CW(V')\cdot e^\CW(\det V)\cdot e^\CW(\det V\otimes \det V')
\end{gather*}
 in the right-hand side of \eqref{eqn:TensorProduct} are in 
$H^4(X, \sK^{MW}_4(L))\cong H^4(X, \sK^{MW}_4)$.  

As in the proof of Theorem~\ref{thm:EulerClassSym}, we may replace $X$ with $\BGL_2\times\BGL_2$ and take $V=p_1^*E_2$, $V'=p_2^*E_2$. We note that $e^\sW(L)=0$ for each line bundle $L$, and $\<-1\>=-1$ in $W(k)$,  so the expression on the right-hand side of the 
\eqref{eqn:TensorProduct} maps to $e^\sW(V)^2 - e^\sW(V')^2$ under the canonical map $\phi_4:H^4(X, \sK^{MW}_4)\to  H^4(X, \sW)$. By \cite[Proposition 9.1]{LevBeckerGottlieb}, the identity \eqref{eqn:TensorProduct} holds after applying $\phi_4$. 

The splitting principle gives
\begin{align*}
c_4(V\otimes V')&=c_2(V)^2+c_2(V')^2+
c_1(V)\cdot c_1(V')(c_2(V)+c_2(V'))\\
&\hskip 10pt+
c_1(V)^2\cdot c_2(V')+c_2(V)\cdot c_1(V')^2
-2c_2(V)c_2(V')\\
&=
c_2(V)^2+c_2(V')^2\\
&\hskip10pt
+c_2(V)\cdot c_1(\det V')\cdot c_1(\det V\otimes \det V')\\
&\hskip10pt+c_2(V')\cdot c_1(\det V)\cdot c_1(\det V\otimes \det V')\\
&\hskip 20pt
-2c_2(V)c_2(V')
\end{align*}
We note that  $\<-1\>$ maps to $1$ under the rank map  $\GW(k)\to \Z$, that $\pi_2(e^\CW(V))=c_2(V)$, $\pi_2(e^\CW(V'))=c_2(V')$ and  $\pi_2(e^\CW(L))=c_1(L)$ for $L$ a line bundle. Thus the identity \eqref{eqn:TensorProduct} holds after applying $\pi_4$, and then Proposition~\ref{prop:Reduction} completes the proof.
\end{proof}

\begin{remark} Using the $\GL_2$-splitting principle (Theorem~\ref{thm:GL2SplittingPrinc}) one reduces the proof of identities for the Euler classes of a functor  $\rho$ of representations applied to a sequence of bundles $V_1,\ldots, V_r$, to the case of direct sums of rank two bundles and line bundles. For instance, Theorem~\ref{thm:TensorProd} gives rise, at least in principle, to formulas for the Euler class $e^\CW$ of tensor products of bundles of arbitrary even ranks. 
\end{remark}

\section{Twisting a bundle by a line bundle}\label{sec:Twisting}
Rather than looking at the Euler class for tensor product of rank 2 bundles, as in Theorem~\ref{thm:TensorProd}, we wish to compute the Euler class of $V\otimes L$, for $L$ a line bundle and $V$ of arbitrary rank $r$.    

Here the situation is a bit more complicated. For example, there is no formula for $e^\CW(L\otimes M)$ in terms of $e^\CW(L)$ and $e^\CW(M)$ for $L, M\to X$ arbitrary line bundles. To see this, consider the universal case $X=\P^\infty\times\P^\infty$, $L=O(1,0)$, $M=O(0,1)$. Then we have
\begin{align*}
&e^\CW(L\otimes M)\in H^1(X, \sK^{MW}_1(L^{-1}\otimes M^{-1})),\\ 
&e^\CW(L)\in H^1(X, \sK^{MW}_1(L^{-1})), \\
&e^\CW(M)\in H^1(X, \sK^{MW}_1(M^{-1})). 
\end{align*}
Thus, if one wishes to express 
$e^\CW(L\otimes M)$ in terms of $e^\CW(L)$ and $e^\CW(M)$, one would need classes in $H^0(X, \sK^{MW}_0(L^{-1}))$ and $H^0(X, \sK^{MW}_0(M^{-1}))$. These groups are however both zero: By 
Proposition~\ref{prop:Reduction} and the vanishing of $H^*(\P^\infty, \sW(O(-1)))$,  $H^0(X, \sK^{MW}_0(M^{-1}))$ is a subgroup of $\CH^0(X)\cong \Z$. But restricting to $pt\times \P^\infty$ is an injective map from $\CH^0(X)$ to $\CH^0(\P^\infty)$, while Wendt's theorem \cite[Theorem 1.1]{Wendt} shows that $H^0(\P^\infty, \sK^{MW}_0(O(-1)))=0$, so the map $H^0(X, \sK^{MW}_0(M^{-1}))\to \CH^0(\P^\infty)$ is the zero map.

We will find a universal formula for $e^\CW(V\otimes L)$ if  $L\cong M^{\otimes 2}$ for some line bundle $M$ (Theorem~\ref{thm:TwistCW} ). For example, in the case of $V=L'$ a line bundle, we have the formula 
\[
e^\CW(L\otimes L')=\bar{h}_{L^{\prime-1}}(c_1(M))+e^\CW(L'), 
\]
where we use the comparison isomorphisms 
\begin{align*}
H^1(X, \sK^{MW}_1(L^{-1}\otimes L^{\prime-1}))&\cong 
H^1(X, \sK^{MW}_1(M^{-2}\otimes L^{\prime-1}))\\&\cong H^1(X, \sK^{MW}_1( L^{\prime-1}))
\end{align*}
and the hyperbolic map $\bar{h}_{L^{\prime-1}}:H^1(X, \sK^M_1)\to H^1(X, \sK^{MW}_1( L^{\prime-1}))$ to put all the classes in the same group. Passing to the first Chern classes, we recover the usual formula
\[
c_1(L\otimes L')=2c_1(M)+c_1(L').
\]
For the $\KO$-valued Euler classes, we also have a universal formula in case $V$ has even rank and $L$ is an arbitrary line bundle (Theorem~\ref{thm:TwistKO}).

We first consider the classes in Milnor-Witt cohomology and Witt cohomology.

\begin{theorem}\label{thm:TwistCW}  Let $V\to X$ be a rank $r$ vector bundle on $X\in \Sm_k$ and let $L$ be a line bundle. \\[5pt]
(1) Suppose $r=2m$ is even. Then 
\[
\psi_{L^{\otimes m}}(e^\sW(V\otimes L))=e^\sW(V) 
\]
in $H^r(X, \sW(\det^{-1}V))$.\\[2pt]
(2) Suppose we have an isomorphism $\rho:L\xrightarrow{\sim} M^{\otimes 2}$ for some line bundle $M$ on $X$. Then
\[
\psi_{M^{\otimes r}}\circ\rho^{\otimes-r}_*(e^\CW(V\otimes L))=e^\CW(V)+\bar{h}_{\det^{-1}V}(c_1(M)\cdot\sum_{i=1}^rc_{r-i}(V)\cdot c_1(L)^{i-1}).
\]
in $H^r(X, \sK^{MW}_r(\det^{-1}V))$ (see \eqref{eqn:TwistIso} for the definition of the map 
$\rho_*$ and \eqref{eqn:SquareTwistIso} for the map $\psi$).
\end{theorem}

\begin{proof} The map 
\[
\rho^{\otimes-r}_*:H^r(X, \det^{-1}(V\otimes L))\to H^r(X, \det^{-1}(V)\otimes M^{\otimes -2r}),
\]
is the map induced by the isomorphisms
\[
\det^{-1}(V\otimes L)\xrightarrow{\sim} \det^{-1}(V)\otimes L^{\otimes-r}\xrightarrow{\id\otimes\rho^{\otimes-r}}
\det^{-1}(V)\otimes M^{\otimes-2r}
\]
where the first isomorphism sends $(v_1\otimes \lambda_1)\wedge\ldots\wedge(v_r\otimes \lambda_r)$ to
 $(v_1\wedge\ldots \wedge v_r)\otimes(\lambda_1\otimes\ldots\otimes \lambda_r)$. Note that different choices of $\rho$ multiplies $\rho^{\otimes-r}_*$ by $\<u^r\>$ for some $u\in \sO_X^\times(X)$. Thus, if $r$ is even, 
$\rho^{\otimes-r}_*$ is independent of the choice of isomorphism $\rho$, and if $r$ is odd,  Lemma~\ref{lemma:OddEulerClass} implies $\rho^{\otimes-r}_*(e^\CW(V\otimes L))$ is independent of the choice of $\rho$. For this reason, we simplify the notation in the proof by assuming $L=M^{\otimes 2}$, $\rho=\id$, in case (2).
 
We first consider the universal situation, namely, on $X:=\BGL_r\times_k\P^\infty$, we consider the bundles $E_r\boxtimes \sO(1)\to X$ and $p_1^*E_r\to X$. Let $0\in \P^\infty(k)$ be the point $(1, 0,\ldots)$, giving the section $s_0:\BGL_r\to \BGL_r\times_k\P^\infty$. Since
\[
H^n(\P^\infty, \sW)=\begin{cases}W(k)&\text{ for }n=0\\0&\text{ else,}
\end{cases}
\]
(see, e.g., \cite[Theorem 4.1]{LevBeckerGottlieb} in the case $n=1$) the K\"unneth formula Proposition~\ref{prop:GrKuenneth} shows that $p_1^*:H^*(\BGL_r,\sW(M))\to H^*(\BGL_r\times_k\P^\infty, \sW(p_1^*M))$ is an isomorphism, with inverse the pull-back by $s_0^*$. In both cases (1) and (2), the comparison isomorphism $\psi_{\sO(2m)}$ restricts via $s_0*$ to the identity, hence 
 \[
\psi_{\sO(m)}(e^\sW(V\otimes \sO(1)))=e^\sW(V)
\]
if $\rnk(V)=2m$ or 
 \[
\psi_{\sO(r)}(e^\sW(V\otimes \sO(2)))=e^\sW(V)
\]
if $\rnk(V)=r$.

In general, using a Jouanolou cover we may assume that $X$ is affine. Pulling back by the classifying morphism $f=(f_V, f_L):X\to \BGL_r\times_k\P^\infty$  in case $V$ has even rank $r=2m$, or by  $f=(f_V, f_M):X\to \BGL_r\times_k\P^\infty$ in case   $L= M^{\otimes 2}$, we have
\[
\psi_{L^{\otimes m}}(e^\sW(V\otimes L))=e^\sW(V)
\]
if $V$ has rank $2m$ and
\[
\psi_{M^{\otimes r}}(e^\sW(V\otimes L))=e^\sW(V)
\]
if $L= M^{\otimes 2}$. 

This proves (1) and in case (2) Lemma~\ref{lem:HMConst}, Proposition~\ref{lem:HtpyModCoh}  and Proposition~\ref{prop:Reduction} reduce us to showing that  
\[
c_r(V\otimes L)=c_r(V)+ \pi_n\circ\bar{h}_{\det^{-1}V}(c_1(M)\cdot \sum_{i=1}^rc_1(L)^{i-1}\cdot c_{r-i}(V))
\]
Since $\pi_n\circ\bar{h}_{\det^{-1}V}$ is multiplication by 2, this follows from the formula
\[
c_r(V\otimes L)=c_r(V)+   \sum_{i=1}^rc_1(L)^i\cdot c_{r-i}(V),
\]
an easy consequence of the splitting principle.
\end{proof}

Here is the analogous result in hermitian $K$-theory. 

\begin{theorem}\label{thm:TwistKO}   Let $V\to X$ be a rank $r$ vector bundle on $X\in \Sm_k$ and let $L$ be a line bundle. \\[5pt]
(1) Suppose $r=2m$ is even. Then 
\[
\psi_{L^{\otimes m}}(e^\KO(V\otimes L))=e^\KO(V)+h_{\det^{-1}V}(\sum_{i=0}^{m-1}(-1)^i[L^{\otimes m-i}\otimes \Lambda^iV^\vee])
\]
 $\KO^{2r, r}(X, \det^{-1}V)$.\\[2pt]
(2) Suppose we have an isomorphism  $\rho:L\xrightarrow{\sim} M^{\otimes 2}$ for some line bundle $M$ on $X$. Then
\[
\psi_{M^{\otimes r}}\circ\rho^{\otimes-r}_*(e^\KO(V\otimes L))=e^\KO(V)+h_{\det^{-1}V,r}(\sum_{i=0}^{[(r-1)/2]}(-1)^i[M^{\otimes r-2i}\otimes \Lambda^iV^\vee])
\]
in $\KO^{2r, r}(X, \det^{-1}V)$.
\end{theorem}

\begin{proof} Suppose $r=2m$. Since
\[
e^\KO(V\otimes L)=(\oplus_{i=0}^r\Lambda^i(V\otimes L)^\vee[i], q_{V\otimes L})
\]
we have
\[
\psi_{L^{\otimes m}}(e^\KO(V\otimes L))=(\oplus_{i=0}^rL^{\otimes m}\otimes \Lambda^i(V\otimes L)^\vee[i],\id\otimes q_{V\otimes L})
\]
Since
\[
L^{\otimes m}\otimes \Lambda^i(V\otimes L)^\vee=
L^{\otimes m-i}\otimes \Lambda^iV^\vee,\ L^{\otimes m}\otimes \Lambda^{r-i}(V\otimes L)^\vee=
L^{\otimes i-m}\otimes  \Lambda^{r-i}V^\vee
\]
we see that 
\[
\psi_{L^{\otimes m}}(e^\KO(V\otimes L))-
e^\KO(V)=\oplus_{i=0}^{[(r-1)/2]}h_{\det^{-1}V, r}(L^{\otimes m-i}\otimes \Lambda^iV^\vee[i])
\]
which proves case (1).

For case (2), we may assume as in the proof of Theorem~\ref{thm:TwistCW} that $L=M^{\otimes 2}$ and $\rho=\id$. We have 
\[
\psi_{M^{\otimes r}}(e^\KO(V\otimes L))=(\oplus_{i=0}^rM^{\otimes r}\otimes \Lambda^i(V\otimes L)^\vee[i],\id\otimes q_{V\otimes L})
\]
and 
\[
M^{\otimes r}\otimes \Lambda^i(V\otimes L)^\vee=
M^{\otimes r-2i}\otimes \Lambda^iV^\vee,\ M^{\otimes r}\otimes \Lambda^{r-i}(V\otimes L)^\vee=
M^{\otimes 2i-r}\otimes  \Lambda^{r-i}V^\vee
\]
so 
\[
\psi_{M^{\otimes r}}(e^\KO(V\otimes L))-
e^\KO(V)=\oplus_{i=0}^{[(r-1)/2]}h_{\det^{-1}V,r}(M^{\otimes r-2i}\otimes \Lambda^iV^\vee[i])
\]
\end{proof}

\section{Quadratic Riemann-Hurwitz formulas}\label{sec:Fibering}

We consider a projective  morphism $f:Y\to X$, with $Y$ a smooth projective integral $k$-scheme, and $X$ a smooth projective curve over $k$. Kass and Wickelgren have raised the question of finding Grothendieck-Witt liftings of classical Euler characteristic formulas for such maps and have obtained formulas of this type. We give a different approach here to this problem.

\begin{proposition}\label{prop:Fibering} Let $f$ be a surjective projective  morphism $f:Y\to X$, with $Y$ a smooth projective integral $k$-scheme of dimension $r$ over $k$, and $X$ a smooth projective curve over $k$.\\[5pt]
1.  Suppose that $X$ admits a half-canonical line bundle $M$, with isomorphism $\rho:\omega_{X/k}\to M^{\otimes 2}$.\footnote{This condition   is satisfied if for instance $X=\P^1_k$ or if $X$ is a hyperelliptic curve but is not satisfied if $X$ is a conic without a rational point.  A half-canonical line bundle is often referred to as a {\em theta characteristic}, see for example \cite{Atiyah, Mumford}.}
 Then
\begin{multline*}
\psi_{M^{\otimes -r}}\circ\rho^{\otimes r}_*(e^\CW(\Omega_{Y/k}\otimes f^*\omega_{X/k}^{-1}))=
e^\CW(\Omega_{Y/k})\\+\bar{h}_{\omega^{-1}_{X/k}}(c_1(f^*M^{-1})\cup c_{r-1}(\Omega_{Y/k})),
\end{multline*}
2. Suppose that $\dim_kX=2m$ is even. Then
\[
\psi_{\omega_{X/k}^{\otimes -m}}(e^\sW(\Omega_{Y/k}\otimes f^*\omega_{X/k}^{-1}))=
e^\sW(\Omega_{Y/k}).
\]
\end{proposition}

\begin{proof} This is just a special case of Theorem~\ref{thm:TwistCW}, noting that $c_1(f^*M^{-1})\cup c_1(f^*\omega_{X/k}^{-1})=f^*(c_1(M^{-1})\cup c_1(\omega_{X/k}^{-1}))=0$ since $X$ is a curve. Note that, just as remarked in the proof of Theorem~\ref{thm:TwistCW}, the choice of isomorphism $\rho$ does not play a role.
\end{proof}

Under the assumption that $X$ admits a half-canonical line bundle $M$, we may transform 
$e(\Omega_{Y/k}\otimes f^*\omega_{X/k}^{-1})\in H^r(Y, \sK^{MW}_r(\omega_{Y/k}^{-1}\otimes f^*\omega^{\otimes r}_{X/k}))$ to an element of $H^r(Y, \sK^{MW}_r(\omega_{Y/k}))$ by applying the isomorphism $\psi_{\omega_{Y/k}\otimes f^*M^{-\otimes r}}\circ \rho^{\otimes r}_*$, and the image of $e^\CW(\Omega_{Y/k}\otimes f^*\omega_{X/k}^{-1})$ is independent of the choice of isomorphism $\rho$ and choice of $M$.  We make a similar adjustment if $r$ is even, using $\psi_{\omega_{Y/k}\otimes f^*\omega_{X/k}^{-\otimes r/2}}$. We will omit the comparison isomorphism from the notation in what follows. For instance, we have the pushforward map
\[
\pi_{Y*}:H^r(Y, \sK^{MW}_r(\omega_{Y/k}))\to H^0(\Spec k, \sK^{MW}_0)=\GW(k)
\]
which induces pushforward maps
\[
\pi_{Y*}:H^r(Y, \sK^{MW}_r(\omega^{-1}_{Y/k}))\to H^0(\Spec k, \sK^{MW}_0)=\GW(k),
\]
and, if $r$ is even or if we have an isomorphism $\rho:\omega_{X/k}\to M^{\otimes 2}$,
\[
\pi_{Y*}:H^r(Y, \sK^{MW}_r(\omega^{-1}_{Y/k}\otimes f^*\omega_{X/k}^{\otimes r}))\to H^0(\Spec k, \sK^{MW}_0)=\GW(k).
\]
The pushforward on  $H^r(Y, \sK^{MW}_r(\omega^{-1}_{Y/k}))$ is induced by the pushforward map on $H^r(Y, \sK^{MW}_r(\omega_{Y/k}))$ by composing with $\psi_{\omega_{Y/k}}$, and  the pushforward on  $H^r(Y, \sK^{MW}_r(\omega^{-1}_{Y/k}\otimes f^*\omega_{X/k}^{\otimes -r}))$ is induced by the one on $H^r(Y, \sK^{MW}_r(\omega^{-1}_{Y/k}))$  by composing with $\psi_{f^*\omega_{X/k}^{\otimes r/2}}$ if $r$ is even, or by composing with $\psi_{f^*M^{\otimes r}}\circ (f^*\rho)^{\otimes -r}_*$ if we have an isomorphism $\rho:\omega_{X/k}\xrightarrow{\sim}M^{\otimes 2}$. As noted above, the value of this last pushforward on $e^\CW(\Omega_{Y/k}\otimes f^*\omega_{X/k}^{-1})$ is independent of the choice of isomorphism $\rho$ and choice of $M$.

\begin{theorem}\label{thm:Fibering1}  Let $f$ be a projective  morphism $f:Y\to X$, with $Y$ a smooth projective integral $k$-scheme of dimension $r$ over $k$, and $X$ a smooth projective curve over $k$. Let 
\[
D(f):=(1/2)[\deg_k(c_r(\Omega_{Y/k}\otimes f^*\omega_{X/k}^{-1}))-\deg(c_r(\Omega_{Y/k}))].
\]
1. $D(f)$ is an integer. \\\\
2. Suppose $X$ admits a half-canonical line bundle or $r$ is even. Then
\[
\pi_{Y*}(e^\CW(\Omega_{Y/k}\otimes f^*\omega_{X/k}^{-1}))=
(-1)^r\chi(Y/k)+D(f)\cdot H
\]
in $\GW(k)$. 
\end{theorem}

\begin{proof} To compute the degree, we may assume that $k$ is algebraically closed. Let $L=f^*\omega_{X/k}^{-1}$. Since $\omega_{X/k}$ has degree $2g_X-2$, and $k$ is algebraically closed, there is a half-canonical line bundle $M$ on $X$.   Since $X$ is a curve, $c_1(L)^2=0$, and we have 
\[
c_r(\Omega_{Y/k}\otimes f^*\omega_{X/k}^{-1}))- c_r(\Omega_{Y/k})=
c_1(L)\cup c_{r-1}(\Omega_{Y/k}) 
\]
so 
\begin{align*}
D(f)&=(1/2)\deg_k [c_1(L)\cup(c_{r-1}(\Omega_{Y/k})]\\
&=\deg_k [c_1(f^*M^{-1})\cup (c_{r-1}(\Omega_{Y/k})]
\end{align*}
so $D(f)$ is an integer. 

If $X$ admits a half-canonical line bundle $M$ (over the original field $k$), then by Theorem~\ref{thm:CWEulerClass}, Theorem~\ref{thm:Dual} and Proposition~\ref{prop:Fibering}, we have
\begin{align*}
\pi_{Y*}(&e^\CW(\Omega_{Y/k}\otimes f^*\omega_{X/k}^{-1}))\\&=
\pi_{Y*}[e^\CW(\Omega_{Y/k})+\bar{h}(c_1(f^*M^{-1})\cup c_{r-1}(\Omega_{Y/k}))]\\
&=\pi_{Y*}(e^\CW(\Omega_{Y/k}))+D(f)\cdot H\\
&=(-1)^r\chi(Y/k)+D(f)\cdot H.
\end{align*}

If on the other hand $Y$ has even dimension $2m$, then we have
\[
\pi_{Y*}(e^\sW(\Omega_{Y/k}\otimes f^*\omega_{X/k}^{-1}))=\pi_{Y*}(e^\sW(\Omega_{Y/k}))
\]
in $W(k)$, so $\pi_{Y*}(e^\CW(\Omega_{Y/k}\otimes f^*\omega_{X/k}^{-1}))-\pi_{Y*}(e^\CW(\Omega_{Y/k}))\in \GW(k)$ goes to zero under the canonical surjection $\GW(k)\to W(k)$. Thus 
\[
\pi_{Y*}(e^\CW(\Omega_{Y/k}\otimes f^*\omega_{X/k}^{-1}))-\pi_{Y*}(e^\CW(\Omega_{Y/k})=\ell\cdot H
\]
for some integer $\ell$. Applying the rank homomorphism gives
\[
\deg_k(c_r(\Omega_{Y/k}\otimes f^*\omega_{X/k}^{-1}))-\deg_k(c_r(\Omega_{Y/k}))=2\cdot\ell
\]
so $\ell=D(f)$. 

\end{proof}

We now turn to the discussion of the local invariants. As usual, a {\em critical point} of $f$ is a point $y\in Y$ with $df(y)=0$, a critical value of $f$ is a point $x=f(y)$ of $X$ with $y$ a critical point. We assume that $f$ has only finitely many critical points and let $\crit(f)$ denote the set of critical points.

In case $Y$ has odd dimension, we assume we have a half-canonical line bundle $M$  and an isomorphism $\rho:\omega_{X/k}\to M^{\otimes 2}$. Thus, we have a comparison isomorphism $\psi:\sK^{MW}_r(\det^{-1}(\Omega_{Y/k}\otimes f^*\omega^{-1}_{X/k}))\to
\sK^{MW}_r(\omega_{Y/k})$,
\[
\psi:=\begin{cases} \psi_{\omega_{Y/k}\otimes f^*\omega^{\otimes r/2}_{X/k}}&\text{ if $r$ is even}\\
\psi_{\omega_{Y/k}\otimes f^*M^{\otimes r}}\circ \rho_*^{\otimes -r}&\text{ if $r$ is odd, $M$ a theta characteristic.}
\end{cases}
\]

Let $y$ be a critical point of $f$, giving the Euler class with support
\[
e_y(df):=e_y(\Omega_{Y/k}\otimes f^*\omega_{X/k}^{-1}; df)\in H^r_y(Y, \sK^{MW}_r(\omega^{-1}_{Y/k}\otimes f^*\omega_{X/k}^{\otimes r})).
\]
Applying the comparison isomorphism  $\psi$ and the  purity isomorphism 
\[
\sK^{MW}_0(y)\cong H^r_y(Y, \sK^{MW}_r(\omega_{Y/k}))
\]
we will also consider $e_y(df)$ as an element of $\sK^{MW}_0(y)$. We have the pushforward
\[
i_{y*}:\sK^{MW}_0(y)\to H^r(Y, \sK^{MW}_r(\omega_{Y/k})).
\]

\begin{remark} If $r$ is odd, we have noted that $\psi(e^\CW(\Omega_{Y/k}\otimes f^*\omega_{X/k}^{-1}))$ does not depend on the choice of $M$ or $\rho$. However, this is not the case for the local Euler classes $e_y(df)$. Nonetheless, we will omit this dependence from the statements below, which remain valid for each such choice.
\end{remark}

\begin{corollary}\label{cor:Fibering1}   Let $f$ be a projective  morphism $f:Y\to X$, with $Y$ a smooth projective integral $k$-scheme of dimension $r$ over $k$, and $X$ a smooth projective curve over $k$.  Suppose  $f$ has only finitely many critical points. In addition, suppose that  $X$ admits a half-canonical line bundle in case $r$ is odd.Then
\[
(-1)^r\cdot \chi(Y/k)=\sum_{y\in \crit(f)}\pi_{Y*}i_{y*}e_y(df)-D(f)\cdot H
\]
in $\GW(k)$. 
\end{corollary}

\begin{proof} Forgetting supports sends the Euler class with supports 
\[
e_{df=0}(\Omega_{Y/k}\otimes f^*\omega_{X/k}^{-1}; df)\in H^r_{df=0}(Y, \sK^{MW}_r(\omega^{-1}_{Y/k}\otimes f^*\omega_{X/k}^{\otimes r}))
\]
to $e^\CW(\Omega_{Y/k}\otimes f^*\omega_{X/k}^{-1}; df)$ in $H^r(Y, \sK^{MW}_r(\omega^{-1}_{Y/k}\otimes f^*\omega_{X/k}^{\otimes r}))$. Applying the comparison isomorphism $\psi$ to 
$e^\CW(\Omega_{Y/k}\otimes f^*\omega_{X/k}^{-1}; df)$ and $\psi$ and the inverse of the Thom isomorphism to the local index $e_y(\Omega_{Y/k}\otimes f^*\omega_{X/k}^{-1}; df)$ as described in the paragraphs above, we have
\[
\psi(e^\CW(\Omega_{Y/k}\otimes f^*\omega_{X/k}^{-1}))= \sum_{y\in \crit(f)}i_{y*}e_y(df).
\]
Applying $\pi_{Y*}$  and using  Proposition~\ref{prop:Fibering}, this gives 
\[
(-1)^r\cdot\chi(Y/k)=\sum_{y\in \crit(f)}\pi_{Y*}i_{y*}e_y(df)-D(f)\cdot H
\]
in $\GW(k)$.
\end{proof}

\begin{remark} The rank of the term $D(f)\cdot H$ in Corollary~\ref{cor:Fibering1}  is $(-1)^{r-1}\cdot \chi^{top}(f^{-1}(x))\cdot\chi^{top}(X)$  for $x\in X$ a general geometric point. 
\end{remark}

In case $Y$ has odd dimension and there is a half-canonical bundle $\rho:\omega_{X/k}\xrightarrow{\sim}M^{\otimes 2}$, there is a normalization that picks out good local parameters at critical values of $f:Y\to X$. As explained in Remark~\ref{rems:NonperfectFields}, we may (and will) assume that $k$ is perfect.

Let $y\in Y$ be a critical point of $f$ and let $x=f(y)$ be the corresponding critical value. A parameter $t_x\in \mathfrak{m}_x$ is {\em normalized} if there is a generating section $\lambda_{M,x}$ of $M$ in a neighborhood of $x\in X$ such that
\[
\rho^{\otimes -1}(\del/\del t_x)=\lambda_{M,x}^{-2}\otimes k(x)
\]
via the canonical isomorphism $\omega_{X/k}^{-1}\otimes k(x)\cong (\mathfrak{m}_x/\mathfrak{m}_x^2)^\vee$. 

\begin{corollary}\label{cor:Fibering2}   Let $f$ be a   projective  morphism $f:Y\to X$, with $Y$ a smooth projective integral $k$-scheme of dimension $r$ over $k$, and $X$ a smooth projective curve over $k$.  If $r$ is odd, we assume $X$ admits a half-canonical line bundle $\rho:\omega_{X/k}\to M^{\otimes 2}$.  Suppose  $f$ has only finitely many critical points.

For each $y\in  \crit(f)$, we choose a system of parameters $t_1^y,\ldots, t^y_r\in \mathfrak{m}_y$ and a parameter $t_x\in \mathfrak{m}_x$, $x=f(y)$; if $r$ is odd, we assume that $t_x$ is  normalized. Write
\[
d(f^*(t_x))=\sum_{i=1}^r s_i^y\cdot dt_i^y;\quad s_i^y\in \sO_{Y,y}.
\]
We have the class of the Scheja-Storch form $[B_{s^y_*, t^y_*}]\in \GW(k(y))$. Then
\[
(-1)^r\chi(Y/k)= \sum_{y\in \crit(f)}\Tr_{k(y)/k}([B_{s^y_*, t^y_*}])-D(f)\cdot H 
\]
in $\GW(k)$. 
\end{corollary}

\begin{remarks} 1. In our earlier version of this paper \cite[Corollary 12.4]{LevQuad}, we had an assumption on the local behavior of $df$ ({\em diagonalizability}) that allowed an explicit computation of the local index 
without having to use the Scheja-Storch form; $df$ is always diagonalizable if $Y$ is a smooth curve. We also assumed in {\em loc. cit.} that the residue field extension $k(y)/k$ was separable. In another result,  \cite[Theorem 12.7]{LevQuad}, we made the expression in  \cite[Corollary 12.4]{LevQuad} even more explicit in the case of a tamely ramified map of curves. \\[2pt]
2. In \cite{BKW} the authors of that paper use a Scheja-Storch form to give a quadratic Riemann-Hurwitz formula for a separable map of smooth projection curves over a field $k$ that is also valid in the case of inseparable residue field extension and for wild ramification; their formula agrees with the one of  \cite[Corollary 12.4]{LevQuad}  in the case of a separable residue field extension. They raise the question (Remark 1.2(2)) of whether their explicit expression agrees with the abstract pushforward of the local index; this has been settled affirmatively in \cite{BachmannWickelgren}. Their formula also agrees with the one given above in the case of a perfect base-field $k$.
\end{remarks}

\begin{proof}[Proof of Corollary~\ref{cor:Fibering2}] By Corollary~\ref{cor:SSLocalEulerClass}, the local class $e_y(\Omega_{Y/k}\otimes f^*\omega_{X/k}^{-1}, df)$ is given by
\[
[B_{s^y_*, t^y_*}]\otimes \del/\del t_1^y\cdot df^*(t_x)\wedge\ldots \wedge \del/\del t_r^y\cdot df^*(t_x)\otimes \del/\del t_1^y\wedge\ldots\wedge\del/\del t_r^y
\]
in $\GW(k(y), \det^{-1}(\Omega_{Y/k}\otimes f^*\omega_{X/k}^{-1})\otimes \det^{-1}\mathfrak{m}_y/\mathfrak{m}_y^2)$.
Under the comparison isomorphism $\psi$, this gets sent to $[B_{s^y_*, t^y_*}] \in
\GW(k(y))$.

Applying the canonical isomorphism $i_{y*}:\GW(k(y))\xrightarrow{\sim} H^r_y(Y, \sK^{MW}_r(\omega_{Y/k}))$ and the forget the supports map, we have 
\[
\psi(e^\CW(\Omega_{Y/k}\otimes f^*\omega_{X/k}^{-1}))=
\sum_{y\in \crit(f)} i_{y*}\psi(e_y(\Omega_{Y/k}\otimes f^*\omega_{X/k}^{-1}))=
\sum_{y\in \crit(f)}[B_{s^y_*, t^y_*}]
\]
and thus by Corollary~\ref{cor:Fibering1}, we have
\[
(-1)^r\cdot \chi(Y/k)=\sum_{y\in \crit(f)} \pi_{y*}([B_{s^y_*, t^y_*}])-D(f)\cdot H
\] 
where  $\pi_y:\Spec k(y)\to \Spec k$ is the structure map. Since $k$ is perfect, $k(y)$ is a separable extension of $k$, so $\pi_{y*}$ is the trace map $\Tr_{k(y)/k}:\GW(k(y))\to \GW(k)$.  This completes the proof.
\end{proof}

Let $f:Y\to X$ be as before a   morphism of a smooth integral projective $k$-scheme 
$Y$ of dimension $r$ to a smooth projective curve $X$. 

Let $y\in Y$ be a critical point of $f$. Let  $t_1,\ldots, t_r$ be a system of parameters at $y$, and let $t_x\in\mathfrak{m}_x$ be a parameter. Since $y$ is a critical point of $f$,  $f^*(t_x)$ is in $\mathfrak{m}_y^2$ and thus
\[
f^*(t_x)=  \sum_{i\le j}a_{ij}t_it_j 
\]
for elements $a_{ij}\in \sO_{Y,y}$, uniquely determined modulo $\mathfrak{m}_y$. Let $\bar{a}_{ij}\in k(y)$ be the residue of $a_{ij}$ modulo $\mathfrak{m}_y$. Let
\[
h_{ij}=\begin{cases} \bar{a}_{ij}&\text{ if }i<j\\2\bar{a}_{ii}&\text{ if }i=j\\
\bar{a}_{ji}&\text{ if }i>j
\end{cases}
\]
The symmetric matrix 
\[
H(f)_y:=\begin{pmatrix}h_{ij}\end{pmatrix}
\]
is the  Hessian matrix of $f$ with respect to the chosen system of parameters.  The point $y$ is called a  non-degenerate critical point of $f$ if $H(f)_y$ is a non-singular matrix and $k(y)/k$ is a separable extension.

Let $y$ be a non-degenerate critical point of $f$, let $x=f(y)$. Choose a system of 
parameters $t_1,\ldots, t_r$ at $y$ and a parameter $t_x$ at $x$, and  let $H(f)_y=\begin{pmatrix}h_{ij}\end{pmatrix}$. The section $df$  satisfies
\[
df\equiv \sum_{i,j}h_{ij}t_idt_j \otimes \del/\del t_x\mod \mathfrak{m}_y^2.
\]
By Example~\ref{ex:NonDegn}, we have
\[
e_y(\Omega_{Y/k}\otimes f^*\omega_{X/k}^{-1}; df)=\<\det H(f)_y\>\otimes\del/\del t_x.
\]

\begin{corollary}\label{cor:Fibering3}  Let $f$ be a surjective projective  morphism $f:Y\to X$, with $Y$ a smooth projective integral $k$-scheme and $X$ a smooth projective curve over $k$. Suppose that $f$ has only non-degenerate critical points. For each $y\in \crit(f)$, let $x=f(y)$, choose a  parameter $t_x\in \mathfrak{m}_x$,  and let $\<\det H(f)_y\>\in \GW(y)$ be the corresponding 1-dimensional quadratic form. In case $Y$ has odd dimension, we assume that $X$ admits a half-canonical line bundle and that $t_x$ is normalized. Then 
\[
(-1)^r\cdot \chi(Y/k)= \sum_{y\in \crit(f)} \Tr_{k(y)/k}(\<\det H(f)_y\>)-D(f)\cdot H
\]
in $\GW(k)$.
\end{corollary}

\begin{proof} This follows directly from Corollary~\ref{cor:Fibering2} and the preceding discussion.
\end{proof}

\begin{remark}\label{rem:Trace}  Suppose $X=\P^1_k$. Let $t=X_1/X_0$ be the standard parameter on $\A^1=\P^1\setminus\{(0:1)\}$.  We have a unique isomorphism
\[
\omega_{\P^1/k}^{-1}\cong O_{\P^1}(2)
\]
sending $\del/\del t$ to $X_0^2$ and $\del/\del t^{-1}$ to $-X_1^2$. We use the section $X_0$ of  $M:=O_{\P^1}(1)$ as our $\lambda_M$. For a closed point $x\in \A^1_k$, let $g_x\in k[t]$ be the monic irreducible polynomial  for $x$ over $k$. Then
\[
t_x:= (dg_x/dt)^{-1}g_x
\]
is a normalized local parameter at $x$.
\end{remark}

We apply these results to the case of a map of smooth projective curves $f:Y\to X$. For $p:C\to \Spec k$ a smooth projective curve over $k$, define
\[
g_{C/k}:=\dim_k H^0(C,\omega_{C/k}).
\]
This is the usual genus of $C$ if $C$ is geometrically integral over $k$. 

\begin{theorem}[Riemann-Hurwitz formula for curves]\label{thm:RH} Let $f:Y\to X$ be a separable surjective morphism of smooth integral projective curves over $k$. Suppose that $X$ admits a half-canonical bundle $M$ and fix an isomorphism $\rho: \omega_{X/k}\to M^{\otimes 2}$. For $y\in \crit(f)$, choose a parameter $t_y\in \mathfrak{m}_y$ and a normalized parameter $t_x\in \mathfrak{m}_x$, $x=f(y)$. Write
\[
f^*(t_x)=u_yt_y^{n_y}
\]
with $u_y\in \sO_{Y,y}^\times$ and let $\bar{u}_y\in k(y)^\times$ be the image of $u_y$.  Suppose that $k(y)$ is separable over $k$ and $n_y$ is prime to $\Char k$ for all $y\in\crit(f)$. Then
\[
\sum_{y\in\crit(f)}\Tr_{k(y)/k}(\<n_y\bar{u}_y\>(n_y-1)_\epsilon)=
(g_{Y/k}-1-\deg f\cdot (g_{X/k}-1))\cdot H
\]
in $\GW(k)$.
\end{theorem}

\begin{proof} Since $f$ is separable and surjective, $\crit(f)$ is a finite set. Near $y\in \crit(f)$ we  have  
\[
df=n_yu_yt_y^{n_y-1}\otimes dt_y\otimes \del/\del t_x.
\]
By Corollary~\ref{cor:Fibering2}, we have
\[
\sum_{y\in\crit(f)}\Tr_{k(y)/k}(\<n_y\bar{u}_y\>(n_y-1)_\epsilon)=-\chi(Y)+D(f)\cdot H.
\]
Since $Y$ has odd dimension over $k$, $\chi(Y)=A\cdot H$ for some integer $A$, by Corollary~\ref{cor:OddEulerChar}. Thus
\[
\sum_{y\in\crit(f)}\Tr_{k(y)/k}(\<n_y\bar{u}_y\>(n_y-1)_\epsilon)=B\cdot H
\]
for some integer $B$. Applying the rank homomorphism gives
\[
\sum_{y\in \crit(f)} [k(y):k](n_y-1) =2B
\]
so the classical Riemann-Hurwitz formula tells us that 
\[
B=(g_{Y/k}-1-\deg f\cdot(g_{X/k}-1)).
\]
\end{proof}

\begin{remark} With notation as in Theorem~\ref{thm:RH}, suppose $y\in Y$ is a ramified point. Then 
\[
\<n_y\bar{u}_y\>\cdot (n_y-1)_\epsilon=\begin{cases}\frac{1}{2}(n_y-1)\cdot H&\text{ if $n_y$ is odd,}\\
\<n_y\bar{u}_y\>+\frac{1}{2}(n_y-2)\cdot H&\text{ if $n_y$ is even.}
\end{cases}
\]
We can rewrite the GW-Riemann-Hurwitz formula as
\begin{multline*}
\sum_{y\in\crit(f), n_y\text{ even }}\Tr_{k(y)/k}(\<n_y\bar{u}_y\>)\\=
\left(g_{Y/k}-1-\deg f\cdot (g_{X/k}-1)-\sum_{y\in\crit(f)}[k(y):k]\cdot\left[\frac{n_y-1}{2}\right]\right)\cdot H.
\end{multline*}
In other words, the ramification points with $n_y$ even impose a global relation in $\GW(k)$ beyond the numerical identity 
\[
2g_{Y/k}-2=\deg f\cdot (2g_{X/k}-2)+\sum_{y\in\crit(f)}[k(y):k]\cdot(n_y-1)
\]
given by the classical Riemann-Hurwitz formula.  One recovers the classical Riemann-Hurwitz formula by applying the rank map to the GW version.
\end{remark}

\begin{remark} Theorem~\ref{thm:RH} covers the case of tame ramification; in case of wild ramification over a perfect base-field, one can use Corollary~\ref{cor:Fibering2} and if the base-field is not perfect, one simply extends to the perfect closure.
\end{remark}

\begin{ex} We take $k=\R$. Suppose we have a surjective map $f:Y\to \P^1_k$ with $Y$ a smooth projective curve of genus $g$. Suppose in addition that $f$ is simply ramified, that is, $n_y\le 2$ for all $y\in Y$. Take a closed point $y\in Y$ with $n_y=2$. If $k(y)=\C$, then the trace form $(e_i,e_j)\mapsto\Tr_{\C/\R}(e_ie_ju)$ is hyperbolic for all $u\in \C^\times$. If $k(y)=\R$, then $\pi_{Y*}(e_y(df))$ is just the quadratic form $\<2\bar{u}_y\>$, using $t_x=t-f(y)$ as the normalized local parameter at $x=f(y)$ and writing $f^*(t_x)=u_yt_y^2$. Thus, the extra information in the GW-Riemann-Hurwitz formula is just that there are the same number of real ramified points $y$ of $f$  with $\bar{u}_y>0$ as there are real ramified points $y$ with $\bar{u}_y<0$. This is also obvious by looking at the real points of $Y$, which is a disjoint union of circles, and using elementary Morse theory. 
\end{ex}

\begin{remark} Going back to the guiding example of smooth projective varieties over $\R$,  the formula of Corollary~\ref{cor:Fibering3} for a map $f:Y\to \P^1$ may be viewed as combining the classical enumerative formulas for counting degeneracies for schemes over $\C$ with using Morse theory to compute the Euler characteristic of a compact oriented manifold $M$ by counting  the number of critical points of a map $f:M\to S^1$ having only non-degenerate critical points, where we count a critical point with the sign of the Hessian determinant. In fact, as the signature of a hyperbolic form in $\GW(\R)$ is zero, and since the trace map $\Tr:\GW(\C)\to \GW(\R)$ sends $q$ to $\text{rank}(q)\cdot h$,  taking the signature of the formula in Corollary~\ref{cor:Fibering3} expresses the Euler characteristic of $Y(\R)^\an$ as  the  sum of the signs of the Hessian determinant at each of the real critical points of $f$. 
\end{remark}

\begin{ex}[Fibering by curves] We consider the case of a pencil of curves in $\P^2$. Let $C, C'$ be smooth curves of degree $d$ in $\P^2$, intersecting transversely. Let $Z=C\cap C'$ and let $F, F'\in k[X_0, X_1, X_2]_d$ be respective defining equations for $C$ and $C'$. Let $\mu_Z:Y\to \P^2$ be the blow-up of $\P^2$ along $Z$. The rational map $f:\P^2\xymatrix{\ar@{-->}[r]&}\P^1$, 
\[
f(x_0:x_1:x_2)=(F(x_0:x_1:x_2):F'(x_0:x_1:x_2)),
\]
defines a morphism $f:Y\to \P^1$ with $f^{-1}(a:b)$ the curve $bF-aF'=0$. We suppose  that for $x\in \P^1$, $f^{-1}(x)$ smooth except for $x\in \{x_1,\ldots, x_s\}$, a set of closed points of $ \P^1$. For simplicity we assume in addition that for each $i$, $f^{-1}(x_i)$ is reduced and has a single ordinary double point $y_i$ as singular point (we do not assume that $k(y_i)=k(x_i)$).  

 Note that our $f$ has only non-degenerate critical points and for $x\in \P^1$ general, the smooth curve $f^{-1}(x)$ satisfies 
 $\deg_k c_1(\omega_{f^{-1}(x)})=d(d-3)$.  This gives us $D(f)=d(d-3)$; applying Corollary~\ref{cor:Fibering3} and Proposition~\ref{prop:EulerCharProperties}, we have
\[
\sum_{i=1}^s\Tr_{k(y_i)/k}\<\det H(f)_{y_i}\>=\<1\>+(d^2-3d+1)\cdot H+\<-1\>\cdot\chi(Z/k).
\]
Since $\chi(Z/k)$ is the trace form of the finite separable extension  $k\to k(Z)$,  our Riemann-Hurwitz formula gives a relation between this trace form and the ``ramification index" $\sum_{i=1}^s\Tr_{k(y_i)/k}\<\det H(f)_{y_i}\>$. Taking the rank recovers the numerical relation given by the classical Riemann-Hurwitz formula, namely
\[
\sum_{i=1}^s[k(y_i):k]= 3d^2-6d+3.
\]

\end{ex}

\section{Generalized Fermat hypersurfaces} We use Corollary~\ref{cor:Fibering2} to compute the Euler characteristic of   a generalized Fermat hypersurface $X$ in $\P^{n+1}$, that is, one with defining polynomial  $\sum_{i=0}^{n+1}a_iX_i^m$. 

Fix an integer $m\ge1$ and a base-field $k$ of characteristic prime to $2m$; we may assume that $k$ is infinite by replacing $k$ with infinite extension of $\ell$-power degree for some odd prime $\ell$. Let $X=X(a_0,\ldots, a_{n+1};m)\subset \P^{n+1}$  be the hypersurface with defining equation $\sum_{i=0}^{n+1}a_i X_i^m=0$, $a_i\in k^\times$. Let $\pi:\tilde{X}\to X$ be the blow-up along the closed subscheme $Z$ defined by $X_n=X_{n+1}=0$; note that $Z=X(a_0,\ldots, a_{n-1};m)$.  We  apply Proposition~\ref{prop:EulerCharProperties} to give
 \begin{equation}\label{eqn:BlowUpHypersurf}
 \chi(\tilde{X})=\chi(X)+\<-1\>\chi(Z).
 \end{equation}
 We have the morphism
 \[
 f:\tilde{X}\to \P^1
 \]
 induced by the rational map $X\to \P^1$, $(x_0:\ldots:x_{n+1})\mapsto 
 (x_n:x_{n+1})$. The map $f$ has non-degenerate critical points $(0:\ldots:0:x_n:x_{n+1})$ satisfying $a_nx_n^m+a_{n+1}x_{n+1}^m=0$ (the critical points do not lie over $Z$, so we may describe the critical points of $f$ as points of $X$). Since $a_na_{n+1}\neq0$, the critical points of $f$ lie in the affine open subset $X_{n+1}\neq0$, so we may use affine coordinates $x_i:=X_i/X_{n+1}$. 
 
On the affine hypersurface $X^0\subset \A^{n+1}$ defined by $\sum_{i=0}^na_ix_i^m+a_{n+1}=0$, the map $f$ is given by
\[
f(x_0,\ldots, x_n)=x_n
\]
and has critical subscheme $X_{crit}\subset X^0$ defined as a subscheme of $\A^{n+1}$ by   $x_i=0$, $i=0,\ldots, n-1$, $x_n^m+a_{n+1}/a_n=0$.

We now apply the Riemann-Hurwitz formula to the projection $f:\tilde{X}\to \P^1$.  Let $y=X_{crit}$, a 0-dimensional reduced closed subscheme of $X$, and use the system of parameters $(x_0,\ldots, x_{n-1})$, generating the maximal ideal in $\sO_{X,y'}$ for each closed  point $y'$ of $y$.  Similarly, we let $g(T)=T^m+a_{n+1}/a_n$, let $x\subset \A^1_k$ be the subscheme $\Spec k[T]/g(T)$ and use the parameter $t_x:=(1/g'(T))g(T)$ in $\sO_{\A^1, x}$. 

 As $df$ is given by the expression
\begin{equation}\tag{$*$}
df=\sum_{i=0}^{n-1} (-1/a_nx_n^{m-1})\cdot a_ix_i^{m-1}dx_i \otimes  \del/\del t_x
\end{equation}
and $\det^{-1}\Omega_{X/k}\otimes \det^{-1}\mathfrak{m}_x/\mathfrak{m}_x^2\cong
(\det(\mathfrak{m}_x/\mathfrak{m}_x^2)^\vee)^{\otimes 2}$, 
we have the local index  $e_x(df)$ living in $\GW(k(x); (\det(\mathfrak{m}_x/\mathfrak{m}_x^2)^\vee)^{\otimes 2}\otimes f^*\omega_{\P^1})$:
\[
e_x(df):=\<(-1/a_nx_n^{m-1})^n((m-1)_\epsilon)^n(\prod_{i=0}^{n-1}a_i)\>\otimes
(\wedge\del/\del x_*)^{\otimes 2}\otimes d t_x
\]
If $n$ is odd, we know that $\chi(X/k)$ is hyperbolic, so we may assume that $n=2r$ is even, in which case this expression reduces to 
\[
e_x(df)=\begin{cases}\frac{1}{2}(m-1)^n\cdot H\<(\prod_{i=0}^{n-1}a_i)\>\otimes(\wedge\del/\del x_*)^{\otimes 2}\otimes d t_x&\text{ if }m\text{ is odd,}\\
[\frac{1}{2}((m-1)^n-1)\cdot H+\<1\>]\<(\prod_{i=0}^{n-1}a_i)\>\otimes(\wedge\del/\del x_*)^{\otimes 2}\otimes d t_x\kern-80pt\\
&\text{ if }m\text{ is even.}
\end{cases}
\]
which reduces further to 
\[
e_x(df)=\begin{cases}
\frac{1}{2}(m-1)^n\cdot H\otimes(\wedge\del/\del x_*)^{\otimes 2}\otimes d t_x&\text{ if }m\text{ is odd,}\\
[(\frac{1}{2}(m-1)^n-1)H+\<\prod_{i=0}^{n-1}a_i\>]\otimes(\wedge\del/\del x_*)^{\otimes 2}\otimes d t_x\kern-25pt\\
&\text{ if }m\text{ is even.}
\end{cases}
\]

Following Remark~\ref{rem:Trace}, our choice of parameter $t_x$ is normalized, so after applying the appropriate comparison isomorphism to put the local index $e_x(df)$ in $\GW(k(x))$, as in the proof of Corollary~\ref{cor:Fibering1},  we have the identity in $H^n(X, \sK^{MW}_n(\omega_{X/k}))$
\[
i_{x*}(e_x(df))=\begin{cases}
i_{x*}[\frac{1}{2}(m-1)^n\cdot H]
&\text{ if }m\text{ is odd,}\\
i_{x*}[\frac{1}{2}((m-1)^n-1)H+\<\prod_{i=0}^{n-1}a_i\>)]&\text{ if }m\text{ is even,}
\end{cases}
\]
where 
\[
i_{x*}:\GW(k(x))=H^0(x, \sK^{MW}_0)\to H^n(X, \sK^{MW}_n(\omega_{X/k}))
\]
is the push-forward. 

The extension $k(x)/ k$ is a finite separable extension, so we have the pushforward map $p_*:\GW(x)\to \GW(k)$ given by  the trace form. Since
\[
Tr_{k(x)/k}(\<1\>)=\begin{cases}\<m\>+\frac{m-1}{2}\cdot H&\text{ for $m$ odd}\\
\<m\>+\<-ma_na_{n+1}\>+\frac{m-2}{2}\cdot H&\text{ for $m$ even,}
\end{cases}
\]
we get 
\begin{equation}\label{eqn:LocalContHypersuf}
p_{X*}i_{x*}(e_x(df))= 
\begin{cases}
\frac{1}{2}(m-1)^nm\cdot H& \text{ for $m$ odd,}\\
\<m\prod_{i=0}^{n-1}a_i\>+\<-m\prod_{i=0}^{n+1}a_i\>+(\frac{1}{2}(m-1)^nm-1)\cdot H\kern-55pt
\\& \text{ for $m$ even.}
\end{cases}
\end{equation}

\begin{theorem}\label{thm:EulerCharDiagonalHyp} Let $X=X(a_0,\ldots, a_{n+1}; m)\subset \P^{n+1}_k$ be a  generalized Fermat hypersurface of degree $m\ge1$, $a_0,\ldots, a_{n+1}\in k^\times$. Suppose that $\Char(k)\hbox{$\not|$\,} 2m$. Let $\delta(X):=\prod_{i=0}^{n+1}a_i$ and define $A_{n,m}\in\Q$ by
\[
A_{n,m}=\begin{cases}
\frac{1}{2}\deg(c_n(T_X))&\text{ for $n$ odd,}\\
\frac{1}{2}\left(\deg(c_n(T_X))-1\right)&\text{ for $n$ even and $m$ odd,}\\
\frac{1}{2}\left(\deg(c_n(T_X))-2\right)&\text{ for $n$ and $m$ even.}
\end{cases}
\]
Then $A_{n,m}$ is an integer, depending only on $n$ and $m$. Moreover, 
\[
\chi(X/k)=\begin{cases} 
A_{n,m}\cdot H&\text{ for $n$ odd,}\\
A_{n,m}\cdot H+\<m\>&\text{ for $n$ even and $m$ odd,}\\
A_{n,m}\cdot H+\<m\>+\<-m\delta(X)\>&\text{ for $n$ and $m$ even.}
\end{cases}
\]
\end{theorem}

\begin{proof} It is clear that the rational number $A_{n,m}$ depends only on $n$ and $m$. For $n$ odd,  the identity $\chi(X/k)=B\cdot h$ for some integer $B$ follows from Corollary~\ref{cor:OddEulerChar}. Since  $\pi_n(e^\CW(T_X))=c_n(T_X)$ in $H^n(X, \sK^M_n)$, we see that 
\[
2B=\rank(\chi(X/k))=\deg(c_n(T_X)), 
\]
so   $A_{n,m}=B$.

We now assume $n$ is even and we prove the identity by induction on $n$. We first consider the case of even $m$. For $n=0$, 
\[
X(a_0, a_1;m)=\Spec k[T]/(T^m+a_{1}/a_0)
\]
and $\chi(X/k)$ is given by the trace form, 
\[
\chi(X/k)=\Tr_{X/k}(\<1\>)=\frac{m-2}{2}\cdot H+\<m\>+\<-m\delta(X)\>. 
\]
As $c_0(T_X)$ has degree $m$, the result is proven in this case. In general, assume the result for $n-2$, and let $Z=X(a_0,\ldots, a_{n-1})$. Then combining \eqref{eqn:BlowUpHypersurf}, Corollary~\ref{cor:Fibering2}  and our computation of the local contributions \eqref{eqn:LocalContHypersuf}, we have
\[
\chi(X/k)=\<m\delta(Z)\>+\<-m\delta(X)\>-\<-1\>\chi(Z/k)+A\cdot H
\]
for some integer $A$. Using our induction hypothesis, this reduces to 
\[
\chi(X/k)=-\<-m\>+\<-m\delta(X)\>+B\cdot H
\]
for some integer $B$. But 
\[
H-\<-m\>=\<m\>+\<-m\>-\<-m\>=\<m\>
\]
so we have 
\[
\chi(X/k)=\<m\>+\<-m\delta(X)\>+(B-1)\cdot H.
\]
In particular, this shows that $\deg(c_n(T_X))=2B$, which shows as above that $A_{n,m}$ is an integer and gives
\[
\chi(X/k)=-\<m\>+\<-m\delta(X)\>+A_{n,m}\cdot H.
\]

For odd $m$, the proof is essentially the same, starting with
\[
\chi(X/k)=\Tr_{X/k}(\<1\>)=\frac{m-1}{2}\cdot H+\<m\>. 
\]
for $X=X(a_0, a_1;m)$; we leave the details to the reader.
\end{proof}

Recall that for a quadratic form $q$  of even rank $n=2m$ over a field $k$ of characteristic different from 2, the  discriminant of $q$ is the element of $k^\times/k^{\times2}$ given by $\det(q)$, where $(q)$ is the  matrix of  the symmetric bilinear form corresponding to $q$, with respect to some choice of basis for the underlying vector space of $q$. For quadrics, Theorem~\ref{thm:EulerCharDiagonalHyp} gives
\begin{corollary}\label{cor:EvenQuadric EulerChar} Let $k$ be a field with $\Char k\neq2$ and let $Q$ be a non-singular quadric hypersurface in $\P^{n+1}_k$. Suppose $Q$ has defining  form $q$, with discriminant $\delta_q$. Then
\[
\chi(Q/k)=\begin{cases}\frac{n+1}{2}\cdot H&\text{ for $n$ odd,}\\
\frac{n}{2}\cdot H+\<2\>+\<-2\delta_q\>&\text{ for $n=2m$ even.}
\end{cases}
\]
\end{corollary}
This answers a question raised by Kass and Wickelgren (private communication).

\begin{proof} If $k$ is algebraically closed, then $Q$ is cellular with $n+1$ cells in case $n$ is odd, and $n+2$ cells if $n$ is even. Thus by Proposition~\ref{prop:EulerCharProperties}
\[
\text{rank}\chi(Q/k)=\begin{cases}n+1&\text{ for $n$ odd,}\\
n+2&\text{ for $n$ even.}
\end{cases}
\]
With this, the corollary  follows from Theorem~\ref{thm:EulerCharDiagonalHyp}, since every quadratic form is diagonalizable by a linear change of coordinates, and the discriminant is invariant modulo squares.
\end{proof}

\begin{remark} Applying Theorem~\ref{thm:EulerCharDiagonalHyp} for $m=1$ gives yet another proof that
\[
\chi(\P^n/k)=\begin{cases} \frac{n+1}{2}\cdot H&\text{ for $n$ odd,}\\
\<1\>+\frac{n}{2}\cdot H&\text{ for $n$ even.}
\end{cases}
\]
\end{remark}

\begin{remark} The fact that $\chi(X(a_0,\ldots, a_{n+1};m)/k)$ depends only on $m$ and $n$ for $m$ odd should not be surprising: every generalized Fermat hypersurface $X(a_0,\ldots, a_{n+1};m)$ with $m$ odd is isomorphic to $X(1,\ldots, 1;m)\subset \P^{n+1}$ after a field extension of odd degree, and the base-extension map $\GW(k)\to \GW(F)$ for a finite field extension $F/k$ is injective if $[F:k]$ is odd.
\end{remark}


\begin{thebibliography}{999} 
	\normalsize
\bibitem{Abelson}
H.  Abelson,  {\em On the Euler characteristic of real varieties}. Michigan Math. J. 23 (1976), no. 3, 267--271 (1977).

\bibitem{Anan}
A. Ananyevskiy, {\em  The special linear version of the projective bundle theorem}. Compos. Math. 151 (2015), no. 3, 461--501.

\bibitem{Anan19}
A. Ananyevskiy, {\sl  SL-oriented cohomology theories}. Motivic homotopy theory and refined enumerative geometry, 1-19, Contemp. Math., 745, Amer. Math. Soc., Providence, RI,  2020.


\bibitem{AsokFaselEuler}
A. Asok, J. Fasel, {\em Comparing Euler classes}. Q. J. Math. 67 (2016), no. 4, 603--635.

\bibitem{Atiyah}
M.~F.~Atiyah, {\em Riemann surfaces and spin structures},  Ann. Sci. \'Ecole Norm. Sup. (4) 4 (1971) 47--62. 

\bibitem{BachmannWickelgren}
T. Bachmann, K. Wickelgren, {\em $\A^1$-Euler classes: six functors formalisms, dualities, integrality and linear subspaces of complete intersections} arXiv:2002.01848   math.KT math.AG math.AT

\bibitem{BargeMorel}
J. Barge, and F. Morel, {\em Groupe de Chow des cycles orient\'es et classe d'Euler des fibr\'es vectoriels}.  C. R. Acad. Sci. Paris S\'er. I Math. 330 (2000), no. 4, 287--290.

\bibitem{BKW}
C. Bethea, J. Kass, K. Wickelgren, {\em  Examples of wild ramification in an enriched Riemann-Hurwitz formula}. Motivic homotopy theory and refined enumerative geometry, 69-82, Contemp. Math., 745, Amer. Math. Soc., Providence, RI, 2020. 

\bibitem{CalmesFasel}
B. Calm\`es, J. Fasel, {\em Finite Chow-Witt correspondences}. arXiv:1412.2989  [math.AG]

\bibitem{CalmesHornbostel}
B. Calm\`es, J. Hornbostel, {\em Push-forwards for Witt groups of schemes}. Comment. Math. Helv. 86 (2011), no. 2, 437-468.


\bibitem{DJK}
Fr\'ed\'eric D\'eglise, Fangzhou Jin, Adeel A. Khan, {\em Fundamental classes in motivic homotopy theory}. Preprint 2018, arXiv:1805.05920 [math.AG]

\bibitem{DoldPuppe}
A. Dold, D. Puppe, {\em Duality, trace, and transfer}, pp. 81-102 in Proceedings of the International
Conference on Geometric Topology (Warsaw, 1978), edited by K. Borsuk and A. Kirkor, PWN, Warsaw. 

\bibitem{Dugger}
D. Dugger,  {\em Coherence for invertible objects and multigraded homotopy rings}.  Algebr. Geom. Topol. 14 (2014), no. 2, 1055-1106.

\bibitem{FaselCW}
J. Fasel, {\em Groupes de Chow-Witt}. M\'em. Soc. Math. Fr. (N.S.), 113 (2008).


\bibitem{FaselSri}
J. Fasel, V. Srinivas, {\em Chow-Witt groups and Grothendieck-Witt groups of regular schemes}. Advances in Mathematics 221 (2009) 302--329.

\bibitem{Feld}
N. Feld, {\em Milnor-Witt cycle modules}.  
J. Pure Appl. Algebra 224 (2020), no. 7, 106298, 44 pp.

\bibitem{Feld2}
N. Feld, {\em Morel Homotopy Modules and Milnor-Witt Cycle Modules}. Preprint 2019. arXiv:1912.12680  [math.AG].

\bibitem{GSZ}
S. Gille, S. Scully, C. Zhong, {\em Milnor-Witt $K$-groups of local rings}.  Adv. Math. 286 (2016), 729--753. 

\bibitem{HornbostelWendt}
J. Hornbostel, M. Wendt, {\em Chow-Witt rings of classifying spaces for symplectic and special linear groups}, Journal of Topology, Volume 12, Issue 3,
September 2019, 916--966.

\bibitem{HoyoisGL}
M. Hoyois,  {\em A quadratic refinement of the Grothendieck-Lefschetz-Verdier trace formula}. Algebr. Geom. Topol. 14 (2014) 3603--3658.
 
\bibitem{Hoyois6}
M. Hoyois,  {\em The six operations in equivariant motivic homotopy theory}.
Adv. Math. 305 (2017), 197-279.

\bibitem{Hu}
P. Hu, {\em On the Picard group of the stable $\A^1$-homotopy category}. Topology 44 (2005), no. 3, 609--640.

\bibitem{JardineMotSym}
J.~F.~Jardine, {\em Motivic symmetric spectra}.  Doc. Math. 5 (2000), 445--553. 

\bibitem{KassWickelgren}
J.L. Kass, K. Wickelgren, {\em The class of Eisenbud-Khimshiashvili-Levine is the local $\A^1$-Brouwer degree} Duke Math. J. 168 (2019), no. 3, 429-469.


\bibitem{KassWickelgrenLines}
J.L. Kass, K. Wickelgren, {\em An Arithmetic Count of the Lines on a Smooth Cubic Surface}, preprint 2017.  arXiv:1708.01175 [math.AG]

\bibitem{Kerz}
M. Kerz, {\em The Gersten conjecture for Milnor K-theory}. Invent. Math. 175 (2009), no. 1, 1-33.


\bibitem{LevQuad}
M. Levine {\em Toward an enumerative geometry with quadratic forms} preprint 18 Oct 2018. 
 arXiv:1703.03049v3 [math.AG] 


\bibitem{LevBeckerGottlieb}
M. Levine, {\em Motivic Euler characteristics and Witt-valued characteristic classes}. Nagoya Math. J. 236 (2019), 251--310.

\bibitem{LREulerChar}
M. Levine, A. Raksit, {\em Motivic Gau{\ss}-Bonnet formulas}. Algebra \& Number Theory Vol. 14 (2020), No. 7, 1801--1851. DOI: 10.2140/ant.2020.14.1801.

\bibitem{EllipticFlop}
M. Levine,  Y. Yang, G. Zhao, J. Riou. {\em Algebraic elliptic cohomology theory and flops I}.  Math. Ann. 375 (2019), no. 3-4, 1823-1855.



\bibitem{May}
J.~P.~May, {\em The Additivity of Traces in Triangulated Categories}, Adv. Math. 163 (2001), no. 1, pp. 34--73.


\bibitem{Milnor}
J. Milnor, {\bf Singular points of complex hypersurfaces}. Annals of Mathematics Studies, No. 61 Princeton University Press, Princeton, N.J.; University of Tokyo Press, Tokyo 1968.

\bibitem{MorelPI}
F. Morel, {\em Sur les puissances de l'id\'eal fondamental de l'anneau de Witt}.  Comment. Math. Helv. 79 (2004), no. 4, 689-703. 

\bibitem{MorelICTP}
F. Morel, {\em Introduction to $\A^1$-homotopy theory}.   Lectures given at the School on Algebraic $K$-Theory and its Applications, ICTP, Trieste. 8-19 July, 2002.                                                                     

\bibitem{MorelA1}
F. Morel, {\bf $\A^1$-algebraic topology over a field}.  Lecture Notes in Mathematics, 2052. Springer, Heidelberg, 2012.

\bibitem{MorelVoevodsky}
F. Morel, V. Voevodsky, {\em $\A^1$-homotopy theory of schemes}. 
Inst. Hautes \'Etudes Sci. Publ. Math. No. 90 (1999), 45-143 (2001). 

\bibitem{Mumford}
D. Mumford, {\em Theta characteristics of an algebraic curve},  Ann. Sci. \'Ecole Norm. Sup. (4) 4 (1971), 181--192. 

\bibitem{OVV}
D. Orlov,  A. Vishik,  V. Voevodsky, {\em 
An exact sequence for $K^M_*/2$ with applications to quadratic forms}. 
Ann. of Math. (2) 165 (2007), no. 1, 1-13. 

\bibitem{Panin}
I. Panin, {\em Oriented cohomology theories of algebraic varieties. II (After I. Panin and A. Smirnov)}. Homology Homotopy Appl. 11 (2009), no. 1, 349-405. 


\bibitem{PPR}
I. Panin,  K. Pimenov,  O. R\"ondigs,{\em On the relation of Voevodsky's algebraic cobordism to Quillen's K-theory}. Invent. Math. 175 (2009), no. 2, 435-451. 


\bibitem{PaninWalter}
I. Panin, C. Walter,  {\em On the motivic commutative ring spectrum {\bf BO}}. Algebra i Analiz 30 (2018), no. 6, 43-96; reprinted in St. Petersburg Math. J. 30 (2019), no. 6, 933-972 

\bibitem{Pauli}
S. Pauli, {\em Quadratic types and the dynamic Euler number of lines on a quintic threefold}.
preprint 2020 arXiv:2006.12089  [math.AG] 


\bibitem{Quillen}
D. Quillen, {\em Higher algebraic K-theory. I}. Algebraic --theory, I: Higher K-theories (Proc. Conf., Battelle Memorial Inst., Seattle, Wash., 1972), pp. 85-147. Lecture Notes in Math., Vol. 341, Springer, Berlin 1973.

\bibitem{Riou}
J. Riou, {\em Dualit\'e de Spanier-Whitehead en g\'eom\'trie alg\'ebrique}. 
C. R. Math. Acad. Sci. Paris 340 (2005), no. 6, 431--436. 


\bibitem{SchejaStorch}
G. Scheja,  U. Storch, {\em \"Uber Spurfunktionen bei vollst\"andigen Durchschnitten}.  J. Reine Angew. Math. 278(279) (1975), 174-190.

\bibitem{SchejaStorchRes}
G. Scheja,  U. Storch, {\em Residuen bei vollst\"andigen Durchschnitten}.  
Math. Nachr. 91 (1979), 157-170.

\bibitem{Schlichting2}
M. Schlichting, {\em Hermitian K-theory, derived equivalences and Karoubi's fundamental theorem}. J. Pure Appl. Algebra 221 (2017), no. 7, 1729-1844.

\bibitem{Schlichting}
M. Schlichting, {\em Hermitian $K$-Theory of exact catgories}. J. K-Theory 5 (2010), 105--165.

\bibitem{SchlichtingTripathi}
M. Schlichting,  G.S. Tripathi, {\em Geometric models for higher Grothendieck-Witt groups in $\A^1$-homotopy theory}. Math. Ann. 362 (2015), no. 3-4, 1143-1167. 

\bibitem{Totaro}
B. Totaro, {\em The Chow ring of a classifying space}. Algebraic K-theory (Seattle, WA, 1997), 249-281, Proc. Sympos. Pure Math., 67, Amer. Math. Soc., Providence, RI, 1999.

\bibitem{Voev}
V. Voeodsky, {\em Motivic cohomology with $\Z/2$-coefficients}. Publ. Math. Inst. Hautes \'Etudes Sci., (98):59--104, 2003.

\bibitem{VoevICM}
V. Voeodsky, {\em $\A^1$-homotopy theory}. Proceedings of the International Congress of Mathematicians, Vol. I (Berlin, 1998). Doc. Math. 1998, Extra Vol. I, 579--604.


\bibitem{Wendt}
M. Wendt, {\em Chow-Witt rings of Grassmannians} preprint (2018, rev. 2020) arXiv:1805.06142 
\end{thebibliography}
\end{document}